\documentclass[reqno]{amsart}
\usepackage{amsmath}
\usepackage{amsfonts}
\usepackage{mathrsfs}
\usepackage {amssymb, euscript}
\usepackage {amsmath}
\usepackage {amsthm}
\usepackage{graphicx,color}
\usepackage {amscd}
\usepackage{geometry}
\usepackage{hyperref}
\usepackage{color}
\usepackage{epsfig}

\usepackage{tikz}

\usepackage{bbm}
\usepackage{graphics}
\usepackage[pagewise]{lineno}

\usepackage{slashed}
\title[A Scale-Critical Trapped Surface Formation Criterion]{A Scale-Critical Trapped Surface Formation Criterion: \\A New Proof via Signature for Decay Rates}
\date{\today}

\author{Xinliang An}
\address{Department of Mathematics, National University of Singapore, Singapore 119076}
\email{matax@nus.edu.sg}

\theoremstyle{definition}
\newtheorem{lemma}{Lemma}[section]

\newtheorem{proposition}[lemma]{Proposition}
\newtheorem{theorem}[lemma]{Theorem}
\newtheorem{corollary}[lemma]{Corollary}

\newtheorem{remark}{Remark}

\numberwithin{equation}{section}

\begin{document}

\newcommand{\ub}{\underline{u}}
\newcommand{\Cb}{\underline{C}}
\newcommand{\Lb}{\underline{L}}
\newcommand{\Lh}{\hat{L}}
\newcommand{\Lbh}{\hat{\Lb}}
\newcommand{\phib}{\underline{\phi}}
\newcommand{\Phib}{\underline{\Phi}}
\newcommand{\Db}{\underline{D}}
\newcommand{\Dh}{\hat{D}}
\newcommand{\Dbh}{\hat{\Db}}
\newcommand{\omb}{\underline{\omega}}
\newcommand{\omh}{\hat{\omega}}
\newcommand{\ombh}{\hat{\omb}}
\newcommand{\Pb}{\underline{P}}
\newcommand{\chib}{\underline{\chi}}
\newcommand{\chih}{\hat{\chi}}
\newcommand{\chibh}{\hat{\chib}}

\newcommand{\alb}{\underline{\alpha}}
\newcommand{\zeb}{\underline{\zeta}}
\newcommand{\beb}{\underline{\beta}}
\newcommand{\etb}{\underline{\eta}}
\newcommand{\Mb}{\underline{M}}
\newcommand{\oth}{\hat{\otimes}}

\def\a {\alpha}
\def\b {\beta}
\def\ab {\alphab}
\def\bb {\betab}
\def\nab {\nabla}

\def\ub {\underline{u}}
\def\th {\theta}
\def\Lb {\underline{L}}
\def\Hb {\underline{H}}
\def\chib {\underline{\chi}}
\def\chih {\hat{\chi}}
\def\chibh {\hat{\underline{\chi}}}
\def\omegab {\underline{\omega}}
\def\etab {\underline{\eta}}
\def\betab {\underline{\beta}}
\def\alphab {\underline{\alpha}}
\def\Psib {\underline{\Psi}}
\def\hot{\widehat{\otimes}}
\def\Phib {\underline{\Phi}}
\def\thb {\underline{\theta}}
\def\t {\tilde}
\def\st {\tilde{s}}

\def\d {\delta}
\def\f {\frac}
\def\i {\infty}
\def\l {\bigg(}
\def\r {\bigg)}
\def\S {S_{u,\underline{u}}}
\def\o{\omega}
\def\O{\Omega}
\def\be{\begin{equation}\begin{split}}
\def\en{\end{split}\end{equation}}

\def\od{\omega^{\dagger}}
\def\ombd{\underline{\omega}^{\dagger}}
\def\K{K-\frac{1}{|u|^2}}
\def\ut{\frac{1}{|u|^2}}
\def\Kb{K-\frac{1}{(u+\underline{u})^2}}
\def\M{\mathcal}
\def\p{\psi}

\def\D{\Delta}
\def\T{\Theta}
\def\s{S_{u',\underline{u}'}}
\def\Hu{H_u^{(0,\underline{u})}}
\def\Hbu{\underline{H}_{\underline{u}}^{(u_{\infty},u)}}
\def\ee{(\eta,\underline{\eta})}

\def\at{a^{\f12}}
\def\sigmac{\check{\sigma}}
\def\p{\psi}
\def\q{\underline{\psi}}
\def\ls{\leq}
\def\de{\delta}
\def\ls{\lesssim}
\def\oo{\Omega\mbox{tr}\chib-\frac{2}{u}}
\def\om{\omega}
\def\Om{\Omega}
\renewcommand{\div}{\mbox{div }}
\newcommand{\curl}{\mbox{curl }}
\newcommand{\trchb}{\mbox{tr} \chib}
\def\trch{\mbox{tr}\chi}

\newcommand{\Ls}{{\mathcal L} \mkern-10mu /\,}
\newcommand{\eps}{{\epsilon} \mkern-8mu /\,}

\newcommand{\tr}{\mbox{tr}}

\newcommand{\xib}{\underline{\xi}}
\newcommand{\psib}{\underline{\psi}}
\newcommand{\rhob}{\underline{\rho}}
\newcommand{\thetab}{\underline{\theta}}
\newcommand{\gammab}{\underline{\gamma}}
\newcommand{\nub}{\underline{\nu}}
\newcommand{\lb}{\underline{l}}
\newcommand{\mub}{\underline{\mu}}
\newcommand{\Xib}{\underline{\Xi}}
\newcommand{\Thetab}{\underline{\Theta}}
\newcommand{\Lambdab}{\underline{\Lambda}}
\newcommand{\vphb}{\underline{\varphi}}

\newcommand{\ih}{\hat{i}}
\newcommand{\ui}{u_{\infty}}
\newcommand{\shb}{L^2_{sc}(\underline{H}_{\ub}^{(u_{\infty},u)})}
\newcommand{\sh}{L^2_{sc}(H_{u}^{(0,\ub)})}
\newcommand{\Rb}{\underline{\mathcal{R}}}
\newcommand{\tc}{\widetilde{\tr\chib}}

\newcommand{\tcL}{\widetilde{\mathscr{L}}}

\newcommand{\sRic}{Ric\mkern-19mu /\,\,\,\,}
\newcommand{\sL}{{\cal L}\mkern-10mu /}
\newcommand{\sLh}{\hat{\sL}}
\newcommand{\sg}{g\mkern-9mu /}
\newcommand{\seps}{\epsilon\mkern-8mu /}
\newcommand{\sd}{d\mkern-10mu /}
\newcommand{\sR}{R\mkern-10mu /}
\newcommand{\snab}{\nabla\mkern-13mu /}
\newcommand{\sdiv}{\mbox{div}\mkern-19mu /\,\,\,\,}
\newcommand{\scurl}{\mbox{curl}\mkern-19mu /\,\,\,\,}
\newcommand{\slap}{\mbox{$\triangle  \mkern-13mu / \,$}}
\newcommand{\sGamma}{\Gamma\mkern-10mu /}
\newcommand{\somega}{\omega\mkern-10mu /}
\newcommand{\somb}{\omb\mkern-10mu /}
\newcommand{\spi}{\pi\mkern-10mu /}
\newcommand{\sJ}{J\mkern-10mu /}
\renewcommand{\sp}{p\mkern-9mu /}
\newcommand{\su}{u\mkern-8mu /}


\maketitle
\begin{abstract}
We provide a simple, self-contained proof of a trapped surface formation theorem that sharpens the previous results both of Christodoulou and An-Luk. Our argument is based on a systematic extension of the scale-critical arguments in An-Luk, to connect Christodoulou's short-pulse method and Klainerman-Rodnianski's signature counting argument to the peeling properties previously used in small-data results such as Klainerman-Nicolo. This in particular allows us to avoid elliptic estimates and geometric renormalizations, and gives us our new technical simplifications.

\end{abstract}

\section{Introduction}

\subsection{Background}
In this paper, we study the evolution of Einstein vacuum equations 
\begin{equation}\label{1.1}
\mbox{Ric}_{\mu\nu}=0
\end{equation}
for a (3+1) dimensional Lorentzian manifold $(\mathcal{M}, g)$.

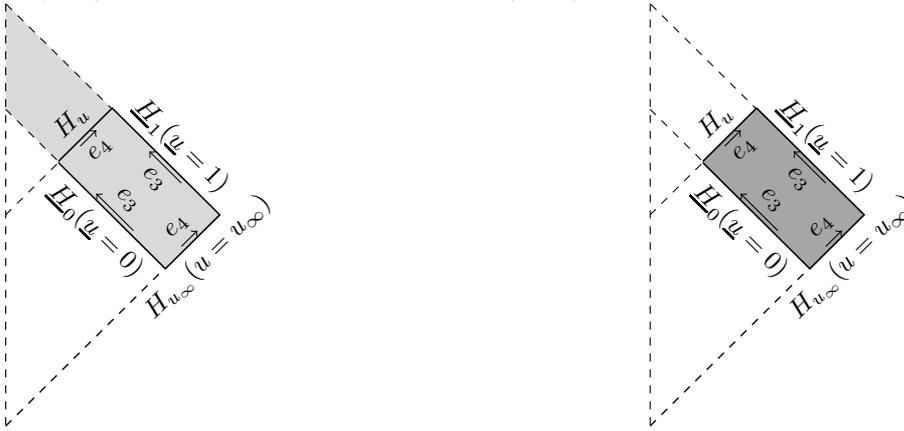
\begin{minipage}[!t]{0.5\textwidth}
\begin{tikzpicture}[scale=0.7]
\fill[black!15!white] (0,2)--(1,1)--(2,2)--(0,4)--(0,2);
\draw [white](3,-1)-- node[midway, sloped, below,black]{$H_{u_{\infty}}(u=u_{\infty})$}(4,0);
\draw [white](1,1)-- node[midway,sloped,above,black]{$H_u$}(2,2);
\draw [white](2,2)--node [midway,sloped,above,black] {$\Hb_{1}(\ub=1)$}(4,0);
\draw [white](1,1)--node [midway,sloped, below,black] {$\Hb_{0}(\ub=0)$}(3,-1);
\draw [dashed] (0, 4)--(0, -4);
\draw [dashed] (0, -4)--(4,0)--(0,4);
\draw [dashed] (0,0)--(2,2);
\draw [dashed] (0,-4)--(4,0);
\draw [dashed] (0,2)--(3,-1);
\draw [very thick] (1,1)--(3,-1)--(4,0)--(2,2)--(1,1);
\fill[black!15!white] (1,1)--(3,-1)--(4,0)--(2,2)--(1,1);
\draw [->] (3.3,-0.6)-- node[midway, sloped, above,black]{$e_4$}(3.6,-0.3);
\draw [->] (1.4,1.3)-- node[midway, sloped, below,black]{$e_4$}(1.7,1.6);
\draw [->] (3.3,0.6)-- node[midway, sloped, below,black]{$e_3$}(2.7,1.2);
\draw [->] (2.4,-0.3)-- node[midway, sloped, above,black]{$e_3$}(1.7,0.4);
\end{tikzpicture}
\end{minipage} 
\hspace{0.02\textwidth} 
\begin{minipage}[!t]{0.5\textwidth}
\begin{tikzpicture}[scale=0.7]
\draw [white](3,-1)-- node[midway, sloped, below,black]{$H_{u_{\infty}}(u=u_{\infty})$}(4,0);
\draw [white](1,1)-- node[midway,sloped,above,black]{$H_{u}$}(2,2);
\draw [white](2,2)--node [midway,sloped,above,black] {$\Hb_{1}(\ub=1)$}(4,0);
\draw [white](1,1)--node [midway,sloped, below,black] {$\Hb_{0}(\ub=0)$}(3,-1);
\draw [dashed] (0, 4)--(0, -4);
\draw [dashed] (0, -4)--(4,0)--(0,4);
\draw [dashed] (0,0)--(2,2);
\draw [dashed] (0,-4)--(4,0);
\draw [dashed] (0,2)--(3,-1);
\draw [very thick] (1,1)--(3,-1)--(4,0)--(2,2)--(1,1);
\fill[black!35!white] (1,1)--(3,-1)--(4,0)--(2,2)--(1,1);
\draw [->] (3.3,-0.6)-- node[midway, sloped, above,black]{$e_4$}(3.6,-0.3);
\draw [->] (1.4,1.3)-- node[midway, sloped, below,black]{$e_4$}(1.7,1.6);
\draw [->] (3.3,0.6)-- node[midway, sloped, below,black]{$e_3$}(2.7,1.2);
\draw [->] (2.4,-0.3)-- node[midway, sloped, above,black]{$e_3$}(1.7,0.4);
\end{tikzpicture}
\end{minipage}

We will introduce coordinates $u$ and $\ub$ in $(\mathcal{M}, g)$ through a \textit{double null foliation}\footnote{The detailed construction of double null foliation will be explained in Section \ref{secdnf}.}. With coordinates $u, \ub$, characteristic initial data will be prescribed along incoming null hypersurface $\Hb_0$, where $\ub=0$, and outgoing null hypersurface $H_{u_{\infty}}$, where $u=u_{\infty}$.

If the characteristic initial data are small enough, by Christodoulou-Klainerman's monumental work \cite{Chr-Kl} we have {\color{black}completeness of all forward geodesics}, which {\color{black}implies} that no singularity would form  in the light gray region above. On the other hand, if the initial data are large, in their domain of influence (gray region above) a geometric object, \textit{trapped surface}\footnote{A 2-surface is called a trapped surface if its area element is infinitesimally decreasing along both families of null geodesics emanating from the surface}, may form dynamically. In 1965, Penrose proved the celebrated incompleteness theorem:
\begin{theorem}(Penrose \cite{Penrose})

For spacetime $(\mathcal{M}, g)$ containing a non-compact Cauchy hypersurface and $g$ satisfying (\ref{1.1}),  if $M$ contains a compact trapped surface, then it is future causally geodesically incomplete. 
\end{theorem}

Therefore, in this setting, proving singularity formation in general relativity is transferred into deriving trapped surface formation. And it is crucial to design initial data prescribed along $\Hb_0$ and $H_{u_{\infty}}$. In order to form a trapped surface, according to stability of Minkowski, the initial data picked cannot be small. Moreover we cannot prescribe spherically symmetric data along both $\Hb_0$ and $H_{u_{\infty}}$ either. This is due to a classic theorem of Birkhoff: spherically symmetric Einstein vacuum spacetimes must be either (flat) Minkowskian\footnote{Metric of Minkowskian spacetime in spherical coordinates: $g_{M}=-dt^2+dr^2+r^2(d\theta^2+\sin\theta^2 d\phi^2)$.}\,\footnote{Metric of Minkowskian spacetime in stereographic coordinates: $g_{M}=-dt^2+dr^2+\f{4r^4}{(r^2+\theta_1^2+\theta_2^2)^2}(d\theta_1^2+d\theta_2^2)$. In Section \ref{sec rescale}, we will do a scaling argument in stereographic coordinates.} or (static) Schwarzschild\footnote{Metric of Schwarzschild spacetime: $g_S=-(1-\f{2M}{r})dt^2+({1-\f{2M}{r}})^{-1}dr^2+r^2(d\theta^2+\sin\theta^2 d\phi^2)$. Here $M$ is a {\color{black}constant}. In $g_S$ all metric components are independent of $t$, Schwarzschild spacetime is static, i.e. not changing with $t$.} metrics.  Hence, \textit{large and non-spherically symmetric} initial data are required. At the same time, solving (\ref{1.1}) with large initial data is really hard. For general large data problem, for the evolution of Einstein vacuum equations, we only have local existence result. However, forming a trapped surface at a later time requires {\color{black} a mathematical result \textit{beyond} local existence}.  These render the problem of trapped surface formation to be a really hard one. And it was open for a long time. \\

In 2008, Christodoulou solved this long-standing open problem with a 587-page monumental work \cite{Chr:book}. He designed an open set of large initial data, which have a special structure, called \textit{short pulse}  ansatz.  In particular, this ansatz allows one to consider a hierarchy of large and small quantities, parametrized by a small parameter $\d$. For initial data these quantities behave differently, being of various sizes in term of $\d$. And their sizes form a hierarchy. But for each quantity, surprisingly, its size is almost preserved by the nonlinear evolution. 
Therefore, once this hierarchy is designed for initial data, it remains for later time. With this philosophy, despite being a large data problem, a long time global existence theorem can be established. Moreover, these initial conditions indeed lead to trapped-surface formation in the future of the characteristic initial data prescribed along $\Hb_0$ and $H_{u_\infty}$.\\

Einstein vacuum equations are a nonlinear hyperbolic system, containing many unknowns. Christodoulou controlled all of them term by term. Later, two systematical approaches by Klainerman-Rodnianski \cite{KR:Trapped} and An \cite{An:Trapped} were provided to simplify Christodoulou's main result in \cite{Chr:book}. In \cite{KR:Trapped}, an index $s_1$ called \textit{signature for short pulse} is introduced. With this index, Klainerman and Rodnianski systematically tracked the $\d$-weights used in the estimates. And they gave a simplified and shorter proof of $\d$-hierarchy's almost preserving in a finite region. In \cite{Chr:book}, besides $\d$-weights, Christodoulou also employed weights related to decay and prove his main theorem that \textit{a trapped surface could form dynamically with initial data prescribed arbitrary dispersed at past null infinity}.  In \cite{An:Trapped}, An introduced a new index $s_2$ called \textit{signature for decay rates}. With the help of this new index, An extended Klainerman and Rodnianski's result \cite{KR:Trapped} from a finite region to an infinite region and re-proved Christodoulou's main theorem in \cite{Chr:book} with around 120 pages. The proof in \cite{An:Trapped} {\color{black} is still quite long because of}:
\begin{itemize}
\item \underline{Obstruction I}: Even with the systematical approach as in \cite{KR:Trapped}, there are still a few anomalous terms in $\d$-weights. To deal with $\d$-anomaly, it takes some pages. Moreover, the $\d$-anomaly would be more severe when using more angular derivatives. Hence in \cite{KR:Trapped} Klainerman and Rodnianski tried to use the least amount of angular derivatives and they didn't use angular derives with order higher than 2. 
\item \underline{Obstruction II}: In \cite{An:Trapped}, An wanted to re-prove the main result in \cite{Chr:book} with the same amount of angular derivatives used in Christodoulou's proof. In both \cite{Chr:book} and \cite{An:Trapped}, two angular derivatives of curvature components are employed. For energy estimates with such limited angular derives, to avoid losing of derivatives we have to go through an additional technical-and-difficult section \textit{elliptic estimates for the third derivatives of Ricci coefficients}. This also prolongs the proof. 

\end{itemize}

In this paper, we find new ways to avoid both obstructions: 
\subsection{New Ingredients}

\begin{enumerate}
\item In Einstein vacuum spacetimes, peeling ties conformal compactification and plays an important role in small data problems (see \cite{KN:peeling} and reference therein). In this paper, via systematically capturing peeling properties with signature for decay rates $s_2$, we find that peeling would also be vital for problems in large data regime.  \\

\item 
In all preceding works, a colored region on the left in the below is considered, where $\d$ is a small positive parameter and all a priori estimates are established with $\d$ and $|u|$ weights. For example, for geometric quantities $\chih, \rho$ (to be defined in Section \ref{define geometric quantities}) we have
$$\|\chih\|_{L^{\infty}(\S)}\leq \f{\d^{-\f12}}{|u|} \,\, \mbox{in \cite{Chr:book, KR:Trapped}, and} \,\, \|\rho\|_{L^{\infty}(\S)}\leq \f{\d a}{|u|^3} \,\, \mbox{in \cite{A-L}.}  $$
In this paper, with a large positive universal number $a$, we consider a different spacetime region (the colored region on the right). And for the characteristic initial data, we construct a new hierarchy based on geometric peeling properties:  we design new weighted norms, and the weights are only depending on index $s_2$ (\textit{signature for decay rates}), which was introduced by An in \cite{An:Trapped}. Since we don't use parameter $\d$, we don't need the index $s_1$ (signature for short-pulse) any more. \textit{With these new norms and new approach, we can avoid all the $\d$-anomaly.}  (This overcomes \underline {Obstruction I}.)

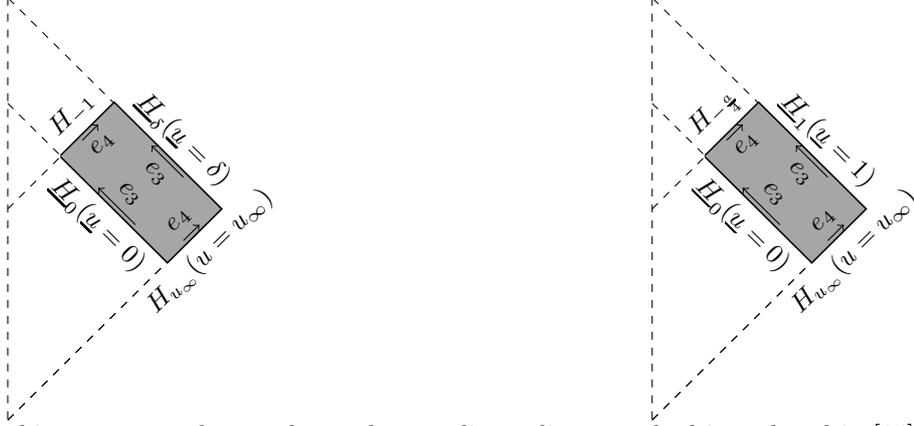
\begin{minipage}[!t]{0.5\textwidth}
\begin{tikzpicture}[scale=0.7]
\draw [white](3,-1)-- node[midway, sloped, below,black]{$H_{u_{\infty}}(u=u_{\infty})$}(4,0);
\draw [white](1,1)-- node[midway,sloped,above,black]{$H_{-1}$}(2,2);
\draw [white](2,2)--node [midway,sloped,above,black] {$\Hb_{\delta}(\ub=\delta)$}(4,0);
\draw [white](1,1)--node [midway,sloped, below,black] {$\Hb_{0}(\ub=0)$}(3,-1);
\draw [dashed] (0, 4)--(0, -4);
\draw [dashed] (0, -4)--(4,0)--(0,4);
\draw [dashed] (0,0)--(2,2);
\draw [dashed] (0,-4)--(4,0);
\draw [dashed] (0,2)--(3,-1);
\draw [very thick] (1,1)--(3,-1)--(4,0)--(2,2)--(1,1);
\fill[black!35!white] (1,1)--(3,-1)--(4,0)--(2,2)--(1,1);
\draw [->] (3.3,-0.6)-- node[midway, sloped, above,black]{$e_4$}(3.6,-0.3);
\draw [->] (1.4,1.3)-- node[midway, sloped, below,black]{$e_4$}(1.7,1.6);
\draw [->] (3.3,0.6)-- node[midway, sloped, below,black]{$e_3$}(2.7,1.2);
\draw [->] (2.4,-0.3)-- node[midway, sloped, above,black]{$e_3$}(1.7,0.4);
\end{tikzpicture}
\end{minipage}
\hspace{0.02\textwidth} 
\begin{minipage}[!t]{0.5\textwidth}
\begin{tikzpicture}[scale=0.7]
\draw [white](3,-1)-- node[midway, sloped, below,black]{$H_{u_{\infty}}(u=u_{\infty})$}(4,0);
\draw [white](1,1)-- node[midway,sloped,above,black]{$H_{-\f{a}{4}}$}(2,2);
\draw [white](2,2)--node [midway,sloped,above,black] {$\Hb_{1}(\ub=1)$}(4,0);
\draw [white](1,1)--node [midway,sloped, below,black] {$\Hb_{0}(\ub=0)$}(3,-1);
\draw [dashed] (0, 4)--(0, -4);
\draw [dashed] (0, -4)--(4,0)--(0,4);
\draw [dashed] (0,0)--(2,2);
\draw [dashed] (0,-4)--(4,0);
\draw [dashed] (0,2)--(3,-1);
\draw [very thick] (1,1)--(3,-1)--(4,0)--(2,2)--(1,1);
\fill[black!35!white] (1,1)--(3,-1)--(4,0)--(2,2)--(1,1);
\draw [->] (3.3,-0.6)-- node[midway, sloped, above,black]{$e_4$}(3.6,-0.3);
\draw [->] (1.4,1.3)-- node[midway, sloped, below,black]{$e_4$}(1.7,1.6);
\draw [->] (3.3,0.6)-- node[midway, sloped, below,black]{$e_3$}(2.7,1.2);
\draw [->] (2.4,-0.3)-- node[midway, sloped, above,black]{$e_3$}(1.7,0.4);
\end{tikzpicture}
\end{minipage}   

\item In this paper, we also employ and generalize a direct method {\color{black}introduced in \cite{Hol} and used in \cite{An:Trapped} and other papers} for deriving energy estimates. The direct method is based on pairing second Bianchi equations $D_{[\iota}R_{\nu\tau]\varphi\lambda}=0$, putting suitable weight for each equation,  and doing integration by parts to cancel the borderline terms.\footnote{The classic approach to deriving energy estimates with Bel-Robinson tensors as in \cite{Chr:book} and \cite{KR:Trapped} is avoid, since for higher order energy estimates there are many more terms (including borderline terms) from deformation tensors would appear. } This approach works well even for higher order energy estimates for Einstein vacuum equations, since potential borderline terms are cancelled and there is no new type of borderline term popping up. In this paper, we adopt and generalize this method and give a \textit{systematical} approach for deriving energy estimates with angular {\color{black}derivatives} of any (high) order. This enables us to use Sobolev's inequality directly. 
\textit{We avoid all technical and long calculations for elliptic estimates.} (This overcomes  \underline{Obstruction II}.)
\begin{remark} 
In \cite{A-A} we are extending the method and result of this paper to the Einstein-Maxwell system. And there we notice that, even for Einstein-Maxwell system, the additional elliptic-estimate part cannot be avoided. The simplification of avoiding elliptic estimates in this paper is because of the special Ricci-flat structures of Einstein vacuum equations. 
\end{remark}

\item In both \cite{KR:Trapped} and \cite{An:Trapped}, the following H\"older's inequality in scale invariant norms is crucial. 
$$\|\phi_1\cdot\phi_2\|_{L_{sc}^{2}(\S)}\leq\frac{\d^{\f12}}{|u|}\|\phi_1\|_{L_{sc}^{\infty}(S)}\|\phi_2\|_{L_{sc}^{2}(\S)}.$$
This inequality tells us if all terms are normal (their scale invariant norms are of size $1$), the nonlinear terms are lower order compared with linear terms. Hence, in the proof we only need to track the linear terms and few anomalous terms, which reduces the workload significantly.  For $|u|\geq 1$, the smallness gained in above inequality is coming from $\d$ being sufficient small. \\

\noindent While for here, when rewriting H\"older's inequality in the new scale invariant norms, we have  (see \ref{Holder's})  
\begin{equation}\label{new holder}
\|\phi_1\cdot\phi_2\|_{L_{sc}^{2}(\S)}\leq\frac{1}{|u|}\|\phi_1\|_{L_{sc}^{\infty}(S)}\|\phi_2\|_{L_{sc}^{2}(\S)}.
\end{equation}
For $|u|\geq a/4$ and $a$ being sufficiently large, the smallness gained in (\ref{new holder}) is coming from $|u|$ weight. In another word,  in the new spacetime region, \textit{the peeling property (encoded in scale invariant norms through the signature for decay rates $s_2$) provides the crucial gain of smallness. And signature $s_2$ captures the information of geometric peeling properties in a systematical way.}  \\

\noindent Putting all the ingredients together, in Section \ref{sec setting}-Section \ref{TSF} we obtain a very direct and short (self-contained) proof, showing that a trapped surface ($S_{-a/4, 1}$) could form in evolution.\textit{The new ansatz and hierarchies designed in this paper are interesting extensions of the established short-pulse method.} \\

\item The results above are also related to a scale-critical theorem near the center. In Section \ref{sec rescale} we observe a new coordinate transformation (rescaling). Under this rescaling, we establish a correspondence between the following two regions:

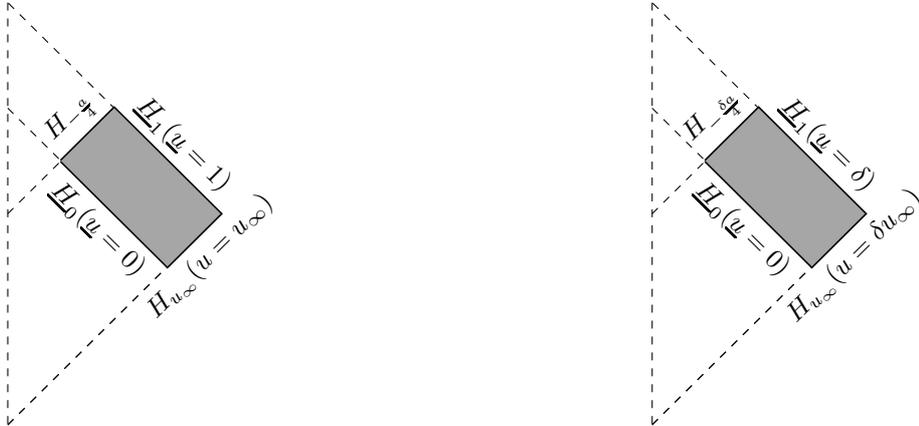
\begin{minipage}[!t]{0.5\textwidth}
\begin{tikzpicture}[scale=0.7]
\draw [white](3,-1)-- node[midway, sloped, below,black]{$H_{u_{\infty}}(u=u_{\infty})$}(4,0);
\draw [white](1,1)-- node[midway,sloped,above,black]{$H_{-\f{a}{4}}$}(2,2);
\draw [white](2,2)--node [midway,sloped,above,black] {$\Hb_{1}(\ub=1)$}(4,0);
\draw [white](1,1)--node [midway,sloped, below,black] {$\Hb_{0}(\ub=0)$}(3,-1);
\draw [dashed] (0, 4)--(0, -4);
\draw [dashed] (0, -4)--(4,0)--(0,4);
\draw [dashed] (0,0)--(2,2);
\draw [dashed] (0,-4)--(4,0);
\draw [dashed] (0,2)--(3,-1);
\draw [very thick] (1,1)--(3,-1)--(4,0)--(2,2)--(1,1);
\fill[black!35!white] (1,1)--(3,-1)--(4,0)--(2,2)--(1,1);
\end{tikzpicture}
\end{minipage}
\hspace{0.02\textwidth} 
\begin{minipage}[!t]{0.5\textwidth}
\begin{tikzpicture}[scale=0.7]
\draw [white](3,-1)-- node[midway, sloped, below,black]{$H_{u_{\infty}}(u=\d u_{\infty})$}(4,0);
\draw [white](1,1)-- node[midway,sloped,above,black]{$H_{-\f{\d a}{4}}$}(2,2);
\draw [white](2,2)--node [midway,sloped,above,black] {$\Hb_{1}(\ub=\d)$}(4,0);
\draw [white](1,1)--node [midway,sloped, below,black] {$\Hb_{0}(\ub=0)$}(3,-1);
\draw [dashed] (0, 4)--(0, -4);
\draw [dashed] (0, -4)--(4,0)--(0,4);
\draw [dashed] (0,0)--(2,2);
\draw [dashed] (0,-4)--(4,0);
\draw [dashed] (0,2)--(3,-1);
\draw [very thick] (1,1)--(3,-1)--(4,0)--(2,2)--(1,1);
\fill[black!35!white] (1,1)--(3,-1)--(4,0)--(2,2)--(1,1);
\end{tikzpicture}
\end{minipage}   
We can translate mathematical results in one picture (left) into the other (right). Since all the estimates derived in Section \ref{sec setting}-Section \ref{TSF} are uniform for $u_{\infty}$, in the right picture we could keep $\d$ and let $u_{\infty}\rightarrow -\infty$, which gives \textit{a scale-critical trapped surface formation criterion from past null infinity}. Here $S_{-\d a/4, \, \d}$ is a tiny trapped surface formed with radius $\d a$. \textit{This is the (sharp) scale critical extension of Christodoulou's monumental work \cite{Chr:book}}. If we let $a=\d^{-1}$, we then recover Christodoulou's main theorem in \cite{Chr:book}.\\

\item In \cite{A-L}, there are three parameters $\{a, b, \d\}$ satisfying $1\ll b\leq \at \leq \d^{-\f12}$. With renormalization techniques, in \cite{A-L} An and Luk derived results scale-critical for $\d$ and also sharp for $a$.  If we let $b=\at$, in Section \ref{rescale bounds} we will see that \textit{by using signature $s_2$ (peeling property), the new approach in this paper would not only systematically sharpen a priori estimates obtained by An-Luk in \cite{A-L}, but also it avoids the technical geometric renormalizations in \cite{A-L} completely.} This paper serves as a more intrinsic and more concise reproof and extension of \cite{A-L} (assuming $b=\at$).\\

For a note on the development of this direction, {\color{black}by designing and employing a different hierarchy}, in \cite{A-L} An and Luk {\color{black}improved \cite{Chr:book} and} proved the first scale-critical result for Einstein vacuum equations. With the same small parameter $\d$, with relatively larger initial data Christodoulou formed a trapped surface of radius 1; while with much smaller initial data An and Luk formed a  trapped surface of radius $\d a$, where $a$ is a universal large constant like $1000$.\footnote{Letting $a=\d^{-1}$, in a finite region they recover Christodoulou's main result of \cite{Chr:book}. } {\color{black}An and Luk want to form a tiny trapped surface with radius $\d a$}, hence they have to deal with the region very close to {\color{black}the center}. In this region all the geometric quantities have growth rates. To bounded these growth rates, they employed weighted estimates as well as several crucial geometric renormalizations. 

Since \cite{A-L} is scale critical, one can keep $a$ as a universal constant and let $\d\rightarrow 0$. Hence a series of trapped surfaces (with radius shrinking to $0$) are obtained. In \cite{An: AH}, An further explored this idea. Together with an elliptic approach to identify the boundary, An showed that a whole black hole region could emerge dynamically from just a ``point" $O$ in the spacetime. For an open set of initial data, this boundary (apparent horizon) is proved to be smooth and spacelike except at $O$. The second law of black hole mechanics is further verified and a conjecture of Ashtekar was proved in \cite{An: AH}.  \\

\item Since An-Luk only dealed with finite spacetime regions in \cite{A-L}, the main result  (Theorem \ref{thm3}) in this paper could also be viewed as  an extension of \cite{A-L}. (There is another way to extend  An-Luk \cite{A-L} by doing a rescaling. In Section \ref{rescale and comparison} we outline that approach and make the comparison.)\\

\end{enumerate}

In summary, the approach in this paper synthesizes new ideas outlined above and 
captures the geometric structures of Einstein vacuum equations in a systematical way via signature for decay rates $s_2$. \textit{The intrinsic\footnote{{\color{black}It comes from spacetime conformal compactification. See \cite{KN:peeling}.}} peeling property plays a crucial role!} Now the new proof of trapped surface formation is self-contained and is less than 50 pages. It simplifies and extends Christodoulou's monumental work \cite{Chr:book} to a (sharp) scale-critical result. It gives another proof, a systematical improvement and an extension of one of the main conclusions in \cite{A-L} by An-Luk. It also has a few very interesting applications \cite{An:TSA}.  

\subsection{Other Related Results}
Besides the results described above, many other improvement or extensions of \cite{Chr:book} have also been achieved. 

In \cite{Chr:book}, Christodoulou required both a homogenous upper and a homogenous lower bound for his short-pulse initial data. The upper bound ensures the semi-global existence of Einstein vacuum equations up to the region, where a trapped surface is about to emerge. The homogenous lower bound is used to confirm trapped-surface formation.   With the same initial data upper bound as in \cite{Chr:book}, in \cite{K-L-R} Klainerman, Luk and Rodnianski relaxed the lower bound requirement vastly. They replaced \textit{inf} by \textit{sup} and obtained a remarkable anisotropic result.

For Einstein vacuum equations, interested readers are also referred to \cite{Dafermos, Le, L-Y, L-R, R-S} and references therein. For Einstein equations coupled with matter, Yu \cite{Yu1, Yu2} extended the result of \cite{KR:Trapped} and obtained similar results for Einstein-Maxwell system with signature for short-pulse. In a recent paper \cite{L-L} by Li and Liu, they studied Einstein-scalar field system and  an almost scale-critical trapped surface formation criterion is achieved.
\\

Next, we start to explain the physical intuition behind trapped surface formation.
\subsection{Heuristic Argument}\label{heuristic}
We consider a spacetime region foliated by incoming and outgoing null hypersurfaces, i.e. $\Hb_{\ub}$ and $H_u$, respectively. Here $\Hb_{\ub}$ and $H_u$ are level sets of two optical functions, which satisfy 
$$g^{\mu\nu}\partial_{\mu}u\partial_{\nu}u=0, \quad \mbox{and} \quad g^{\mu\nu}\partial_{\mu}\ub\partial_{\nu}\ub=0.$$
For the colored region, we have $u_{\infty}\leq u \leq -a/4 <0$ and $0\leq \ub \leq 1$.
Here each point ($S_{u,\ub}=H_u \cap \Hb_{\ub}$) in the Penrose's diagram is corresponding to a 2-sphere.\\

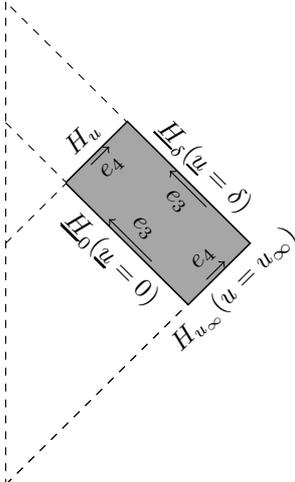
\begin{minipage}[!t]{0.4\textwidth}
\begin{tikzpicture}[scale=0.8]
\draw [white](3,-1)-- node[midway, sloped, below,black]{$H_{u_{\infty}}(u=u_{\infty})$}(4,0);
\draw [white](1,1)-- node[midway,sloped,above,black]{$H_u$}(2,2);
\draw [white](2,2)--node [midway,sloped,above,black] {$\Hb_{\delta}(\ub=\delta)$}(4,0);
\draw [white](1,1)--node [midway,sloped, below,black] {$\Hb_{0}(\ub=0)$}(3,-1);
\draw [dashed] (0, 4)--(0, -4);
\draw [dashed] (0, -4)--(4,0)--(0,4);
\draw [dashed] (0,0)--(2,2);
\draw [dashed] (0,-4)--(4,0);
\draw [dashed] (0,2)--(3,-1);
\draw [very thick] (1,1)--(3,-1)--(4,0)--(2,2)--(1,1);
\fill[black!35!white] (1,1)--(3,-1)--(4,0)--(2,2)--(1,1);
\draw [->] (3.3,-0.6)-- node[midway, sloped, above,black]{$e_4$}(3.6,-0.3);
\draw [->] (1.4,1.3)-- node[midway, sloped, below,black]{$e_4$}(1.7,1.6);
\draw [->] (3.3,0.6)-- node[midway, sloped, below,black]{$e_3$}(2.7,1.2);
\draw [->] (2.4,-0.3)-- node[midway, sloped, above,black]{$e_3$}(1.7,0.4);
\end{tikzpicture}
\end{minipage}
\hspace{0.02\textwidth} 
\begin{minipage}[!t]{0.5\textwidth}
Let $e_3$ and $e_4$ be null vectors and be tangent to each $\Hb_{\ub}$ and $H_u$, respectively. Moreover, we require $g(e_3, e_4)=-2$.  These $\{e_3,e_4\}$ are called a null pair.\\

On each $\S$, we also define {\color{black}$\{e_A, e_B\}_{A, B=1,2}$ an arbitrary tangent frame on it.} \\

We then define null second fundamental forms $\chi_{AB}$, $\chib_{AB}$ associated with $\S$:
$$\chi_{AB}:=g(D_{e_A}e_4, e_B), \quad \chib_{AB}:=g(D_{e_A}e_3, e_B).$$

\end{minipage}

We further decompose $\chi_{AB}$ and $\chib_{AB}$ into trace part $\tr\chi$, $\tr\chib$ and traceless part $\chih_{AB}, \,\, \chibh_{AB}$:
$$\chi_{AB}=\f12\tr\chi \cdot \gamma_{AB}+\chih_{AB}, \quad \quad \chib_{AB}=\f12\tr\chib \cdot \gamma_{AB}+\chibh_{AB},$$
where $\gamma_{AB}$ is the induced metric on $\S$.

A \textit{trapped surface} is a $2$-sphere, of which both null expansions are negative, i.e.
$$\tr\chi<0 \, \, \mbox{\textit{and}} \, \, \tr\chib<0 \, \, \mbox{\textit{hold pointwisely on}} \, \, \S.$$

We will prescribe Minkowkian data along $\Hb_0$, i.e. each $S_{u,0}$ is a standard $2$-sphere embedded in Minkowski spacetime with radius $|u|$. For Minkowskian data, we have
$$\tr\chi(u,0)=\frac{2}{|u|}, \quad \tr\chib(u,0)=-\frac{2}{|u|}.$$

It is easy to show that $\tr\chib$ is always negative: for initial data along $H_{u_{\infty}}$, we have $\tr\chib(u_{\infty},\ub)=-{2}/{|u_{\infty}|}+l.o.t.\footnote{In this article, we use $l.o.t.$ to mean lower order terms.}<0$. Moreover $\tr\chib$ is decreasing along $e_3$ direction
$$\nab_3\tr\chib=-\f12(\tr\chib)^2-|\chibh|^2+l.o.t.,$$
this implies $\tr\chib<0$ in the whole colored region.

For $\chi$, from $\mbox{Ric}_{44}=0$, we have two transport equations:
\begin{equation}\label{trapped surface 1}
\nabla_4\tr\chi+\frac{1}{2}(\tr\chi)^2=-|\hat{\chi}|^{2}+l.o.t.,
\end{equation}
and
\begin{equation}\label{trapped surface 2}
\nabla_3\hat{\chi}+\frac{1}{2}\tr{\underline{\chi}}\hat{{\chi}}=l.o.t..
\end{equation}
Using $\nab_4 \tr\chi\leq -|\chih|^2,$ we have
$$\tr\chi(u,\ub)\leq \tr\chi(u,0)-\int_0^{\ub}|\chih|^2(u,{\ub}')d{\ub}'=\frac{2}{|u|}-\int_0^{\ub}|\chih|^2(u,{\ub}')d{\ub}'.$$
With the derived fact $\tr\chib=-{2}/{|u|}+l.o.t.$, \eqref{trapped surface 2} would {\color{black} imply} 
$$|u|^2|\chih|^2(u,\ub)= |u_{\infty}|^2|\chih|^2(u_{\infty},\ub)+l.o.t.$$
These imply that along $H_{-a/4}$
\begin{equation}\label{TP1}
\begin{split}
\tr\chi(-\f14a,\ub)\leq& \tr\chi(-\f14a,0)-\int_0^{\ub}|\chih|^2(-\f14a,{\ub}')d{\ub}'+l.o.t.\\
=&\frac{2}{|\f14 a|}-\frac{|u_{\infty}|^2}{|\f14 a|^2}\int_0^{\ub}|\chih|^2(u_{\infty},{\ub}')d{\ub}'+l.o.t.
\end{split}
\end{equation}
If we choose $\chih(u_{\infty}, \ub)$ along $H_{u_{\infty}}$ such that 
\begin{equation}\label{chih infinity}
\|\chih\|_{L^{\infty}(S_{u_{\infty},\ub})} \approx \f{\at}{|u_{\infty}|}, \quad \mbox{and} \quad |u_{\infty}|^2\int_0^{1}|\chih|^2(u_{\infty}, \ub')d\ub'\geq a.
\end{equation}
Then from (\ref{TP1}) we arrive at  
$$\tr\chi(-\f14a,1)\leq\frac{2}{|\f14 a|}-\frac{|u_{\infty}|^2}{|\f14 a|^2}\int_0^{1}|\chih|^2(u_{\infty},{\ub}')d{\ub}'+l.o.t.<\f{8}{a}-\f{16}{a}+l.o.t.< 0. $$
Hence, $S_{-a/4, 1}$ is a trapped surface.
\begin{remark}
In the argument above, choosing $\chih$ which satisfies (\ref{chih infinity}) is crucial. We make the following choice
\begin{equation}\label{chih initial}
|\chih|(u_{\infty}, \ub)\approx \at/|u_{\infty}|,
\end{equation}
 which will provide a new hierarchy (in terms of $a$ and $u$) for all geometric components.  
\end{remark}

At the same time, to rigorously verify this heuristic argument, we need to overcome two main difficulties:
\begin{enumerate}
\item We need to make sure that all lower order terms listed above are truly negligible compared with main terms.
Since Einstein vacuum equations are a coupled system of many geometric quantities, this requires detailed understandings of nonlinear interaction of all geometric quantities.

\item We need to prove a global-existence theorem in the large data regime. The physical intuition behind is that \textit{focusing of gravitational waves leads to trapped-surface formation}. With arbitrary dispersed data at past null infinity, we need to ensure that the gravitational radiation can go sufficiently far inside from past null infinity. From PDE point of view, this means to establish a global existence result for Einstein vacuum equations (\ref{1.1}) with no symmetric assumption. This will be a large data problem for an energy super-critical hyperbolic system. \\
\end{enumerate}

In the below we outline an approach, which overcomes these two difficulties. 

\subsection{Geometric Quantities and Signature for Decay Rates}\label{define geometric quantities} 
In our dynamical spacetime, different curvature components and Ricci coefficients would behave distinguishingly. We hence decompose them with respect to a null frame $e_3, e_4$ and a frame $e_1, e_2$ tangent to the 2-sphere $S_{u,\ub}$.

Denote the indices $A,B$ to be $1,2$. With frames $\{e_3, e_4, e_A,  e_B\}$, we define the (null) curvature components:
 \begin{equation}
\begin{split}
\a_{AB}&=R(e_A, e_4, e_B, e_4),\quad \, \,\,   \ab_{AB}=R(e_A, e_3, e_B, e_3),\\
\b_A&= \frac 1 2 R(e_A,  e_4, e_3, e_4) ,\quad \bb_A =\frac 1 2 R(e_A,  e_3,  e_3, e_4),\\
\rho&=\frac 1 4 R(e_4,e_3, e_4,  e_3),\quad \sigma=\frac 1 4  \,^*R(e_4,e_3, e_4,  e_3).
\end{split}
\end{equation}
Here $\, ^*R$ is the Hodge dual of $R$. 

\noindent Denote $D_A:=D_{e_{A}}$. We define Ricci coefficients:

 \begin{equation}
\begin{split}
&\chi_{AB}=g(D_A e_4,e_B),\, \,\, \quad \chib_{AB}=g(D_A e_3,e_B),\\
&\eta_A=-\frac 12 g(D_3 e_A,e_4),\quad \etab_A=-\frac 12 g(D_4 e_A,e_3),\\
&\omega=-\frac 14 g(D_4 e_3,e_4),\quad\,\,\, \omegab=-\frac 14 g(D_3 e_4,e_3),\\
&\zeta_A=\frac 1 2 g(D_A e_4,e_3),
\end{split}
\end{equation}
We further decompose $\chi$ and $\chib$ into trace and traceless part. Denote $\chih_{AB}$ and $\chibh_{AB}$ are the traceless part of $\chi_{AB}$ and $\chib_{AB}$ respectively. 

To capture the information of their behaviours, to each $\phi\in \{\a,\beta,\rho,\sigma, \beb,\ab, \chi,\chib,\zeta,\eta,\etb,\omega,\omb\}$ we assign a number called \textit{signatures} $s_2(\phi)$ to it. The rule is the following:

\begin{equation}\label{signature 21}
s_2(\phi):=0\cdot{N_4}(\phi)+0.5\cdot{N_A}(\phi)+1\cdot{N_3}(\phi)-1.
\end{equation}
Here $N_4(\phi)$ is the number of times $e_4$ appears in the definition of $\phi$. 
Similarly we define $N_3(\phi)$ and $N_A(\phi)$ where $A=1,2$.

For example, in the definition for $\etab_A=-\f12 g(D_4 e_A, e_3)$,  we have one $e_4$, one $e_A$ and one $e_3$. Hence
$$N_4(\etab_A)=1,\quad N_A(\etab_A)=1, \quad N_3(\etab_A)=1.$$ 
According to (\ref{signature 21}), $\etab_A$ has signature
$$s_2(\etab_A)=0\cdot 1+0.5\cdot 1+1\cdot 1-1=0.5.$$ 
Similarly, for $\chi_{AB}=g(D_A e_4, e_B)$ we have
$$N_4(\chi_{AB})=1, \quad N_A(\chi_{AB})=2, \quad N_3(\chi_{AB})=0.$$ 
Hence (\ref{signature 21}) implies
$$s_2(\chi_{AB})=0\cdot 1+0.5\cdot 2+1\cdot 0-1=0.$$ 
Gather these signatures, we have the \textit{signature table}:\\

\begin{tabular}{|r|r|r|r|r|r|r|r|r|r|r|r|r|r|r|r|r|}
  \hline
     & $\alpha$ & $\beta$ & $\rho$ & $\sigma$ & $\underline{\beta}$ & $\underline{\alpha}$ & $\chi$ & $\omega$ & $\zeta$ & $\eta$ & $\underline{\eta}$ & $\mbox{tr}\underline{\chi}$ & $\hat{\underline{\chi}}$ & $\underline{\omega}$ \\
  $s_2$ & 0 & 0.5 & 1 & 1 & 1.5 & 2 & 0 & 0 & 0.5 & 0.5 & 0.5 & 1 & 1 & 1 \\
  \hline
\end{tabular} \\

Based on signature $s_2(\phi)$, we then define \textit{scale invariant norms}: 

\begin{equation}\label{si norms}
\begin{split}
\|\phi\|_{L^{\infty}_{sc}(\S)}:=&a^{-s_2(\phi)}|u|^{2s_2(\phi)+1}\|\phi\|_{L^{\infty}(\S)},\\
\|\phi\|_{L^{2}_{sc}(\S)}:=&a^{-s_2(\phi)}|u|^{2s_2(\phi)}\|\phi\|_{L^{2}(\S)}.
\end{split}
\end{equation}

\begin{remark} A main reason for using scale invariant norms is that for most geometric quantities $\phi$, we will show that $\|\phi\|_{L^{\infty}_{sc}(\S)}$ and $\|\phi\|_{L^{2}_{sc}(\S)}$ are of size $1$. Later we call these $\phi$ normal terms. Through the definitions in (\ref{si norms}), the $a$-weights and $u$-weights are naturally built in the norms. Furthermore, one important identity holds for nonlinear interactions.  From the definition of signature\footnote{More details will be provided in Section \ref{Signatures}.}, we have
$$s_2(\phi_1\cdot\phi_2)=s_2(\phi_1)+s_2(\phi_2).$$
With it, we could rewrite H\"older's inequality in scale invariant norms and obtain

$$\|\phi_1\cdot\phi_2\|_{L^2_{sc}(\S)}\leq \f{1}{|u|}\|\phi_1\|_{L^{\infty}_{sc}(\S)}\|\phi_2\|_{L^{2}_{sc}(\S)}.$$
In the spacetime region studied, we have $1/|u|\leq 1/\at \ll 1.$  Since $a$ is a large universal number, the above inequality tells us, if all the terms are normal, then the nonlinear interactions can be treated as lower order terms. Therefore, only rare anomalous terms are left for further analysis. 
\end{remark}
\begin{remark}
Using signature $s_2$ will also help a lot in deriving (higher order) energy estimates, which are the core of the global existence result. 
\end{remark}

\subsection{Main Results}
In Sections \ref{secbasic}-\ref{energy estimate} we will first derive
\begin{theorem}(An Existence Result near Past Null Infinity)\label{main.thm1}

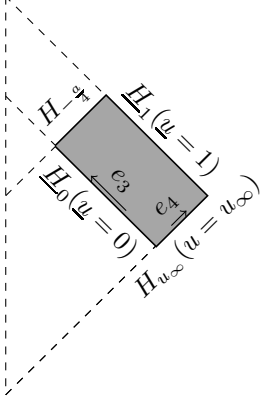
\begin{minipage}[!t]{0.27\textwidth}
\begin{tikzpicture}[scale=0.66]
\draw [white](3,-1)-- node[midway, sloped, below,black]{$H_{u_{\infty}}(u=u_{\infty})$}(4,0);
\draw [white](1,1)-- node[midway,sloped,above,black]{$H_{-\f{a}{4}}$}(2,2);
\draw [white](2,2)--node [midway,sloped,above,black] {$\Hb_{1}(\ub=1)$}(4,0);
\draw [white](1,1)--node [midway,sloped, below,black] {$\Hb_{0}(\ub=0)$}(3,-1);
\draw [dashed] (0, 4)--(0, -4);
\draw [dashed] (0, -4)--(4,0)--(0,4);
\draw [dashed] (0,0)--(2,2);
\draw [dashed] (0,-4)--(4,0);
\draw [dashed] (0,2)--(3,-1);
\draw [very thick] (1,1)--(3,-1)--(4,0)--(2,2)--(1,1);
\fill[black!35!white] (1,1)--(3,-1)--(4,0)--(2,2)--(1,1);
\draw [->] (3.3,-0.6)-- node[midway, sloped, above,black]{$e_4$}(3.6,-0.3);
\draw [->] (2.4,-0.3)-- node[midway, sloped, above,black]{$e_3$}(1.7,0.4);
\end{tikzpicture}
\end{minipage}
\hspace{0.01\textwidth} 
\begin{minipage}[!t]{0.63\textwidth}

Given $\M I^{(0)}$, there exists a sufficiently large $a_0=a_0(\M I^{(0)})$. For $0<a_0<a$, with initial data:
\begin{itemize}
\item $\sum_{i\leq 10, k\leq 3}a^{-\frac12}\|\nab^{k}_4(|u_{\infty}|\nab)^{i}\chih_0\|_{L^{\infty}(S_{u_{\infty},\ub})}\leq \M I^{(0)}$ 

\noindent along $u=u_{\infty}$, 
   
\item Minkowskian initial data  along $\ub=0$,

\end{itemize}
Einstein vacuum equations (\ref{1.1}) admit a unique smooth solution in the colored region:
$$u_{\infty}\leq u \leq -a/4, \quad 0\leq \ub \leq 1.$$
\end{minipage}   

\end{theorem}

We then verify the heuristic argument with estimates derived in Sections \ref{secbasic}-\ref{energy estimate}. In Section \ref{TSF}, we prove

\begin{theorem}(Formation of Trapped Surfaces)\label{main.thm2}

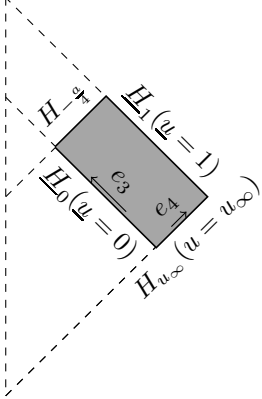
\begin{minipage}[!t]{0.27\textwidth}
\begin{tikzpicture}[scale=0.66]
\draw [white](3,-1)-- node[midway, sloped, below,black]{$H_{u_{\infty}}(u=u_{\infty})$}(4,0);
\draw [white](1,1)-- node[midway,sloped,above,black]{$H_{-\f{a}{4}}$}(2,2);
\draw [white](2,2)--node [midway,sloped,above,black] {$\Hb_{1}(\ub=1)$}(4,0);
\draw [white](1,1)--node [midway,sloped, below,black] {$\Hb_{0}(\ub=0)$}(3,-1);
\draw [dashed] (0, 4)--(0, -4);
\draw [dashed] (0, -4)--(4,0)--(0,4);
\draw [dashed] (0,0)--(2,2);
\draw [dashed] (0,-4)--(4,0);
\draw [dashed] (0,2)--(3,-1);
\draw [very thick] (1,1)--(3,-1)--(4,0)--(2,2)--(1,1);
\fill[black!35!white] (1,1)--(3,-1)--(4,0)--(2,2)--(1,1);
\draw [->] (3.3,-0.6)-- node[midway, sloped, above,black]{$e_4$}(3.6,-0.3);
\draw [->] (2.4,-0.3)-- node[midway, sloped, above,black]{$e_3$}(1.7,0.4);
\end{tikzpicture}
\end{minipage}
\hspace{0.01\textwidth} 
\begin{minipage}[!t]{0.63\textwidth}

Given $\M I^{(0)}$, there exists a sufficiently large $a_0=a_0(\M I^{(0)})$.  For $0<a_0<a$, solving Einstein vacuum equations (\ref{1.1}) with initial data:
\begin{itemize}
\item $\sum_{i\leq 10, k\leq 3}a^{-\frac12}\|\nab^{k}_4(|u_{\infty}|\nab)^{i}\chih_0\|_{L^{\infty}(S_{u_{\infty},\ub})}\leq \M I^{(0)}$ 

\noindent along $u=u_{\infty}$,
   
\item Minkowskian initial data  along $\ub=0$,

\item $\int_0^1|u_{\infty}|^2|\chih_0|^2(u_{\infty}, \ub')d\ub'\geq a$ for every direction   

\noindent along $u=u_{\infty}$,
\end{itemize}
we have that $S_{-a/4,1}$ is a trapped surface.
\end{minipage}   

\end{theorem}

In Section \ref{sec rescale}, we will describe a new coordinate transformation. With it we convert above results into our main conclusion
\begin{theorem}(A Scale-Critical Trapped Surface Formation Criterion from Past Null Infinity)\label{thm3}

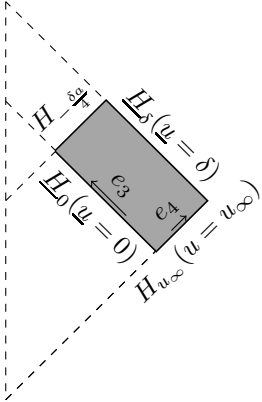
\begin{minipage}[!t]{0.27\textwidth}
\begin{tikzpicture}[scale=0.66]
\draw [white](3,-1)-- node[midway, sloped, below,black]{$H_{u_{\infty}}(u=u_{\infty})$}(4,0);
\draw [white](1,1)-- node[midway,sloped,above,black]{$H_{-\f{\d a}{4}}$}(2,2);
\draw [white](2,2)--node [midway,sloped,above,black] {$\Hb_{\d}(\ub=\d)$}(4,0);
\draw [white](1,1)--node [midway,sloped, below,black] {$\Hb_{0}(\ub=0)$}(3,-1);
\draw [dashed] (0, 4)--(0, -4);
\draw [dashed] (0, -4)--(4,0)--(0,4);
\draw [dashed] (0,0)--(2,2);
\draw [dashed] (0,-4)--(4,0);
\draw [dashed] (0,2)--(3,-1);
\draw [very thick] (1,1)--(3,-1)--(4,0)--(2,2)--(1,1);
\fill[black!35!white] (1,1)--(3,-1)--(4,0)--(2,2)--(1,1);
\draw [->] (3.3,-0.6)-- node[midway, sloped, above,black]{$e_4$}(3.6,-0.3);
\draw [->] (2.4,-0.3)-- node[midway, sloped, above,black]{$e_3$}(1.7,0.4);
\end{tikzpicture}
\end{minipage}
\hspace{0.01\textwidth} 
\begin{minipage}[!t]{0.65\textwidth}

Given $\M I^{(0)}$, for fixed $\d$ there exists a sufficiently large $a_0=a_0(\M I^{(0)},\d)$.  For $0<a_0<a$, solving Einstein vacuum equations (\ref{1.1}) with initial data:
\begin{itemize}
\item $\sum_{i\leq 10, k\leq 3}a^{-\frac12}\|(\d\nab_4)^k(|u_{\infty}|\nab)^{i}\chih_0\|_{L^{\infty}(S_{u_{\infty},\ub})}\leq \M I^{(0)}$ 

\noindent along $u=u_{\infty}$,
   
\item Minkowskian initial data  along $\ub=0$,

\item $\int_0^{\delta}|u_{\infty}|^2|\chih_0|^2(u_{\infty}, \ub')d\ub'\geq \d a$ for every direction   

\noindent along $u=u_{\infty}$,
\end{itemize}
we have that $S_{-\d a/4,\d}$ is a trapped surface.
\end{minipage}   
\end{theorem}

\begin{remark}
Theorem \ref{thm3} is an extension of An-Luk \cite{A-L} to allow characteristic initial data prescribed at very far away (at $u=u_{\infty}$). At the same time, Theorem \ref{thm3} could also be viewed as a scale-critical extension of Christodoulou \cite{Chr:book}. 
\end{remark}
In Theorem \ref{thm3}, if we further choose $a=4c\cdot\d^{-1}$, where $0< c\leq 1$ being of size $1$, we then have

\begin{corollary}\label{Corollary1.5} (Recovery of Christodoulou's monumental work \cite{Chr:book})

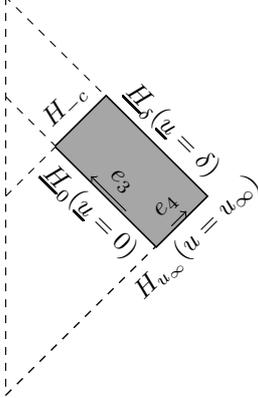
\begin{minipage}[!t]{0.27\textwidth}
\begin{tikzpicture}[scale=0.66]
\draw [white](3,-1)-- node[midway, sloped, below,black]{$H_{u_{\infty}}(u=u_{\infty})$}(4,0);
\draw [white](1,1)-- node[midway,sloped,above,black]{$H_{-c}$}(2,2);
\draw [white](2,2)--node [midway,sloped,above,black] {$\Hb_{\d}(\ub=\d)$}(4,0);
\draw [white](1,1)--node [midway,sloped, below,black] {$\Hb_{0}(\ub=0)$}(3,-1);
\draw [dashed] (0, 4)--(0, -4);
\draw [dashed] (0, -4)--(4,0)--(0,4);
\draw [dashed] (0,0)--(2,2);
\draw [dashed] (0,-4)--(4,0);
\draw [dashed] (0,2)--(3,-1);
\draw [very thick] (1,1)--(3,-1)--(4,0)--(2,2)--(1,1);
\fill[black!35!white] (1,1)--(3,-1)--(4,0)--(2,2)--(1,1);
\draw [->] (3.3,-0.6)-- node[midway, sloped, above,black]{$e_4$}(3.6,-0.3);
\draw [->] (2.4,-0.3)-- node[midway, sloped, above,black]{$e_3$}(1.7,0.4);
\end{tikzpicture}
\end{minipage}
\hspace{0.01\textwidth} 
\begin{minipage}[!t]{0.65\textwidth}

Given $\M I^{(0)}$ and constant $c$ (where $0<c \leq 1$ being of size $1$), there exists a sufficiently small $\d_0=\d_0(\M I^{(0)},c)$.  For $0<\d<\d_0\ll c$, solving Einstein vacuum equations (\ref{1.1}) with initial data:
\begin{itemize}
\item $\sum_{i\leq 10, k\leq 3}\d^{\frac12}\|(\d\nab_4)^k(|u_{\infty}|\nab)^{i}\chih_0\|_{L^{\infty}(S_{u_{\infty},\ub})}\leq \M I^{(0)}$ 

\noindent along $u=u_{\infty}$,
   
\item Minkowskian initial data  along $\ub=0$,

\item $\int_0^{\delta}|u_{\infty}|^2|\chih_0|^2(u_{\infty}, \ub')d\ub'\geq 4c$ for every direction   

\noindent along $u=u_{\infty}$,
\end{itemize}
we have that $S_{-c, \d}$ is a trapped surface.
\end{minipage}

Note: we could also obtain a priori bounds that are in line with \cite{Chr:book} by Christodoulou. 

\end{corollary}

\section{Setting, Equations and Notations} \label{sec setting}

\subsection{Double Null Foliation}\label{secdnf}
We construct a double null foliation in a neighborhood of $S_{u_{\infty},0}$ as follows: 

\begin{minipage}[!t]{0.35\textwidth}
\begin{tikzpicture}
\draw [white](3,-1)-- node[midway, sloped, below,black]{$H_{u_{\infty}}(u=u_{\infty})$}(4,0);

\draw [white](2,2)--node [midway,sloped,above,black] {$\Hb_1(\ub=1)$}(4,0);
\draw [white](1,1)--node [midway,sloped, below,black] {$\Hb_{0}(\ub=0)$}(3,-1);
\draw [dashed] (0, 4)--(0, -4);
\draw [dashed] (0, -4)--(4,0)--(0,4);
\draw [dashed] (0,0)--(2,2);
\draw [dashed] (0,-4)--(2,-2);
\draw [dashed] (0,2)--(3,-1);
\draw [very thick] (1,1)--(3,-1)--(4,0)--(2,2)--(1,1);
\fill [black!35!white]  (1,1)--(3,-1)--(4,0)--(2,2)--(1,1);
\draw [white](1,1)-- node[midway,sloped,above,black]{$H_{u}$}(2,2);
\draw [->] (3.3,-0.6)-- node[midway, sloped, above,black]{$L'$}(3.6,-0.3);
\draw [->] (1.4,1.3)-- node[midway, sloped, below,black]{$L'$}(1.7,1.6);
\draw [->] (3.3,0.6)-- node[midway, sloped, below,black]{$\Lb'$}(2.7,1.2);
\draw [->] (2.4,-0.3)-- node[midway, sloped, above,black]{$\Lb'$}(1.7,0.4);
\end{tikzpicture}
\end{minipage}
\hspace{0.02\textwidth} 
\begin{minipage}[!t]{0.55\textwidth}

Starting from a point $p$ on $2$-sphere $S_{u_{\infty}, 0}$, in $2$-dimensional $T_p^{\perp}S_{u_{\infty}, 0}$, we could find two future-directed vectors $L'_p, \underline{L}'_p$ such that 
$$g(L'_p, L'_p)=0, \, \, g(\underline{L}'_p, \underline{L}'_p)=0, \,  \, g(L'_p, \Lb'_p)=-2.\footnote{ $\{L'_p, \Lb'_p\}$ are uniquely determined up to a scaling factor $\lambda>0$: \quad $\{L'_p, \Lb'_p\}\rightarrow \{\lambda L'_p, \lambda^{-1}\Lb'_p\}$. }$$ Based on $p$ and along $L'_p$ direction, a unique geodesic $l_p$ is sent out.  We extend $L'$ along $l_p$ such that 
$D_{L'}L'=0$. We then have $l_p$ is null. This is because $g(L'_p, L'_p)=0$ and
$$L'(g(L', L'))=2g(D_{L'}L', L')=0.$$
We hence have $g(L',L')=0$ along $l_p$.
Gathering all the $\{l_p\}$ together, we then have an outgoing null hypersurface called $H_{u_{\infty}}$. Similarly, we obtain the incoming null hypersurface $\Hb_0$ emitting from $S_{u_{\infty},0}$.\\

Note that, by above construction, for each point $q$ on $H_{u_{\infty}}$ or $\Hb_0$, in $T_q H_{u_{\infty}}$ or $T_q \Hb_0$, there is the preferred null vector $L'_q$ or $\Lb'_q$ associated with $q$. \\
\end{minipage}

We define function $\O$ to be 1 on $S_{u_{\infty},0}$ and extend $\O$ as a continuous function along $H_{u_{\infty}}$ and $\Hb_0$. \footnote{For a general double null foliation, we have the gauge freedom of choosing how to extend $\O$ along $H_{u_{\infty}}$ and $\Hb_0$. In this paper, we extend $\O\equiv 1$ on both $H_{u_{\infty}}$ and $\Hb_0$.} We consider vector fields 
$$L=\O^2 L' \, \,\, \mbox{along}\, \, \, H_{u_{\infty}}, \,\, \mbox{and} \, \, \Lb=\O^2 \Lb' \, \, \,\mbox{along}\, \, \, \Hb_0,$$
and define functions
$$\ub \,\, \mbox{on} \,\, H_{u_{\infty}}\,\,\, \mbox{satisfying}\, \, \, L\ub=1 \,\, \mbox{and} \,\, \ub=0\, \,\mbox{on}  \,\, S_{u_{\infty}, 0},$$
$$u \, \, \mbox{on} \, \, \Hb_0 \, \, \, \mbox{satisfying} \, \, \,\Lb u=1 \, \, \mbox{and} \, \, u=u_{\infty}\, \,\mbox{on} \, \, S_{u_{\infty}, 0}.$$
Let $S_{u_{\infty}, \ub'}$ be the embedded $2$-surface on $H_{u_{\infty}}$, such that $\ub=\ub'$. 
Similarly, define $S_{u', 0}$ to be the embedded $2$-surface on $\Hb_0$, such that $u=u'$. 

At each point $q$ on $2$-surface $S_{u_{\infty}, \ub'}$, we already have the preferred outgoing null vector $L'_q$ tangent to $H_{u_{\infty}}$. Hence, at $q$ we can also fix a unique incoming null vector $\Lb_q'$ via requiring
$$g(\Lb'_q, \Lb'_q)=0 \quad \mbox{and} \quad g(\Lb'_q, L'_q)=-2\O^{-2}|_q.$$
There exists a unique geodesic $\underline{l}_q$ emitting from $q$ with direction $\Lb'$. We then extend $\Lb'$ along $\underline{l}_q$ through $D_{\Lb'}\Lb'=0$.  Gathering all the $\{\underline{l}_q\}$ for $q\in S_{u_{\infty}, \ub'}$, we have constructed the incoming null hypersurface $\Hb_{\ub'}$ emitting from $S_{u_{\infty}, \ub'}$. Similarly, from $S_{u',0}$ we also construct the outgoing null hypersurface $H_{u'}$. We further define $2$-sphere $S_{u', \ub'}:=H_{u'}\cap \Hb_{\ub'}$. 

At each point $p$ of $S_{u',\ub'}$, we define positive-valued function $\O$ via 
\begin{equation}\label{define Omega}
g(L'_p, \Lb'_p)=:-2\O^{-2}| _p.
\end{equation}
Note $L'_p$ is well-defined on $H_{u'}$, along an outgoing null geodesic $l$ passing through $p$;  $\Lb'_p$ is also well-defined on $\Hb_{\ub'}$, along an incoming null geodesic $\underline{l}$ crossing $p$. \\

These $3$-dimensional incoming null hypersurfaces $\{\Hb_{\ub'}\}_{0\leq \ub' \leq 1}$ and outgoing null hypersurfaces $\{H_{u'}\}_{u_{\infty}\leq u' \leq -a/4}$ together with their intersections $S_{u', \ub'}=H_{u'}\cap \Hb_{\ub'}$ give us the so called \textit{double null foliation}. \\

On $\S$, by (\ref{define Omega}) we have $g(L', \Lb')=-2\O^{-2}$. Thus, $g(\O L', \O \Lb')=-2$. Throughout this paper we will work with the normalized null pair $(e_3, e_4)$:
$$e_3:=\Omega\Lb', \quad e_4:=\Omega L', \,\, \mbox{and} \quad g(e_3, e_4)=-2.$$

Moreover, for the characteristic initial data, we choose the following gauge:
$$\Omega\equiv 1 \quad\mbox{on $H_{u_{\infty}}$ and $\Hb_0$}.$$
\begin{remark}
Functions $u$ and $\ub$ defined above also satisfy the eikonal equations
$$g^{\mu\nu}\partial_\mu u\partial_\nu u=0,\quad g^{\mu\nu}\partial_\mu\ub\partial_\nu \ub=0.$$
And it is straight forward to check
$$L'^\mu=-2g^{\mu\nu}\partial_\nu u,\quad \Lb'^\mu=-2g^{\mu\nu}\partial_\nu \ub, \quad  L\ub=1, \quad \Lb u=1.  $$
Here $\Lb:=\Omega^2\Lb', \quad L:=\Omega^2 L'$ are also called equivariant vector fields. 
\end{remark}

\subsection{The Coordinate System}\label{coordinates}
We will use a coordinate system $(u,\ub, \theta^1, \theta^2)$. Here $u$ and $\ub$ are defined above. To get $(\theta^1, \theta^2)$ for each point on $\S$, we follow the approach in Chapter 1 of \cite{Chr:book}: we first define a coordinate system $(\theta^1, \theta^2)$ on $S_{u_{\infty},0}$. Since $S_{u_{\infty},0}$ is the standard $2$-sphere in Minkowskian spacetime, here we use the coordinates of stereographic projection. Then we extend this coordinate system to $\Hb_0$ by requiring
$$\Ls_{\Lb} \theta^A=0\mbox{ on $\Hb_0$}. \footnote{On $\Hb_0$, we have $\O=1$ and $\Ls_{\Lb} \theta^A=\f{\partial}{\partial u}\theta^A$.}$$
Here $\Ls_L$ is the restriction of the Lie derivative to $TS_{u,\ub}$. In another word,  given a point $p$ on $S_{u_{\infty}, 0}$, assuming $l_p$ is the incoming null geodesics on $\Hb_0$ emitting from $p$, then all the points along $l_p$ have the same angular coordinate $(\theta^1, \theta^2)$.
We further extend this coordinate system from $\Hb_0$ to the whole spacetime under requirement
$$\Ls_L \th^A=0,$$ 
i.e. all the points along the same outgoing null geodesics (along $L$) on $H_u$ have the same angular coordinate. We hence have established a coordinate system in a neighborhood of $S_{u_{\infty}, 0}$. With this coordinate system, we can rewrite $e_3$ and $e_4$ as 
$$e_3=\Omega^{-1}\left(\frac{\partial}{\partial u}+d^A\frac{\partial}{\partial \th^A}\right), e_4=\Omega^{-1}\frac{\partial}{\partial \ub}.$$
And the Lorentzian metric $g$ takes the form
\begin{equation}\label{equation g}
g=-2\O^2(du\otimes d\ub+d\ub\otimes du)+\gamma_{AB}(d\theta^A-d^A du)\otimes(d\theta^B-d^B du).
\end{equation}
We require {\color{black}$d^A$ to satisfy $d^A=0$} on $\Hb_0$.

\subsection{Equations}\label{seceqn}
We then decompose curvature components and Ricci coefficients with respect to  null frames $e_3, e_4$ and frames $e_1, e_2$ tangent to the 2-sphere $S_{u,\ub}$. Denote the indices $A,B$ to be $1,2$. With frame $\{e_3, e_4, e_A,  e_B\}$, we define null curvature components:
 \begin{equation}\label{def curvatures}
\begin{split}
\a_{AB}&=R(e_A, e_4, e_B, e_4),\quad \, \,\,   \ab_{AB}=R(e_A, e_3, e_B, e_3),\\
\b_A&= \frac 1 2 R(e_A,  e_4, e_3, e_4) ,\quad \bb_A =\frac 1 2 R(e_A,  e_3,  e_3, e_4),\\
\rho&=\frac 1 4 R(e_4,e_3, e_4,  e_3),\quad \sigma=\frac 1 4  \,^*R(e_4,e_3, e_4,  e_3). 
\end{split}
\end{equation}
Here $\, ^*R$ stands for the Hodge dual of $R$. Denote $D_A:=D_{e_{A}}$. We define Ricci coefficients:

\begin{equation}\label{def Ricci coefficients}
\begin{split}
&\chi_{AB}=g(D_A e_4,e_B),\, \,\, \quad \chib_{AB}=g(D_A e_3,e_B),\\
&\eta_A=-\frac 12 g(D_3 e_A,e_4),\quad \etab_A=-\frac 12 g(D_4 e_A,e_3),\\
&\omega=-\frac 14 g(D_4 e_3,e_4),\quad\,\,\, \omegab=-\frac 14 g(D_3 e_4,e_3),\\
&\zeta_A=\frac 1 2 g(D_A e_4,e_3).
\end{split}
\end{equation}
Let $\gamma_{AB}$ be the induced metric on $\S$, we further decompose $\chi, \chib$ into
$$\chi_{AB}=\f12\tr\chi\cdot \gamma_{AB}+\chih_{AB}, \quad \chib_{AB}=\f12\tr\chib\cdot \gamma_{AB}+\chibh_{AB}.$$
Here $\chih_{AB}$ and $\chibh_{AB}$ are the corresponding traceless parts. \\

Denote $\nab$ to be the induced covariant derivative operator on $S_{u,\ub}$. And let $\nab_3$ and $\nab_4$ to be the projections of covariant derivatives $D_3$ and $D_4$ to $S_{u,\ub}$.
By the definitions of Ricci coefficients, one can verify:
\begin{equation}
\begin{split}
&\omega=-\frac 12 \nab_4 (\log\Omega),\qquad \omegab=-\frac 12 \nab_3 (\log\Omega),\\
&\eta_A=\zeta_A +\nab_A (\log\Omega),\quad \etab_A=-\zeta_A+\nab_A (\log\Omega).
\end{split}
\end{equation}
\\ 

We then define several different contractions between tensors.  Let
$$(\phi^{(1)}\hot\phi^{(2)})_{AB}:=\phi^{(1)}_A\phi^{(2)}_B+\phi^{(1)}_B\phi^{(2)}_A-\gamma_{AB}(\phi^{(1)}\cdot\phi^{(2)}) \quad\mbox{for one forms $\phi^{(1)}_A$, $\phi^{(2)}_A$,}$$
$$(\phi^{(1)}\wedge\phi^{(2)})_{AB}:=\eps^{AB}(\gamma^{-1})^{CD}\phi^{(1)}_{AC}\phi^{(2)}_{BD}\quad\mbox{for symmetric $2$-tensors $\phi^{(1)}_{AB}$, $\phi^{(2)}_{AB}$}.$$
Here $\eps$ is the volume form associated to the metric $\gamma$. For simplicity, we employ  $\phi^{(1)}\cdot\phi^{(2)}$ to represent an arbitrary contraction of the tensor product of $\phi^{(1)}$ and $\phi^{(2)}$ with respect to the metric $\gamma$. We also use $\div$, $\curl$ and $\tr$ operators. For totally symmetric tensors, define these operators by
$$(\div\phi)_{A_1...A_r}:=\nabla^B\phi_{BA_1...A_r},$$
$$(\curl\phi)_{A_1...A_r}:=\eps^{BC}\nabla_B\phi_{CA_1...A_r},$$
$$(\mbox{tr}\phi)_{A_1...A_{r-1}}:=(\gamma^{-1})^{BC}\phi_{BCA_1...A_{r-1}}.$$
{\color{black}
We also define by $^*$ for $1$-forms and symmetric $2$-tensors respectively as follows (note that on $1$-forms this is the Hodge dual on $\S$):
$$
{^*}\phi_A:=\gamma_{AC}\eps^{CB}\phi_B,
$$
$$
{^*}\phi_{AB}:=\gamma_{BD}\eps^{DC}\phi_{AC}.
$$
And define the operator $\nab\hot$ on a $1$-form $\phi_A$ by
$$
(\nab\hot\phi)_{AB}:=\nab_A\phi_B+\nab_B\phi_A-\gamma_{AB}\div\phi.
$$
}
\\

We are ready to state the transport equations for curvature components and Ricci coefficients. Rewrite the second Bianchi equations $D_{[\iota}R_{\nu\tau]\varphi\lambda}=0$ with null frames, we arrive at 
\begin{equation}
\label{eq:null.Bianchi}
\begin{split}
&\nab_3\alpha+\frac 12 \trchb \alpha=\nabla\hot \beta+ 4\omegab\alpha-3(\chih\rho+^*\chih\sigma)+
(\zeta+4\eta)\hot\beta, \\
&\nab_4\beta+2\trch\beta = \div\alpha - 2\omega\beta +  \eta \alpha,\\
&\nab_3\beta+\trchb\beta=\nabla\rho + 2\omegab \beta +^*\nabla\sigma +2\chih\cdot\betab+3(\eta\rho+^*\eta\sigma),\\
&\nab_4\sigma+\frac 32\trch\sigma=-\div^*\beta+\frac 12\chibh\cdot ^*\alpha-\zeta\cdot^*\beta-2\etab\cdot
^*\beta,\\
&\nab_3\sigma+\frac 32\trchb\sigma=-\div ^*\betab+\frac 12\chih\cdot ^*\alphab-\zeta\cdot ^*\betab-2\eta\cdot 
^*\betab,\\
&\nab_4\rho+\frac 32\trch\rho=\div\beta-\frac 12\chibh\cdot\alpha+\zeta\cdot\beta+2\etab\cdot\beta,\\
&\nab_3\rho+\frac 32\trchb\rho=-\div\betab- \frac 12\chih\cdot\alphab+\zeta\cdot\betab-2\eta\cdot\betab,\\
&\nab_4\betab+\trch\betab=-\nabla\rho +^*\nabla\sigma+ 2\omega\betab +2\chibh\cdot\beta-3(\etab\rho-^*\etab\sigma),\\
&\nab_3\betab+2\trchb\, \betab=-\div\alphab-2\omegab\betab+\etab \cdot\alphab,\\
&\nab_4\alphab+\frac 12 \trch\alphab=-\nabla\hot \betab+ 4\omega\alphab-3(\chibh\rho-^*\chibh\sigma)+
(\zeta-4\etab)\hot \betab.
\end{split}
\end{equation}
Here $^*$ denotes the Hodge dual on $S_{u,\ub}$. The above transport equations for curvature are called \textit{null Bianchi equations}.

We then rewrite $\mbox{Ric}_{\mu\nu}=0$ with null frames.  For $\chi$ and $\chib$ we have 
\begin{equation}
\label{null.str1}
\begin{split}
\nab_4 \trch+\frac 12 (\trch)^2&=-|\chih|^2-2\omega \trch,\\
\nab_4\chih+\trch \chih&=-2 \omega \chih-\alpha,\\
\nab_3 \trchb+\frac 12 (\trchb)^2&=-2\omegab \trchb-|\chibh|^2,\\
\nab_3\chibh + \trchb\,  \chibh&= -2\omegab \chibh -\alphab,\\
\nab_4 \trchb+\frac1 2 \trch \trchb &=2\omega \trchb +2\rho- \chih\cdot\chibh +2\div \etab +2|\etab|^2,\\
\nab_4\chibh +\frac 1 2 \trch \chibh&=\nab\widehat{\otimes} \etab+2\omega \chibh-\frac 12 \trchb \chih +\etab\widehat{\otimes} \etab,\\
\nab_3 \trch+\frac1 2 \trchb \trch &=2\omegab \trch+2\rho- \chih\cdot\chibh+2\div \eta+2|\eta|^2,\\
\nab_3\chih+\frac 1 2 \trchb \chih&=\nab\widehat{\otimes} \eta+2\omegab \chih-\frac 12 \trch \chibh +\eta\widehat{\otimes} \eta.
\end{split}
\end{equation}
For the remaining Ricci coefficients, we have
\begin{equation}
\label{null.str2}
\begin{split}
\nabla_4\eta&=-\chi\cdot(\eta-\etab)-\b,\\
\nabla_3\etab &=-\chib\cdot (\etab-\eta)+\bb,\\
\nabla_4\omegab&=2\omega\omegab+\frac 34 |\eta-\etab|^2-\frac 14 (\eta-\etab)\cdot (\eta+\etab)-
\frac 18 |\eta+\etab|^2+\frac 12 \rho,\\
\nabla_3\omega&=2\omega\omegab+\frac 34 |\eta-\etab|^2+\frac 14 (\eta-\etab)\cdot (\eta+\etab)- \frac 18 |\eta+\etab|^2+\frac 12 \rho.\\
\end{split}
\end{equation}
\noindent These above transport equations for Ricci coefficients are call \textit{null structure equations}.

\begin{remark}
In this article, we will also need another form of equation $\nab_4\tr\chib$
\begin{equation*}\label{eqn 4 trchib}
\begin{split}
&\nab_4(\tr\chib+\f{2}{|u|})+\f12\tr\chi(\tr\chib+\f{2}{|u|})\\
=&\f{1}{|u|}\tr\chi-\f{4}{|u|}\o+2\o(\tr\chib+\f{2}{|u|})+2\rho-\chih\cdot\chibh+2\div\etb+2|\etb|^2\\
=&\f12(\tr\chib+\f{2}{|u|})\tr\chi-\f12\tr\chib\tr\chi+2\tr\chib\o+2\rho-\chih\cdot\chibh+2\div\etb+2|\etb|^2,
\end{split}
\end{equation*}
and another form of $\nab_3\tr\chib$
\begin{equation}\label{eqn 2 trchib}
\begin{split}
\nab_3(\tr\chib+\f{2}{|u|})+\tr\chib(\tr\chib+\f{2}{|u|})=\f{2}{|u|^2}(\O^{-1}-1)+\f12(\tr\chib+\f{2}{|u|})(\tr\chib+\f{2}{|u|})-2\omb\tr\chib-|\chibh|^2.
\end{split}
\end{equation} 

\end{remark}

When embedding $\S$ into $4$-dimensional spacetime, we have Gauss-Codazzi equations and in null frames we have 

\begin{equation}
\label{null.str3}
\begin{split}
\div\chih&=\frac 12 \nabla \trch - \frac 12 (\eta-\etab)\cdot (\chih -\frac 1 2 \trch\cdot\gamma) -\beta,\\
\div\chibh&=\frac 12 \nabla \trchb + \frac 12 (\eta-\etab)\cdot (\chibh-\frac 1 2   \trchb\cdot\gamma) +\betab,\\
\curl\eta &=-\curl\etab=\sigma +\frac 1 2\chibh \wedge\chih,\\
K&=-\rho+\frac 1 2 \chih\cdot\chibh-\frac 1 4 \trch \trchb.
\end{split}
\end{equation}
Here $K$ is Gaussian curvature of spheres $S_{u,\ub}$.

\subsection{Integration} Let $U$ be a coordinate patch on $\S$. Denote $p_U$ to be the corresponding partition of unity. 
For a function $\phi$,  we define its integration on $S_{u,\ub}$, $H_u$ and $\Hb_{\ub}$ via 
\begin{equation}\label{int S}
\int_{S_{u,\ub}} \phi :=\sum_U \int_{-\infty}^{\infty}\int_{-\infty}^{\infty}\phi\cdot p_U\cdot \sqrt{\det\gamma}\,d\th^1 d\th^2,
\end{equation}
$$\int_{H_{u}^{(0,\ub)}} \phi :=\sum_U \int_0^{\ub}\int_{-\infty}^{\infty}\int_{-\infty}^{\infty}\phi\cdot 2p_U\cdot\Omega\cdot\sqrt{\det\gamma}\,d\th^1 d\th^2d\ub',$$
$$\int_{H_{\ub}^{(\ui,u)}} \phi :=\sum_U \int_{u_{\infty}}^{u}\int_{-\infty}^{\infty}\int_{-\infty}^{\infty}\phi\cdot 2p_U\cdot\Omega\cdot\sqrt{\det\gamma}\,d\th^1 d\th^2du'.$$
Let $D_{u ,\ub}$ be the region $u_{\infty}\leq u'\leq u$, $0\leq \ub'\leq \ub$. We define the integration of $\phi$ in $D_{u,\ub}$ as
\begin{equation*}
\begin{split}
\int_{D_{u,\ub}} \phi :=&\sum_U \int_{u_{\infty}}^u\int_0^{\ub}\int_{-\infty}^{\infty}\int_{-\infty}^{\infty}\phi\cdot p_U\cdot \O^2\cdot\sqrt{-\det g}\, d\th^1 d\th^2 du' d\ub'.\\
\end{split}
\end{equation*}
We further define the $L^p$ ($1\leq p < \infty$) norms for an arbitrary tensorfield $\phi$:
$$||\phi||_{L^p(S_{u,\ub})}^p:=\int_{S_{u,\ub}} <\phi,\phi>_\gamma^{p/2},$$
$$||\phi||_{L^p(H_u)}^p:=\int_{H_{u}} <\phi,\phi>_\gamma^{p/2},$$
$$||\phi||_{L^p(\Hb_{\ub})}^p:=\int_{\Hb_{\ub}} <\phi,\phi>_\gamma^{p/2}.$$
When $p=\infty$, we define the $L^\infty$ norm by
$$||\phi||_{L^\infty(S_{u,\ub})}:=\sup_{\th\in S_{u,\ub}} <\phi,\phi>_\gamma^{1/2}(\th).$$
We also employ mixed-type norms in this paper: 
$$||\phi||_{L^2_{\ub}L^\infty_u L^p(\S)}:=\left(\int_0^{1}(\sup_{u_{\infty}\leq u \leq -\f{a}{4}}||\phi||_{L^p(S_{u,\ub'})})^2d\ub'\right)^{\frac{1}{2}},$$
$$||\phi||_{L^2_{u}L^\infty_{\ub} L^p(\S)}:=\left(\int_{u_{\infty}}^{-\f{a}{4}}(\sup_{0\leq \ub \leq 1}||\phi||_{L^p(S_{u',\ub})})^2du'\right)^{\frac{1}{2}}.$$
\begin{remark}
In this paper the following Minkowski's inequality will be used frequently:
$$||\phi||_{L^\infty_{\ub}L^2_{u} L^p(\S)}\leq||\phi||_{L^2_{u}L^\infty_{\ub} L^p(\S)}.$$  
\end{remark}

\subsection{Definition of Signatures} \label{Signatures}
As explained in heuristics, we want to prescribe $\chih$ such that
$$|\chih|(u_{\infty}, \ub)\approx \f{\at}{|u_{\infty}|} \quad \mbox{ along } \quad H_{u_{\infty}}.$$ Following the same procedures explained in details in Chapter 2 of \cite{Chr:book}, we obtain the following estimates on $H_{u_{\infty}}$:
\begin{equation}\label{initial hierarchy}
\begin{split}
&|\a|\lesssim \f{\at}{|u_{\infty}|}, \quad |\beta|\lesssim \f{\at}{|u_{\infty}|^2}, \quad |\rho|\lesssim \f{a}{|u_{\infty}|^3},\\
&|\sigma|\lesssim \f{a}{|u_{\infty}|^3}, \quad |\beb|\lesssim \f{a}{|u_{\infty}|^4}, \quad |\ab|\lesssim \f{a^{\f32}}{|u_{\infty}|^5},\\
&|\o|\lesssim \f{1}{|u_{\infty}|}, \quad |\tr\chi|\lesssim \f{1}{|u_{\infty}|}, \quad |\eta|\lesssim \f{\at}{|u_{\infty}|^2}, \quad |\etb|\lesssim \f{\at}{|u_{\infty}|^2},\\
&|\tr\chib-\f{2}{u_{\infty}}|\lesssim \f{a}{|u_{\infty}|^3}\lesssim \f{1}{|u_{\infty}|^2}, \quad |\omb|\lesssim \f{a}{|u_{\infty}|^3}, \quad |\chibh|\lesssim \f{\at}{|u_{\infty}|^2}.
\end{split}
\end{equation}
Note that near $S_{u_{\infty}, 0}$ all geometric quantities have decay rates and \textit{they obey peeling property} (see \cite{KN:peeling}). In PDE estimates, it will be hard to track these $|u|$ and $a$ weights term by term. We hope to design a ``\textit{scale invariant norm} -$L^{\infty}_{sc}(\S)$'' with $|u|$ and $a$ weights built in, such that for most geometric quantities $\phi$, we have 
$$\|\phi\|_{L^{\infty}_{sc}(\S)}\lesssim 1.$$
To achieve this, first we need to find some connections between the definitions of various geometric quantities in (\ref{def curvatures}), (\ref{def Ricci coefficients})  and the $|u_{\infty}|$-weights, $a$-weights listed above. 

By relaxing the above estimate for $\beb \mbox{ and } \ab$
\begin{equation}\label{relax beb ab}
\begin{split}
&\mbox{from }\quad |\beb|\lesssim \f{a}{|u_{\infty}|^4} \quad \mbox{to } \quad |\beb|\lesssim \f{a^{\f32}}{|u_{\infty}|^4},\\
&\mbox{from }\quad |\ab|\lesssim \f{a^{\f32}}{|u_{\infty}|^5} \quad \mbox{to } \quad |\ab|\lesssim \f{a^2}{|u_{\infty}|^5},
\end{split}
\end{equation}
and keeping the other estimates for now, we find a systematical way to define $L^{\infty}_{sc}(\S)$. 

We first introduce \textit{signature for decay rates}: to $\phi\in \{\a,\beta,\rho,\sigma, K, \beb,\ab, \chi,\chib,\zeta,\eta,\etb,\omega,\omb, \gamma\}$, 
we assign signatures $s_2(\phi)$ according to the rule:
\begin{equation*}\label{signature 2}
s_2(\phi):=0\cdot{N_4}(\phi)+\frac{1}{2}\cdot{N_a}(\phi)+1\cdot{N_3}(\phi)-1.
\end{equation*}
$N_4(\phi)$ is the number of times $e_4$ appears in the definition of $\phi$. 
Similarly we define $N_3(\phi)$ and $N_a(\phi)$ where $a=1,2$. Following the definition, we then have the \textit{signature table}\\

\begin{tabular}{|r|r|r|r|r|r|r|r|r|r|r|r|r|r|r|r|r|}
  \hline
     & $\alpha$ & $\beta$ & $\rho$ & $\sigma$ & $K$ & $\underline{\beta}$ & $\underline{\alpha}$ & $\chi$ & $\omega$ & $\zeta$ & $\eta$ & $\underline{\eta}$ & $\mbox{tr}\underline{\chi}$ & $\hat{\underline{\chi}}$ & $\underline{\omega}$ & $\gamma$ \\
  $s_2$ & 0 & 0.5 & 1 & 1 & 1 & 1.5 & 2 & 0 & 0 & 0.5 & 0.5 & 0.5 & 1 & 1 & 1 & 0 \\
  \hline
\end{tabular} \\
\begin{remark}
With above definition, we also have
\begin{equation}\label{signature of derivative}
s_2(\nabla_4{\phi})=s_2(\phi), \quad s_2(\nabla{\phi})=s_2(\phi)+\frac{1}{2}, \quad s_2(\nabla_3{\phi})=s_2(\phi)+1.
\end{equation}
\end{remark}

\subsection{Scale Invariant Norms} \label{Scale invariant norms}

For any horizontal tensor-field $\phi$ with signature $s_2(\phi)$, we further define \textit{scale invariant norms on $S_{u,\ub}$}:
\begin{equation}\label{scale invariant norms}
\begin{split}
\|\phi\|_{L_{sc}^{\infty}(\S)}:=&a^{-s_2(\phi)}|u|^{2s_2(\phi)+1}\|\phi\|_{L^{\infty}(\S)},\\
\|\phi\|_{L_{sc}^{2}(\S)}:=&a^{-s_2(\phi)}|u|^{2s_2(\phi)}\|\phi\|_{L^{2}(\S)},\\
\|\phi\|_{L^1_{sc}(\S)}:=&a^{-s_2(\phi)}|u|^{2s_2(\phi)-1}\|\phi\|_{L^1(\S)}.
\end{split}
\end{equation}
For convenience, along $H_u^{(0,\ub)}$ and $\Hb_{\ub}^{(u_{\infty},u)}$ we also define
\textit{scale invariant norms along null hypersurfaces}
\begin{equation}\label{scale invariant norms 2}
\begin{split}
\|\phi\|^2_{L^2_{sc}(H_u^{(0,\underline{u})})}:=&
\int_0^{\underline{u}}\|\phi\|^2_{L^2_{sc}(S_{u,\underline{u}'})}d\underline{u}',\\
\|\phi\|^2_{L^2_{sc}(\underline{H}_{\underline{u}}^{(u_{\infty},u)})}:=&
\int_{u_{\infty}}^{u}{\frac{a}{|u'|^2}}\|\phi\|^2_{L^2_{sc}(S_{u',\underline{u}})}du'.
\end{split}
\end{equation}

\begin{remark}
Let $\phi\in \{\beta,\rho,\sigma, \beb,\ab, \tr\chi, \tr\chib+\f{2}{|u|}, \eta,\etb,\omega,\omb\}$. After relaxing estimates for $\beb, \ab$ as in (\ref{relax beb ab}), along $H_{u_{\infty}}$, (\ref{initial hierarchy}) could be rewritten in scale invariant norms
\begin{equation}\label{initial scale invariant}
\begin{split}
\f{1}{\at}\|\chih\|_{L^{\infty}_{sc}(S_{u_{\infty},\ub})}&+\f{\at}{|u_{\infty}|}\|\chibh\|_{L^{\infty}_{sc}(S_{u_{\infty},\ub})}+\f{a}{|u_{\infty}|^2}\|\tr\chib\|_{L^{\infty}_{sc}(S_{u_{\infty},\ub})}\\
&+\f{1}{\at}\|\a\|_{L^{\infty}_{sc}(S_{u_{\infty},\ub})}+\|\phi\|_{L^{\infty}_{sc}(S_{u_{\infty},\ub})}\leq 1.
\end{split}
\end{equation}
For most geometric terms, their scale invariant norms are of size 1. But for $\chih$ and $\a$ it  requires an additional smallness $1/\at$, for $\chibh$ it requires an additional smallness $\at/|u_{\infty}|$ and for $\tr\chib$ it requires an additional smallness $a/|u_{\infty}|^2$.  We hence call $\chih, \chibh, \tr\chib, \a$ \textit{anomalous terms}.
\end{remark}

\subsection{Conservation of Signatures}
A key property of signature $s_2$ is that the induced metric $\gamma_{ab}$ on $\S$ satisfies $s_2(\gamma_{ab})=0$. This ensures \textit{signature conservation}:
\begin{equation}\label{signature conservation 2}
s_2(\phi_1\cdot{\phi_2})=s_2(\phi_1)+s_2(\phi_2).
\end{equation}
For example, we have one of the null structure equations
\begin{equation}\label{nab 3 o}
\nabla_3{\omega}=2\omega\underline{\omega}+\frac{3}{4}|\eta-\underline{\eta}|^2+\frac{1}{4}(\eta-\underline{\eta})\cdot(\eta+\underline{\eta})-\frac{1}{8}|\eta+\underline{\eta}|^2+\frac{1}{2}\rho.
\end{equation}
From signature table and (\ref{signature of derivative}), it can be read that all the nonlinear terms and linear terms have the same signature $s_2$, that is 1:
$$s_2(\o\omb)=s_2(\omega)+s_2(\omb)=0+1=1, \quad s_2(\eta\cdot\eta)=s_2(\eta)+s_2(\eta)=\f12+\f12=1,$$
$$s_2(\eta\cdot\etb)=s_2(\eta)+s_2(\etb)=\f12+\f12=1, \quad s_2(\etb\cdot\etb)=s_2(\etb)+s_2(\etb)=\f12+\f12=1,$$
$$s_2(\nab_3\o)=s_2(\o)+1=0+1=1, \quad  s_2(\rho)=1.$$
This delightful fact is true not only for the equation of $\nab_3 \o$, but also true for all null structure equations, null Bianchi equations and constrain equations \footnote{That's because (\ref{1.1}) is a geometric PDE system and it respects some natural scalings.}.  When using scale invariant norms, this key feature enables us to treat all the terms on the right hand side of (\ref{nab 3 o}) as one term, since they share the same signature $s_2$. 

Moreover, when using scale invariant norms to rewrite H\"{o}lder's inequalities, we get
\subsection{H\"older's Inequality in Scale Invariant Norms}
\begin{equation}\label{Holder's}
\begin{split}
\|\phi_1\cdot\phi_2\|_{L_{sc}^{2}(\S)}\leq&\frac{1}{|u|}\|\phi_1\|_{L_{sc}^{\infty}(S)}\|\phi_2\|_{L_{sc}^{2}(\S)},\\
\|\phi_1\cdot\phi_2\|_{L^1_{sc}(\S)}\leq&\f{1}{|u|}\|\phi_1\|_{L^{\infty}_{sc}(\S)}\|\phi_2\|_{L^1_{sc}(\S)},\\
\|\phi_1\cdot\phi_2\|_{L^1_{sc}(\S)}\leq&\f{1}{|u|}\|\phi_1\|_{L^{2}_{sc}(\S)}\|\phi_2\|_{L^2_{sc}(\S)}.
\end{split}
\end{equation}
Note in the region of study we have $1/|u|\leq 4/a<<1$. This means if all terms are normal, the nonlinear terms in (\ref{nab 3 o}) or in other equations could be treated as lower order terms. This will simply the proof a lot.

\subsection{Norms} Here we define norms, which will be used throughout the paper. 

\noindent Let 
\begin{equation}\label{psi Psi}
\t \psi\in\{\o, \tr\chi, \eta, \etb, \omb\}, \t \Psi\in \{\b,\rho, \sigma,\beb,\ab\}, \t \Psi'\in \{\rho, \sigma,\beb,\ab\}.
\end{equation}
We also denote $\tc:=\tr\chib+\f{2}{|u|}$.

\noindent For $0\leq i \leq 6$, we define
\begin{equation}\label{O i infty}
\begin{split}
\mathcal{O}_{i,\infty}(u,\underline{u}):= &\f{1}{\at}\|(\at\nab)^i\chih\|_{L^{\infty}_{sc}(\S)}+\|(\at\nab)^i\t \psi\|_{L^{\infty}_{sc}(\S)}+\f{\at}{|u|}\|(\at\nab)^i\chibh\|_{L^{\infty}_{sc}(\S)}\\
&+\f{a}{|u|^2}\|(\at\nab)^i\tr\chib\|_{L^{\infty}_{sc}(\S)}+\f{a}{|u|}\|(\at\nab)^i \tc\|_{L^{\infty}_{sc}(\S)},
\end{split}
\end{equation}

\begin{equation}\label{R i infty}
\begin{split}
\mathcal{R}_{i,\infty}(u,\underline{u}):= \f{1}{\at}\|(\at\nab)^i\a\|_{L^{\infty}_{sc}(\S)}+\|(\at\nab)^i\t \Psi\|_{L^{\infty}_{sc}(\S)}.
\end{split}
\end{equation}
For $0\leq i \leq {\color{black}10}$, we define
\begin{equation}\label{O i 2}
\begin{split}
\mathcal{O}_{i,2}(u,\underline{u}):=& \f{1}{\at}\|(\at\nab)^i\chih\|_{L^{2}_{sc}(\S)}+\|(\at\nab)^i\t \psi\|_{L^{2}_{sc}(\S)}+\f{\at}{|u|}\|(\at\nab)^i\chibh\|_{L^{2}_{sc}(\S)}\\
&+\f{a}{|u|^2}\|(\at\nab)^i\tr\chib\|_{L^{2}_{sc}(\S)}+\f{a}{|u|}\|(\at\nab)^i \tc\|_{L^{2}_{sc}(\S)}.
\end{split}
\end{equation}
For $0\leq i \leq 9$, we define
\begin{equation}\label{R i 2}
\begin{split}
\mathcal{R}_{i,2}(u,\underline{u}):= \f{1}{\at}\|(\at\nab)^i\a\|_{L^{2}_{sc}(\S)}+\|(\at\nab)^i\t \Psi\|_{L^{2}_{sc}(\S)}.
\end{split}
\end{equation}
For $0\leq i \leq 10$, we define
\begin{equation}\label{R i H}
\mathcal{R}_i(u,\underline{u}):=\f{1}{\at}\|(\at\nab)^i\alpha\|_{L^{2}_{sc}(H_u^{(0,\ub)})}+\|(\at\nab)^i\t\Psi\|_{L^{2}_{sc}(H_u^{(0,\ub)})},
\end{equation}

\begin{equation}\label{R i Hb}
\mathcal{\underline{R}}_i(u,\underline{u}):=\f{1}{\at}\|(\at\nab)^i \beta\|_{L^{2}_{sc}(\underline{H}_{\ub}^{(u_{\infty}, u)})}+\|(\at\nab)^i\t\Psi'\|_{L^{2}_{sc}(\underline{H}_{\ub}^{(u_{\infty}, u)})}.
\end{equation}
We then set $\mathcal{O}_{i,\i},\, \mathcal{O}_{i,2},\, \mathcal{R}_{i,\infty},\, \mathcal{R}_{i, 2},\, \mathcal{R}_i,\, \underline{\mathcal{R}}_i$ to be the supremum over $u,\ub$ in our spacetime region of $\mathcal{O}_{i,\i}(u,\ub), \mathcal{O}_{i,2}(u,\ub), \mathcal{R}_{i,\infty}(u,\ub), \mathcal{R}_{i, 2}(u,\ub), \mathcal{R}_i(u,\ub), \underline{\mathcal{R}}_i(u,\ub)$, respectively. Finally, we define $\mathcal{O}, \mathcal{R}$:
$$\mathcal{O}:=\sum_{i\leq 6}(\mathcal{O}_{i,\infty}+\mathcal{R}_{i,\infty})+\sum_{i\leq 9}(\M O_{i,2}+ \M R_{i,2}),$$
$$\mathcal{R}:=\sum_{i\leq 10}\mathcal{R}_i+\mathcal{\underline{R}}_i.$$
And let $\mathcal{O}^{(0)}, \mathcal{R}^{(0)}, \underline{\M R}^{(0)}$ be the corresponding norms of the initial hypersurfaces $H_{u_{\infty}}$ and $\Hb_0$. 

\noindent Lastly, we define the initial data quantity
$$\mathcal{I}^{(0)}:=\sup_{0\leq \ub \leq 1}\mathcal{I}^{(0)}(\ub),$$
where
\begin{equation*}
\begin{split}
\mathcal{I}^{(0)}(\ub):=&\f{|u_{\infty}| }{\at}\|\chih_0\|_{L^{\infty}(S_{u_{\infty},\ub})}+\sum_{0\leq k \leq 10,}\sum_{0\leq m \leq 20}\f{1}{\at}\|(|u_{\infty}|\nab)^{m}(\nab_4)^k\chih_0\|_{L^2(S_{u_{\infty},\ub})}.
\end{split}
\end{equation*}
Here $\chih_0$ denotes $\chih$ along $H_{u_{\infty}}^{(0,\ub)}$. 

\subsection{Notation}
We collect the notations that are employed for convenience throughout the article:
\begin{itemize}
\item We denote $\sup_{u,\ub}$ to be the supremum over all values of $u,\ub$, where $u_{\infty}\leq u\leq -\f{a}{4}$ and $0\leq \ub \leq 1$. 
\item If $A$ and $B$ are two quantities, we often use $A\lesssim B$ meaning that there exists a constant $C>0$, which is independent of $a$, such that $A\leq CB$. Whenever there is no danger of confusion, we substitute $\leq$ for $\lesssim$. 
\item For equations involving many terms, the coefficients on the left are kept precise. Whenever there is no danger of confusion, the coefficients on the right are allowed to vary up to a nonzero constant.
\item We will employ $(\,,\,)$ to denote sum of all terms, which have one of the components in the bracket. For instance, the notation $\phi_1(\phi_2, \phi_3)$ means 
the sum of all terms in the form of $\phi_1\phi_2$ or $\phi_1\phi_3$.
\item Denote $D$ to be the spacetime region $\{(u,\ub)\,\,| \,\, u_{\infty}\leq u \leq -a/4, \quad 0\leq\ub\leq 1\}.$

\item {\color{black} For integers $i_1\geq 0$ and $i_2\geq 1$, sometimes we use $\nab^{i_1}\p^{i_2}$ to express a product of $i_2$ terms: 
$$\nab^{i_1}\p^{i_2}=\nab^{j_1}\p\cdot\nab^{j_2}\p\cdot\cdot\cdot\nab^{j_{i_2}}\p, \mbox{ where } j_1, j_2,..., j_{i_2}\in \mathbb{N} \mbox{ and } i_1=j_1+j_2+...+j_{i_2}.$$
Here we assume that $j_{i_2}$ is the largest number. }
\end{itemize} 

\section{The Preliminary Estimates}\label{secbasic}  
\subsection{An Approach of Bootstrap}
In this article, we will employ a bootstrap argument to derive uniform upper bounds of $\M O, \M R, \underline{\M R}$ for nonlinear Einstein vacuum equations. Along $H_{u_{\infty}}$ and $\Hb_{0}$, by analysing characteristic initial data we have 
\begin{equation}\label{EBA0}
\M O^{(0)}+\M R^{(0)}+\underline{\M R}^{(0)}\ls \M I^{(0)}.
\end{equation}
Here we have $\M I^{(0)}\lesssim 1$. Our goal is to show that in $D=\{(u,\ub)\, | \, u_{\infty}\leq u \leq -a/4, \,\, 0\leq \ub \leq 1\}$ we have
\begin{equation}\label{EBA1}
\M O(u,\ub)+\M R(u,\ub)+\underline{\M R}(u,\ub)\ls \M I^{(0)}+(\M I^{(0)})^2+1.
\end{equation}
Once these uniform bounds are obtained, by {\color{black} characteristic-initial-data local existence result}\footnote{See full details in Chapter 16 of \cite{Chr:book} or \cite{Luk} or Section 10 of \cite{Dafermos} for a beautiful exposition. {\color{black}}}, the solutions can always be extended a bit towards the future direction of $u$. Hence, uniform estimate (\ref{EBA1}) for $u_{\infty}\leq u \leq -a/4$ implies global existence of Einstein vacuum equations in $D=\{(u,\ub)\, | \, u_{\infty}\leq u \leq -a/4, \,\, 0\leq \ub \leq 1\}$.

To derive the uniform bound (\ref{EBA1}), we make bootstrap assumptions 
\begin{equation}\label{BA.0}
 \M O(u, \ub)\leq O, \quad  \M R(u,\ub) +\underline{\M R}(u,\ub) \leq R.
\end{equation}
Here $O$ and $R$ are large numbers, such that 
$$\M I^{(0)}+(\M I^{(0)})^2+1\ll O, \quad \M I^{(0)}+(\M I^{(0)})^2+1\ll R, \quad \mbox{but} \quad (O+R)^{20}\leq {a^{\f{1}{16}}}.$$
\noindent We also define $\Upsilon=\{u\,|\,\, u_{\infty}\leq u \leq-a/4 \mbox{ and } (\ref{BA.0}) \mbox{ hold for every } 0\leq \ub \leq 1\}$. First, we hope to prove
$\Upsilon=[u_{\infty}, -a/4].$ At $u=u_{\infty}$, we have (\ref{EBA0}). By continuity of solutions (via local existence), for small $\epsilon>0$  it holds for $u_{\infty}\leq u \leq u_{\infty}+\epsilon$
$$ \M O^{(0)}\ls \M I^{(0)}\ll O, \quad  \M R^{(0)} +\underline{\M R}^{(0)} \ls \M I^{(0)}\ll R,$$
$$ \M O(u, \ub)\ls 2\,\M I^{(0)}\ll O, \quad  \M R(u,\ub) +\underline{\M R}(u,\ub) \ls 2\,\M I^{(0)}\ll R.$$ This implies $[u_{\infty}, u_{\infty}+\epsilon]\subseteq \Upsilon$ and $\Upsilon$ is not empty. Since $\Upsilon\subseteq [u_{\infty}, -a/4]$, if we are able to prove that $\Upsilon$ is a set being both open and closed, then we prove $\Upsilon= [u_{\infty}, -a/4]$. Closeness follows from uniform estimates and continuity of solutions in $u$ variable, which doesn't rise a challenge.  Efforts are dedicated to verifying that $\Upsilon$ is open. 

The main parts of this paper are to show that for any $u\in\Upsilon$ we have \footnote{See Remark \ref{bootstrap openness}.}
$$\M O(u,\ub)\ls \M I^{(0)}+\M R(u,\ub)+\Rb(u,\ub)+1+\f{C_1}{a^{\f18}}\cdot (O+R)^{20},\quad \M R(u,\ub)+\underline{\M R}(u,\ub)\ls \M I^{(0)}+(\M I^{(0)})^2+1+\f{C_2}{a^{\f18}}\cdot (O+R)^{20}.$$ Here $C_1$ and $C_2$ are integers independent of $a$ and basically count how many terms popping up in the estimates. By employing $(O+R)^{20}\leq {a^{\f{1}{16}}}$ and further requiring $a$ to be sufficiently large, we obtain
\begin{equation}\label{EBA3}
\begin{split}
&\M R(u,\ub)+\underline{\M R}(u,\ub)\ls \M I^{(0)}+(\M I^{(0)})^2+1+\f{C_2}{a^{\f18}}\cdot a^{\f{1}{16}}\ls \M I^{(0)}+(\M I^{(0)})^2+1,\\
&\M O(u,\ub)\ls \M I^{(0)}+\M R(u,\ub)+\Rb(u,\ub)+1+\f{C_1}{a^{\f18}}\cdot a^{\f{1}{16}}\ls \M I^{(0)}+(\M I^{(0)})^2+1.
\end{split}
\end{equation}
These are improvements of the upper bounds in bootstrap assumptions (\ref{BA.0}):
$$\M O(u,\ub)\leq O, \,\, \M R(u,\ub)+\underline{\M R}(u,\ub)\leq R,$$ where $\M I^{(0)}+(\M I^{(0)})^2+1\ll O, \, \, \M I^{(0)}+(\M I^{(0)})^2+1\ll R.$
By continuity of solutions via local existence, $\Upsilon$ could be extended a bit towards larger $u$. This implies $\Upsilon$ being open. Together with $\Upsilon$ being closed and non-empty, we have $\Upsilon=[u_{\infty}, -a/4]$. Thus, for the whole region $D=\{(u,\ub)\, | \, u_{\infty}\leq u \leq -a/4, \,0\leq \ub \leq 1\}$, estimates in (\ref{BA.0}) hold. They imply (\ref{EBA3}) and bounds in (\ref{EBA1}):
$$\M O(u,\ub)+\M R(u,\ub)+\underline{\M R}(u,\ub)\ls \M I^{(0)}+(\M I^{(0)})^2+1\, \mbox{ in } \,D.$$

\subsection{Estimates for Metric Components}\label{metric}
We derive bound for $\Omega$ first:
\begin{proposition}\label{Omega}
Under the assumptions of Theorem \ref{main.thm1} and bootstrap assumption \eqref{BA.0}, we have
$$\|\Omega-1\|_{L^\i(\S)}\ls \f{O}{|u|}.$$
\end{proposition}
\begin{proof}
Consider the equation
\begin{equation}\label{Omegatransport}
 \omega=-\frac{1}{2}\nabla_4\log\Omega=\frac{1}{2}\Omega\nabla_4\Omega^{-1}=\frac{1}{2}\frac{\partial}{\partial \ub}\Omega^{-1}.
\end{equation}
We integrate respect to $d\ub$.  On $\Hb_0$ we have $\Omega^{-1}=1$ and this leads to
\begin{equation}\label{Omega -1}
||\Omega^{-1}-1||_{L^\infty(S_{u,\ub})}\ls \int_0^{\ub}||\omega||_{L^\infty(S_{u,\ub'})}d\ub'\ls \f{O}{|u|}.
\end{equation}
Here we have used the bootstrap assumption \eqref{BA.0}: 
$$\|\o\|_{L^{\infty}_{sc}(\S)}\leq O \Leftrightarrow \|\o\|_{L^{\infty}(\S)}\leq \f{O}{|u|}.$$
Finally, notice that 
$$\|\Omega-1\|_{L^\i(\S)}\leq\|\Omega\|_{L^\i(\S)}\|\Omega^{-1}-1\|_{L^\i(\S)}\ls(1+\f{O}{|u|})^{-1}\cdot\f{O}{|u|}\ls \f{O}{|u|}.$$
\end{proof}

We then move to control induced metric $\gamma$ on $\S$:
\begin{proposition}\label{gamma}
Under the assumptions of Theorem \ref{main.thm1} and the bootstrap assumptions \eqref{BA.0}, for metric $\gamma$ on $\S$ we have
$$c'\leq \det\gamma\leq C'. $$
Here $C'$ and $c'$ are constants depending only on initial data. Moreover, in $D$
$$|\gamma_{AB}|,|(\gamma^{-1})^{AB}|\leq C'.$$
\end{proposition}
\begin{proof}
We employ the first variation formula $\Ls_L\gamma=2\Omega\chi.$
In coordinates, it states
\begin{equation}\label{ub gamma}
\frac{\partial}{\partial \ub}\gamma_{AB}=2\Omega\chi_{AB}.
\end{equation}
This implies
$$\frac{\partial}{\partial \ub}\log(\det\gamma)=2\Omega\trch.$$
Let $\gamma_0(u,\ub,\th^1,\th^2)=\gamma(u,0,\th^1,\th^2)$. Then with $|2\O\tr\chi|\leq O/|u|$ it follows
$$\f{\det \gamma}{\det \gamma_0}=e^{\int_0^{\ub}2\O\tr\chi d\ub'}\leq e^{\f{O}{a}}.$$
Via Taylor expansion, this implies
\begin{equation}\label{detgaper}
|\det\gamma-\det(\gamma_0)|\leq \det(\gamma_0)|1-e^{\f{O}{a}}|\ls \f{O}{a},
\end{equation}
which gives lower and upper bound for $\det \gamma$.
For $\gamma$, denote $\Lambda$ to be the greater eigenvalue. We have
\begin{equation*}\label{La}
\Lambda\leq\sup_{A,B=1,2}\gamma_{AB},
\end{equation*}

$$\sum_{A,B=1,2}|\chi_{AB}|\leq\Lambda ||\chi||_{L^\infty(S_{u,\ub})},$$

$$|\gamma_{AB}-(\gamma_0)_{AB}|\leq \int_0^{\ub}|\chi_{AB}|d\ub'\leq\Lambda\f{a^{\f12}}{|u|}O\ls\f{O}{\at}.$$
We hence bound $|\gamma_{AB}|$ from above. We further bound $|(\gamma^{-1})^{AB}|$ from above by using the upper bound for $|\gamma_{AB}|$ and the lower bound for $\det\gamma$.
\end{proof}

For metric $\gamma$, we will also need the following
\begin{proposition}\label{gamma2}
We continue to work under the assumptions of Theorem \ref{main.thm1} and the bootstrap assumptions \eqref{BA.0}. Fix a point $(u,\theta^1, \theta^2)$ on the initial hypersurface $\Hb_0$. Along the outgoing null geodesics emitting from $(u, \theta^1, \theta^2)$, denote $\Lambda(\ub)$ and $\lambda(\ub)$ to be the larger and smaller eigenvalue of $\gamma^{-1}(u, \ub=0, \theta^1, \theta^2)\gamma(u, \ub, \theta^1, \theta^2)$. 
Then we have
$$|\Lambda(\ub)-1|+|\lambda(\ub)-1|\leq \f{1}{\at}$$.
\end{proposition}
\begin{proof}
Define $\nu(\ub):=\sqrt{\f{\Lambda(\ub)}{\lambda(\ub)}}.$ Following the derivation of (5.93) in \cite{Chr:book}, by (\ref{ub gamma}) we can derive 
$$\nu(\ub)\leq 1+\int_0^{\ub}|\O\chih(\ub')|_{\gamma} \nu(\ub')d\ub'.$$
Via Gr\"onwall's inequality, this {\color{black}implies} 
\begin{equation}\label{nu}
|\nu(\ub)|\ls 1 \quad \mbox{ and } \quad |\nu(\ub)-1|\leq \f{\at\cdot O}{|u|^2}\leq \f{O}{a^{\f32}}\leq \f{1}{a}.
\end{equation}
The desired estimate follows from (\ref{detgaper}) and (\ref{nu}).
\end{proof}

The above two propositions also imply
\begin{proposition}\label{area}
Under the assumptions of Theorem \ref{main.thm1} and the bootstrap assumptions \eqref{BA.0}, in $D$ we have
$$\sup_{\ub}|\mbox{Area}(S_{u,\ub})-\mbox{Area}(S_{u,0})|\leq\f{O^{\f12}}{a^{\f12}}|u|^2.$$
\end{proposition}
\begin{proof}
This follows from definition in (\ref{int S}) and estimate in \eqref{detgaper}.
\end{proof}

\subsection{Estimates for Transport Equations}\label{transportsec}
In latter sections, we will employ following propositions for transport equations:

\begin{proposition}
Under the assumptions of Theorem \ref{main.thm1} and the bootstrap assumptions \eqref{BA.0}, for an $S_{u,\ub}$ tangent tensor $\phi$ of arbitrary rank, we have \begin{equation}\label{transport1}
 ||\phi||_{L^2(S_{u,\ub})}\ls ||\phi||_{L^2(S_{u,\ub'})}+\int_{\ub'}^{\ub} ||\nabla_4\phi||_{L^2(S_{u,\ub''})}d{\ub''},
\end{equation}
\begin{equation}\label{transport3}
 ||\phi||_{L^2(S_{u,\ub})}\ls ||\phi||_{L^2(S_{u',\ub})}+\int_{u'}^{u} ||\nabla_3\phi||_{L^2(S_{u'',\ub})}d{u''}. 
 \end{equation}
\end{proposition}

\begin{proof}
Here we first prove (\ref{transport1}). For a scalar function $f$, by variation of area formula, we have
\[
 \frac{d}{d\ub}\int_{\S} f=\int_{\S} \left(\frac{df}{d\ub}+\Omega \trch f\right)=\int_{\S} \Omega\left(e_4(f)+ \trch f\right).
\]
Taking $f=|\phi|_{\gamma}^2$, using Cauchy-Schwarz inequality on the sphere and $L^\infty$ bounds for $\Omega$ and $\trch$, we obtain
$$2\|\phi\|_{L^2(\S)}\cdot \f{d}{d\ub}\|\phi\|_{L^2(\S)}\ls \|\phi\|_{L^2(\S)}\cdot \|\nab_4\phi\|_{L^2(\S)}+\f{O}{|u|}\|\phi\|^2_{L^2(\S)}.$$
This implies
$$\f{d}{d\ub}\|\phi\|_{L^2(\S)}\ls \|\nab_4\phi\|_{L^2(\S)}+\f{O}{|u|}\|\phi\|_{L^2(\S)}.$$
And (\ref{transport1}) can be concluded by applying Gr\"onwall's inequality for $\ub$ variable. 

Inequality (\ref{transport3}) could be proved in {\color{black}a similar fashion.  For a scalar function $f$, we arrive at
\[
\Lb\int_{\S} f=\int_{\S} \left(\Lb f+\Omega \tr\chib f\right)=\int_{\S} \Omega\left(e_3(f)+ \tr\chib f\right).
\]
Taking $f=|\phi|_{\gamma}^2$, using Cauchy-Schwarz inequality on the sphere and the fact $\O>0, \tr\chib<0$, we obtain
$$2\|\phi\|_{L^2(\S)}\cdot \Lb\|\phi\|_{L^2(\S)}\ls \|\phi\|_{L^2(\S)}\cdot \|\nab_3\phi\|_{L^2(\S)}.$$
This implies $ \Lb\|\phi\|_{L^2(\S)}\ls \|\nab_3\phi\|_{L^2(\S)}$ and (\ref{transport3}) follows.
 }
\end{proof}

We then rewrite the above inequalities in scale invariant norms:
\begin{proposition}\label{transport}
For an $\S$ tangent tensor $\phi$ of arbitrary rank, we have
\begin{equation*}
\|\phi\|_{L^2_{sc}(S_{u,\underline{u}})}\ls
\|\phi\|_{L^2_{sc}(S_{u,0})}+\int_0^{\underline{u}} \|\nabla_4\phi\|_{L^2_{sc}(S_{u,\underline{u}'})}d\underline{u}',
\end{equation*}

\begin{equation*}
\|\phi\|_{L^2_{sc}(S_{u,\underline{u}})}\ls
\|\phi\|_{L^2_{sc}(S_{u_{\infty},\underline{u}})}+\int_{u_{\infty}}^{u}\frac{a}{|u'|^2}\|\nabla_3\phi\|_{L^2_{sc}(S_{u',\underline{u}})}du'. 
\end{equation*}
\end{proposition}

For $\nab_3$ equations, sometimes we need more precise estimates to deal with borderline terms. Typically, a borderline term contains $\tr\chib$. Thus, the coefficients in front of $\trchb$ play an important role.
\begin{proposition}\label{el}
We work under the assumptions of Theorem \ref{main.thm1} and bootstrap assumptions \eqref{BA.0}. Let $\phi$ and $F$ be $\S$-tangent tensor fields of rank $k$ satisfying the following transport equation:
\begin{equation*}
\nab_3 \phi_{A_1...A_k}+\lambda_0{\tr\underline{\chi}}\phi_{A_1...A_k}=F_{A_1...A_k}.
\end{equation*}
Denoting $\lambda_1=2(\lambda_0-\frac{1}{2})$, for $\phi$ we have
\begin{equation*}
|u|^{\lambda_1}\|\phi\|_{L^{2}(\S)}\lesssim
|u_{\infty}|^{\lambda_1}\|\phi\|_{L^{2}(S_{u_{\infty},\underline{u}})}+\int_{u_{\infty}}^u|u'|^{\lambda_1}\|F\|_{L^{2}(S_{u',\underline{u}})}du'.
\end{equation*}
\end{proposition}
\begin{proof}
We use variation of area formula for equivariant vector $\Lb$ \footnote{Recall $\Lb=\O e_3$.} and a scalar function $f$: 
\[
\Lb\int_{\S} f=\int_{\S} \left(\Lb f+\Omega \trchb f\right)=\int_{\S} \Omega\left(e_3(f)+ \trchb f\right).
\]
With this identity, 
we obtain
\begin{equation}\label{evolution.id}
\begin{split}
&\Lb(\int_{\S}|u|^{2\lambda_1 }|\phi|^{2})\\
=&\int_{\S}\Omega\l -2\lambda_1 |u|^{2\lambda_1-1}(e_3 u)|\phi|^{2}+2|u|^{2\lambda_1}<\phi,\nab_3\phi>+ \tr\underline{\chi}|u|^{2\lambda_1}|\phi|^{2}\r\\
=&\int_{\S}\Omega\l 2|u|^{2\lambda_1}<\phi, \nab_3\phi+{\lambda_0}\trchb\phi>\r\\
&+\int_{\S}\Omega |u|^{2\lambda_1}\l -\f{2\lambda_1 (e_3u)}{|u|}+(1-2\lambda_0)\trchb\r|\phi|^2.
\end{split}
\end{equation}
Observe that we have
\begin{equation}\label{trchib additional}
\begin{split}
&-\f{2\lambda_1 (e_3u)}{|u|}+(1-2\lambda_0)\trchb\\
= &-\f{2\lambda_1 \Omega^{-1}}{|u|}+(1-2\lambda_0)\trchb\\
= &-\f{2\lambda_1 (\Omega^{-1}-1)}{|u|}+(1-2\lambda_0)(\trchb+\f{2}{|u|})-\f{2\lambda_1+2-4\lambda_0}{|u|}\\
\ls &\f{O}{|u|^2}.
\end{split}
\end{equation}
For the last inequality, we employ (\ref{Omega -1}), bootstrap assumption $\|\tr\chib+\f{2}{|u|}\|_{L^{\infty}(\S)}\leq \f{O}{|u|^2}$ and $\lambda_1=2(\lambda_0-1/2)$.

Using Cauchy-Schwarz for the first term and applying Gr\"onwall's inequality for the second term, we obtain
\begin{equation*}
\begin{split}
&|u|^{\lambda_1}\|\phi\|_{L^2(\S)}\\
\ls &e^{O\|u^{-2}\|_{L^1_u}}\l|u_{\infty}|^{\lambda_1}\|\phi\|_{L^2(S_{u_{\infty},\underline{u}})}+\int_{u_{\infty}}^u |u'|^{\lambda_1}\|F\|_{L^2(S_{u',\underline{u}})}du'\r\\
\ls &|u_{\infty}|^{\lambda_1}\|\phi\|_{L^2(S_{u_{\infty},\underline{u}})}+\int_{u_{\infty}}^u |u'|^{\lambda_1}\|F\|_{L^2(S_{u',\underline{u}})}du'.
\end{split}
\end{equation*}
In the last step, we use $O\|u^{-2}\|_{L^1_u}\ls O/a\leq 1$.
\end{proof}

\subsection{Sobolev Embedding}\label{Embedding}
With the derived estimates for metric $\gamma$, we follow \cite{Chr:book} to obtain a bound on the isoperimetric constant for a $2$-sphere $S$
$$I(S)=\sup_{\substack{U\subset S\\\partial U \in C^1}} \f{\min\{\mbox{Area}(U),\mbox{Area}(U^c)\}}{(\mbox{Perimeter}(\partial U))^2}.$$
\begin{proposition}\label{isoperimetric}
Under the assumptions of Theorem \ref{main.thm1} and the bootstrap assumption \eqref{BA.0}, the isoperimetric constant obeys an upper bound
$$I(\S)\leq \f1{\pi},$$
where $u_{\infty}\leq u\leq -a/4$ and $0\leq \ub \leq 1$.
\end{proposition}
\begin{proof}
Fix $u$. For $U_{\ub}\subset S_{u,\ub}$, we denote $U_0\subset S_{u,0}$ to be the backward image of $U_{\ub}$ under the diffeomorphism generated
by the equivariant vector $L$.
Using Proposition \ref{gamma}, Proposition \ref{gamma2} and their proof, we obtain
$$\f{\mbox{Perimeter}(\partial U_{\ub})}{\mbox{Perimeter}(\partial U_0)}\geq \sqrt{\inf_{S_{u,0}} \lambda(\ub)} $$
and
$$\f{\mbox{Area}(U_{\ub})}{\mbox{Area}(U_0)}\leq \sup_{S_{u,0}} \f{\det(\gamma_{\ub})}{\det(\gamma_0)},\quad\f{\mbox{Area}(U^c_{\ub})}{\mbox{Area}(U^c_0)}\leq \sup_{S_{u,0}} \f{\det(\gamma_{\ub})}{\det(\gamma_0)}.$$
The conclusion then follows from the fact that $I(S_{u,0})=\f{1}{2\pi}$ and the bounds in Proposition \ref{gamma} and Proposition \ref{gamma2}.
\end{proof}
We will use an $L^2-L^\i$ Sobolev embedding inequality in this article. In order to derive it, we will use two propositions quoted directly from \cite{Chr:book}:
\begin{proposition}[\cite{Chr:book}, Lemma 5.1]\label{Lp}
For any Riemannian $2$-manifold $(S,\gamma)$, it holds
$$(\mbox{Area}(S))^{-\f1p}\|\phi\|_{L^p(S)}\leq C_p\sqrt{\max\{I(S),1\}}\bigg(\|\nab\phi\|_{L^2(S)}+\big(\mbox{Area}(S)\big)^{-\f12}\|\phi\|_{L^2(S)}\bigg)$$
for $2<p<\infty$ and for any tensor field $\phi$.
\end{proposition}
\begin{proposition}[\cite{Chr:book}, Lemma 5.2]\label{Linfty}
For any Riemannian $2$-manifold $(S,\gamma)$, we have 
$$\|\phi\|_{L^\infty(S)}\leq C_p\sqrt{\max\{I(S),1\}}(\mbox{Area}(S))^{\f12-\f1p}\bigg(\|\nab\phi\|_{L^p(S)}+\big(\mbox{Area}(S)\big)^{-\f12}\|\phi\|_{L^p(S)}\bigg)$$
for $p>2$ and for any tensor field $\phi$.
\end{proposition}
Note by Proposition \ref{area}, we have $\mbox{Area}(\S)\sim |u|^2$. Combining Propositions \ref{isoperimetric}, \ref{Lp} and \ref{Linfty}, we have 
\begin{proposition}\label{Sobolev}
Under the assumptions of Theorem \ref{main.thm1} and the bootstrap assumption \eqref{BA.0}, it holds
\begin{equation}\label{Sobolev 1}
\begin{split}
\|\phi\|_{L^\infty(S_{u,\ub})} \ls & \sum_{i\leq 2}\|u^{i-1}\nab^i\phi\|_{L^2(\S)}+\|\phi\|_{L^2(\S)}.
\end{split}
\end{equation}
Written in scale invariant norms:
\begin{equation}\label{Sobolev 2}
\|\phi\|_{L_{sc}^{\infty}(S_{u,\ub})}\ls  \sum_{i\leq 2}\|(\at\nab)^i\phi\|_{L_{sc}^{2}(S_{u,\ub})}+\|\phi\|_{L_{sc}^{2}(S_{u,\ub})}.
\end{equation}
\end{proposition}

\subsection{Commutation Formula}\label{commutation}
We move to derive general commutation formulae. We first list the following formula from \cite{KNI:book}:

\begin{proposition}\label{commute0}
For a scalar function $f$, it holds

$$[\nab_4,\nab]f=\frac 12 (\eta+\etb)\nab_4f-\chi\cdot\nab f,$$
$$[\nab_3,\nab]f=\frac 12 (\eta+\etb)\nab_3f-\chib\cdot\nab f.$$
\end{proposition}

\begin{proposition}\label{commute1}
For a 1-form $U_b$ tangent to $S_{u,\ub}$, we have

$$[\nab_4,\nab_a]U_b=-\chi_{ac}\nab_cU_b+\epsilon_{ac}{^*\b_b} U_c+\frac 12(\eta_a+\etb_a)\nab_4U_b-\chi_{ac}\etb_bU_c+\chi_{ab}\etb\cdot U,$$ 
$$[\nab_3,\nab_a]U_b=-\chib_{ac}\nab_cU_b+\epsilon_{ac}{^*\beb_b} U_c+\frac 12(\eta_a+\etb_a)\nab_3U_b-\chib_{ac}\eta_bU_c+\chib_{ab}\eta\cdot U.$$ 
\end{proposition}

\begin{proposition}\label{commute2}
For a 2-form $V_{bc}$ tangent to $\S$, we have
\begin{equation*}
\begin{split}
[\nab_4,\nab_a]V_{bc}=&\frac 12(\eta_a+\etb_a)\nab_4V_{bc}-\etb_bV_{dc}\chi_{ad}-\etb_cV_{bd}\chi_{ad}-\epsilon_{bd}{^*\b_a}V_{dc}-\epsilon_{cd}{^*\b_c}V_{bd}\\
&+\chi_{ac}V_{bd}\etb_d+\chi_{ab}V_{dc}\etb_d-\chi_{ad}\nab_dV_{bc},
\end{split}
\end{equation*}

\begin{equation*}
\begin{split}
[\nab_3,\nab_a]V_{bc}=&\frac 12(\eta_a+\etb_a)\nab_3V_{bc}-\eta_bV_{dc}\chib_{ad}-\eta_cV_{bd}\chib_{ad}+\epsilon_{bd}{^*\beb_a}V_{dc}+\epsilon_{cd}{^*\beb_c}V_{bd}\\
&+\chib_{ac}V_{bd}\eta_d+\chib_{ab}V_{dc}\eta_d-\chib_{ad}\nab_dV_{bc}.
\end{split}
\end{equation*}
\end{proposition}

\begin{remark}
In this article, we require $i_1, i_2,..., i_n\geq 0$. 
\end{remark}

\noindent {\color{black}Applying Proposition \ref{commute1} and Proposition \ref{commute2} through a mathematical induction,
we then give the below general formulas (see also \cite{A-L, L-R:Propagation}):}
\begin{proposition}\label{commute3}
Assume $\nabla_4\phi=F_0$. Let $\nabla_4\nabla^i\phi=F_i$.
Then we have
\begin{equation*}
\begin{split}
F_i= &\sum_{i_1+i_2+i_3=i}\nabla^{i_1}(\eta+\underline{\eta})^{i_2}\nabla^{i_3} F_0+\sum_{i_1+i_2+i_3+i_4=i-1} \nabla^{i_1}(\eta+\underline{\eta})^{i_2}\nabla^{i_3}\beta\nabla^{i_4} \phi\\
&+\sum_{i_1+i_2+i_3+i_4=i}\nabla^{i_1}(\eta+\underline{\eta})^{i_2}\nabla^{i_3}\chi\nabla^{i_4} \phi.
\end{split}
\end{equation*}

Similarly, assume $\nabla_3\phi=G_{0}$. Let $\nabla_3\nabla^i\phi=G_{i}$.
We get
\begin{equation*}
\begin{split}
G_{i}{\color{black}+}\frac{i}{2}\tr\chib \nab^i \phi= &\sum_{i_1+i_2+i_3=i}\nabla^{i_1}(\eta+\underline{\eta})^{i_2}\nabla^{i_3} G_{0}\\
&+\sum_{i_1+i_2+i_3+i_4=i-1} \nabla^{i_1}(\eta+\underline{\eta})^{i_2}\nabla^{i_3}\underline{\beta}\nabla^{i_4} \phi\\
&+\sum_{i_1+i_2+i_3+i_4=i}\nabla^{i_1}(\eta+\underline{\eta})^{i_2}\nabla^{i_3}(\chibh,\tc)\nabla^{i_4} \phi\\
&+\sum_{i_1+i_2+i_3+i_4=i-1}\nabla^{i_1}(\eta+\underline{\eta})^{i_2+1}\nabla^{i_3}\tr\chib\nabla^{i_4} \phi.
\end{split}
\end{equation*}
\end{proposition}
$$$$
Finally, by replacing $\b, \beb$ via Codazzi equations:
$$\beta=-\div \chih+\f12\nab \tr\chi-\f12(\eta-\etab)\cdot(\chih-\f12\tr\chi),$$
$$\beb=\div\chibh-\f12\nab\tr\chib-\f12(\eta-\etab)\cdot(\chibh-\f12\tr\chib),$$
and substituting $\eta, \etb, \tr\chi, \tr\chib+\f{2}{|u|}$ with $\p$, we arrive at 
\begin{proposition}\label{commute}
Suppose $\nabla_4\phi=F_0$. Let $\nabla_4\nabla^i\phi=F_i$.
Then
\begin{equation*}
\begin{split}
F_i= &\sum_{i_1+i_2+i_3=i}\nabla^{i_1}\p^{i_2}\nabla^{i_3} F_0+\sum_{i_1+i_2+i_3+i_4=i}\nabla^{i_1}\p^{i_2}\nabla^{i_3}(\p, \chih)\nabla^{i_4} \phi.\\
\end{split}
\end{equation*}
Similarly, suppose $\nabla_3\phi=G_{0}$. Let $\nabla_3\nabla^i\phi=G_{i}$.
Then
\begin{equation*}
\begin{split}
G_{i}{\color{black}+}\f {i}{2} \tr\chib\nab^i\phi 
=&\sum_{i_1+i_2+i_3=i}\nabla^{i_1}\p^{i_2}\nabla^{i_3} G_{0}\\
&+\sum_{i_1+i_2+i_3+i_4=i}\nabla^{i_1}\p^{i_2}\nabla^{i_3}(\p, \chibh, \tc)\nabla^{i_4} \phi\\
&+\sum_{i_1+i_2+i_3+i_4=i-1} \nabla^{i_1}\p^{i_2+1}\nabla^{i_3}\tr\chib\nab^{i_4}\phi.\\
\end{split}
\end{equation*}
\end{proposition}

\section{$L^2(\S)$ estimates for Ricci coefficients}\label{secRicci}

We start from several useful estimates. {\color{black}Denote} 
\begin{equation}\label{psi Psi 2}
\p\in \{\f{\chih}{\at}, \tr\chi, \o, \eta, \etb, \zeta, \omb, \f{a}{|u|}\tc, \f{\at}{|u|}\chibh, \f{a}{|u|^2}\tr\chib \}, \quad\mbox{ and } \quad \Psi\in \{\f{\a}{\at}, \beta,\rho,\sigma, \beb,\ab\}.
\end{equation}
\begin{proposition} Under the assumption of Theorem \ref{main.thm1} and bootstrap assumption (\ref{BA.0}), we have 

\begin{equation}\label{4.1}
\sum_{\substack{i_1+i_2\leq 9 \\}}\|(\at)^{i_1+i_2}\nab^{i_1}\p^{i_2}\|_{L^{2}_{sc}(S_{u,\ub})}\leq |u|,
\end{equation}

\begin{equation}\label{4.2}
\sum_{\substack{i_1+i_2\leq 9\\}}\|(\at)^{i_1+i_2}\nab^{i_1}\p^{i_2+1}\|_{L^{2}_{sc}(S_{u,\ub})}\leq O,
\end{equation}

\begin{equation}\label{4.3}
\sum_{\substack{i_1+i_2\leq 9\\}} \|(\at)^{i_1+i_2}\nab^{i_1}\p^{i_2+2}\|_{L^{2}_{sc}(S_{u,\ub})}\leq \f{O^2}{|u|},
\end{equation}

\begin{equation}\label{4.4}
\sum_{\substack{i_1+i_2\leq 9\\}} \|(\at)^{i_1+i_2}\nab^{i_1}\p^{i_2+3}\|_{L^{2}_{sc}(S_{u,\ub})}\leq \f{O^3}{|u|^2}, 
\end{equation}

\begin{equation}\label{4.5}
\sum_{\substack{i_1+i_2+i_3\leq 9\\ }} \|(\at)^{i_1+i_2+i_3}\nabla^{i_1}\p^{i_2}\nab^{i_3}\Psi\|_{L^2_{sc}(S_{u,\ub})}\leq O,
\end{equation}

\begin{equation}\label{4.6}
\sum_{i_1+i_2+i_3\leq 9}\|(\at)^{i_1+i_2+i_3+1}\nabla^{i_1}\p^{i_2+1}\nab^{i_3}\Psi\|_{L^2_{sc}(S_{u,\ub})}\leq \f{\at}{|u|}\cdot O^2,
\end{equation}

\begin{equation}\label{4.7}
\sum_{i_1+i_2+i_3\leq 9}\|(\at)^{i_1+i_2+i_3+2}\nabla^{i_1}\p^{i_2+2}\nab^{i_3}\Psi\|_{L^2_{sc}(S_{u,\ub})}\leq \f{a}{|u|^2}\cdot O^3.
\end{equation}
\end{proposition} 

\begin{proof}
We will prove (\ref{4.1}) first. For $i_2=0$, (\ref{4.1}) is true because naturally we could let $s_2(1)=0$ and
$$\|1\|_{L^2_{sc}(\S)}=|u|.$$
For $i_2\geq 1$, we could rewrite $\nab^{i_1}\p^{i_2}$ {\color{black}as} a product of $i_2$-terms
$$\nab^{i_1}\p^{i_2}=\nab^{j_1}\p\cdot\nab^{j_2}\p\cdot\cdot\cdot\nab^{j_{i_2}}\p, \mbox{ where } i_1=j_1+j_2+...+j_{i_2},$$
and assume that $j_{i_2}$ is the largest number. We then rewrite
 $$(\at)^{i_1+i_2}\nab^{i_1}\p^{i_2}=a^{i_2}\cdot(\at\nab)^{j_{i_2}}\p \cdot \Pi_{k=1}^{k=i_2-1}(\at\nab)^{j_k}\p.$$ 
We bound $(\at\nab)^{j_{i_2}}\p$ in $L^2_{sc}(\S)$ norm and bound other terms in $L^{\infty}_{sc}(\S)$ norms. By employing (\ref{Holder's}) for $i_2-1$ times, we obtain
\begin{equation*}
\begin{split}
&\f{1}{|u|}\cdot\sum_{i_1+i_2\leq 9}\|(\at)^{i_1+i_2}\nab^{i_1}\p^{i_2}\|_{L^{2}_{sc}(S_{u,\ub})}\\
\leq&\f{1}{|u|}\cdot \sum_{i_1+i_2\leq 9}  \f{(\at)^{i_2}}{|u|^{i_2-1}} \|(\at\nab)^{j_{i_2}}\p\|_{L^{2}_{sc}(S_{u,\ub})} \Pi_{k=1}^{k={i_2-1}}\|(\at\nab)^{j_k}\p\|_{L^{\infty}_{sc}(S_{u,\ub})},\\
\leq& \f{(\at)^{i_2}\cdot O^{i_2}}{|u|^{i_2}}\leq 1.
\end{split}
\end{equation*}
We prove (\ref{4.2}) in the same fashion. If $i_2=0$, (\ref{4.2}) is true according to the definition of $O$. For $i_2\geq 1$, assume $i_1=j_1+j_2+...+j_{i_2+1}$. And assume $j_{i_2+1}$ is the largest. It follows
\begin{equation*}
\begin{split}
&\sum_{i_1+i_2\leq 9} \|(\at)^{i_1+i_2}\nab^{i_1}\p^{i_2+1}\|_{L^{2}_{sc}(S_{u,\ub})}\\
\leq& \sum_{i_1+i_2\leq 9} \f{(\at)^{i_2}}{|u|^{i_2}} \|(\at\nab)^{j_{i_2+1}}\p\|_{L^{2}_{sc}(S_{u,\ub})} \Pi_{k=1}^{k={i_2}}\|(\at\nab)^{j_k}\p\|_{L^{\infty}_{sc}(S_{u,\ub})},\\
\leq& \f{(\at)^{i_2}\cdot O^{i_2+1}}{|u|^{i_2}}\leq O.
\end{split}
\end{equation*}
For (\ref{4.3}), we use (\ref{Holder's}), (\ref{4.2}) and derive 
\begin{equation*}
\begin{split}
&|u|\cdot\sum_{i_1+i_2\leq 9}\|(\at)^{i_1+i_2}\nab^{i_1}\p^{i_2+2}\|_{L^{2}_{sc}(S_{u,\ub})}\\
\leq&|u|\cdot\f{1}{|u|}\cdot \sum_{i_1+i_2\leq 9}  \|(\at\nab)^{i_3}\p\|_{L^{\infty}_{sc}(S_{u,\ub})}\|(\at)^{i_2+i_4}\nab^{i_4}\p^{i_2+1}\|_{L^{2}_{sc}(S_{u,\ub})},\, \, \mbox{where} \, \, i_3+i_4=i_1\\
\leq&O\cdot O=O^2.
\end{split}
\end{equation*}
With the same approach, for (\ref{4.4}), we use (\ref{Holder's}), (\ref{4.3}) and obtain
\begin{equation*}
\begin{split}
&|u|^2\cdot\sum_{i_1+i_2\leq 9}\|(\at)^{i_1+i_2}\nab^{i_1}\p^{i_2+3}\|_{L^{2}_{sc}(S_{u,\ub})}\\
\leq&|u|^2\cdot\f{1}{|u|}\cdot \sum_{i_1+i_2\leq 9}  \|(\at\nab)^{i_3}\p\|_{L^{\infty}_{sc}(S_{u,\ub})}\|(\at)^{i_2+i_4}\nab^{i_4}\p^{i_2+2}\|_{L^{2}_{sc}(S_{u,\ub})},\, \, \mbox{where} \, \, i_3+i_4=i_1\\
\leq&|u|\cdot O\cdot \f{O^2}{|u|}\leq O^3.
\end{split}
\end{equation*}
We then move to (\ref{4.5}). If $i_3\leq 7$, we bound $(\at\nab)^{i_3}\Psi$ with $L^{\infty}_{sc}(\S)$ norm; otherwise we bound $(\at\nab)^{i_3}\Psi$ with $L^{2}_{sc}(\S)$ norm. As before assume $i_1= j_1+j_2+...+j_{i_2}$. With (\ref{Holder's}) and (\ref{4.1}) we obtain
\begin{equation*}
\begin{split}
&\sum_{i_1+i_2+i_3\leq 9} \|(\at)^{i_1+i_2+i_3}\nabla^{i_1}\p^{i_2}\nab^{i_3}\Psi\|_{L^2_{sc}(S_{u,\ub})}\\
\leq&\f{1}{|u|}\cdot \sum_{i_1+i_2+i_3\leq 9}    \|(\at)^{i_1+i_2}\nab^{i_1}\p^{i_2}\|_{L^{2}_{sc}(S_{u,\ub})}\|(\at\nab)^{i_3}\Psi\|_{L^{\infty}_{sc}(S_{u,\ub})}\\
&+\sum_{i_1+i_2+i_3\leq 9}\f{(\at)^{i_2}}{|u|^{i_2}} \Pi_{k=1}^{k=i_2}\|(\at\nab)^{j_k}\p\|_{L^{\infty}_{sc}(S_{u,\ub})}\cdot  \|(\at\nab)^{i_3}\Psi\|_{L^{2}_{sc}(S_{u,\ub})},\\
\leq& O.
\end{split}
\end{equation*}
Similarly, for (\ref{4.6}) we decompose $i_1=j_1+j_2+...+j_{i_2}$ and derive
\begin{equation*}
\begin{split}
&\sum_{i_1+i_2+i_3\leq 9} \|(\at)^{i_1+i_2+i_3+1}\nabla^{i_1}\p^{i_2+1}\nab^{i_3}\Psi\|_{L^2_{sc}(S_{u,\ub})}\\
\leq&\f{1}{|u|}\cdot\sum_{i_1+i_2+i_3\leq 9}\|(\at)^{i_1+i_2+1}\nab^{i_1}\p^{i_2+1}\|_{L^{2}_{sc}(S_{u,\ub})}\|(\at\nab)^{i_3}\Psi\|_{L^{\infty}_{sc}(S_{u,\ub})}\\
&+\f{(\at)^{i_2+1}}{|u|^{i_2+1}}\sum_{i_1+i_2+i_3\leq 9}  \Pi_{k=1}^{k=i_2+1}\|(\at\nab)^{j_k}\p\|_{L^{\infty}_{sc}(S_{u,\ub})}\cdot  \|(\at\nab)^{i_3}\Psi\|_{L^{2}_{sc}(S_{u,\ub})}, \\
\leq&\f{\at}{|u|}\cdot O^2.
\end{split}
\end{equation*}
Finally, we prove (\ref{4.7}). We use (\ref{Holder's}) once and reduce it to (\ref{4.6}). 

\begin{equation}
\begin{split}
&\sum_{i_1+i_2+i_3\leq 9} \|(\at)^{i_1+i_2+i_3+2}\nabla^{i_1}\p^{i_2+2}\nab^{i_3}\Psi\|_{L^2_{sc}(S_{u,\ub})}\\
\leq&\sum_{\substack{i_1+i_2+i_3\leq 9\\i_4+i_5=i_1}}\f{\at}{|u|}\|(\at\nab)^{i_4}\p\|_{L^{\infty}_{sc}(S_{u,\ub})}\cdot\|(\at)^{i_2+i_3+i_5+1}\nabla^{i_5}\p^{i_2+1}\nab^{i_3}\Psi\|_{L^2_{sc}(S_{u,\ub})}\\
\leq&\f{\at}{|u|}\cdot O\cdot \f{\at}{|u|}\cdot O^2=\f{a}{|u|^2}\cdot O^3.\\
\end{split}
\end{equation}
\end{proof}

We are now ready to estimate Ricci coefficients and we start from $\o$
\begin{proposition}\label{o.bd}
Under the assumptions of Theorem \ref{main.thm1} and the bootstrap assumptions \eqref{BA.0}, we have
\[
 \sum_{i\leq 10}\|(\at\nab)^i\o\|_{L^2_{sc}(S_{u,\ub})} \ls \f{\at}{|u|^{\f12}}\big(\underline{\M R}[\rho]+1\big).
\]
\end{proposition}
\begin{proof}
We use the following schematic null structure equation for $\o$:
$$\nab_3\o=\f12\rho+\p\p.$$ 
Commuting it with angular derivative for $i$ times, we have
\begin{equation*}
\begin{split}
&\nab_3 \nab^i\o +\frac i2 \trchb\nab^i\o\\
= &\nab^i\rho+\sum_{i_1+i_2+i_3+1=i}\nab^{i_1}\p^{i_2+1}\nab^{i_3}\rho+\sum_{i_1+i_2+i_3+i_4=i}\nabla^{i_1}\p^{i_2}\nabla^{i_3}(\p, \chibh, \tc)\nabla^{i_4} \p\\
&+\sum_{i_1+i_2+i_3+i_4=i-1}\nabla^{i_1}\p^{i_2+1}\nabla^{i_3}\tr\chib\nabla^{i_4} \p.
\end{split}
\end{equation*}
Denote the above equality as 
$$\nab_3 \nab^i\o +\frac i2 \trchb\nab^i\o=G.$$
Applying Proposition \ref{el}, it holds
\begin{equation*}
\begin{split}
|u|^{i-1}\|\nab^i\o\|_{L^2(\S)}\leq&|u_{\infty}|^{i-1}\|\nab^i\o\|_{L^2(S_{u_{\infty},\ub})}+\int_{u_{\infty}}^u|u'|^{i-1}\|G\|_{L^{2}(S_{u',\underline{u}})}du'.
\end{split}
\end{equation*}
Times $|u|$ on both sides and using $|u|\leq |u'|, \, |u|\leq |u_{\infty}|$ we have
\begin{equation}\label{o G}
\begin{split}
|u|^{i}\|\nab^i\o\|_{L^2(\S)}\leq&|u_{\infty}|^{i}\|\nab^i\o\|_{L^2(S_{u_{\infty},\ub})}+\int_{u_{\infty}}^u|u'|^{i}\|G\|_{L^{2}(S_{u',\underline{u}})}du'.
\end{split}
\end{equation}
From signature table and property (\ref{signature of derivative}), we have
$$s_2(\nab^i\o)=s_2(\o)+i\cdot\f{1}{2}=0+\f{i}{2}=\f{i}{2}.$$ 
By conversation of signatures in each equation and property (\ref{signature of derivative}), it holds
$$s_2(G)=s_2(\nab_3 \nab^i\o)=s_2(\nab^i \o)+1=\f{i}{2}+1.$$
Using the definition of $L^2_{sc}(\S)$ norms
$$\|\phi\|_{L_{sc}^{2}(\S)}:=a^{-s_2(\phi)}|u|^{2s_2(\phi)}\|\phi\|_{L^{2}(\S)},$$
we have
$$\|\nab^i\o\|_{L^2_{sc}(\S)}=a^{-\f{i}{2}}|u|^i\|\nab^i\o\|_{L^2(\S)}, \quad \|G\|_{L^2_{sc}(\S)}=a^{-\f{i}{2}-1}|u|^{i+2}\|G\|_{L^2(\S)}.$$
That is equivalent to
$$|u|^i\|\nab^i\o\|_{L^2(\S)}=\|(\at\nab)^i\o\|_{L^2_{sc}(\S)}, \quad |u|^{i}\|G\|_{L^2(\S)}= \f{a}{|u|^2}\|(\at)^iG\|_{L^2_{sc}(\S)}.$$
We then rewrite (\ref{o G}) in $L^2_{sc}(\S)$ norms 
\begin{equation*}
\begin{split}
\|(\at\nab)^i\o\|_{L^2_{sc}(\S)}\leq&\|(\at\nab)^i\o\|_{L^2_{sc}(S_{u_{\infty},\ub})}+\int_{u_{\infty}}^{u}\f{a}{|u'|^2}\|(\at\nab)^{i}\rho\|_{L^2_{sc}(S_{u',\ub})}du'\\
&+\int_{u_{\infty}}^{u}\f{a}{|u'|^2}\|\sum_{i_1+i_2+i_3+1=i}(\at)^i\nabla^{i_1}\p^{i_2+1}\nab^{i_3}\rho\|_{L^2_{sc}(S_{u',\ub})}du'\\
&+\int_{u_{\infty}}^{u}\f{a}{|u'|^2}\|\sum_{i_1+i_2+i_3+i_4=i}(\at)^i\nabla^{i_1}\p^{i_2}\nabla^{i_3}(\p, \chibh, \tc)\nabla^{i_4} \p\|_{L^2_{sc}(S_{u',\ub})}du'\\
&+\int_{u_{\infty}}^{u}\f{a}{|u'|^2}\|\sum_{i_1+i_2+i_3+i_4=i-1}(\at)^i\nabla^{i_1}\p^{i_2+1}\nabla^{i_3}\tr\chib\nabla^{i_4} \p\|_{L^2_{sc}(S_{u',\ub})}du'.
\end{split}
\end{equation*}
For the first term, since we prescribe $\O|_{u=u_{\infty}}=1$, note by 
$$\o=-\f12\nab_4(\log\O), \mbox{ we have } \|(\at\nab)^i\o\|_{L^2_{sc}(S_{u_{\infty},\ub})}=0.$$
For the two terms involving $\rho$, we have
\begin{equation*}
\begin{split}
&\int_{u_{\infty}}^{u}\f{a}{|u'|^2}\|(\at\nab)^{i}\rho\|_{L^2_{sc}(S_{u',\ub})}du'+\int_{u_{\infty}}^{u}\f{a}{|u'|^2}\|\sum_{i_1+i_2+i_3+1=i}(\at)^i\nabla^{i_1}\p^{i_2+1}\nab^{i_3}\rho\|_{L^2_{sc}(S_{u',\ub})}du'\\
\leq&\bigg(\int_{u_{\infty}}^{u}\f{a}{|u'|^2}\|(\at\nab)^{i}\rho\|^2_{L^2_{sc}(S_{u',\ub})}du'\bigg)^{\f12} \bigg(\int_{\ui}^u\f{a}{|u'|^2} du'\bigg)^{\f12}+\int_{\ui}^u \f{a}{|u'|^2}\cdot \f{\at}{|u'|}\cdot O^2\, du'\\
=&\|(\at\nab)^{i}\rho\|_{L^2_{sc}(\Hb_{\ub}^{(\ui,u)})}  \cdot \f{\at}{|u|^{\f12}}+\f{a^{\f32}}{|u|^2}O^{2}\\\leq& \underline{\M R}[\rho]\cdot \f{\at}{|u|^{\f12}}+\f{a^{\f32}}{|u|^2}O^{2}\leq \f{\at}{|u|^{\f12}}\big(\underline{\M R}[\rho]+1\big).
\end{split}
\end{equation*}
where H\"older's inequality and (\ref{4.6}) are used in the first inequality; the definition in (\ref{scale invariant norms 2}) is used in the identity; (\ref{R i Hb}) is employed in the second inequality.

For the last two terms, we have
\begin{equation*}
\begin{split}
&\int_{u_{\infty}}^{u}\f{a}{|u'|^2} \|\sum_{i_1+i_2+i_3+i_4=i} (\at)^i \nabla^{i_1}\p^{i_2}\nabla^{i_3}(\p,\chibh, \tc)\nab^{i_4}\p\|_{L^2_{sc}(S_{u',\ub})}du'\\
\leq&\int_{u_{\infty}}^{u}\f{\at}{|u'|}\|\sum_{i_1+i_2+i_3+i_4=i} (\at)^i \nabla^{i_1}\p^{i_2}\nabla^{i_3}(\f{\at}{|u'|}\p, \f{\at}{|u'|}\chibh, \f{\at}{|u'|}\tc)\nab^{i_4}\p\|_{L^2_{sc}(S_{u',\ub})}du'\\
\leq&\int_{u_{\infty}}^{u}\f{\at}{|u'|}\cdot \f{O^2}{|u'|} du' \leq \f{\at}{|u|} O^{2} \leq \f{\at}{|u|^{\f12}},
\end{split}
\end{equation*}
where we use $\at/|u|\leq 1/\at$ and (\ref{4.3}) in the second inequality.  And
\begin{equation*}
\begin{split}
&\int_{u_{\infty}}^{u}\f{a}{|u'|^2}\|\sum_{i_1+i_2+i_3+i_4=i-1}(\at)^i\nabla^{i_1}\p^{i_2+1}\nabla^{i_3}\tr\chib\nabla^{i_4} \p\|_{L^2_{sc}(S_{u',\ub})}du'\\
\leq&\int_{u_{\infty}}^{u}\at\|\sum_{i_1+i_2+i_3+i_4=i-1}(\at)^{i-1}\nabla^{i_1}\p^{i_2+1}\nabla^{i_3}(\f{a}{|u'|^2}\tr\chib)\nabla^{i_4} \p\|_{L^2_{sc}(S_{u',\ub})}du'\\
\leq& \int_{u_{\infty}}^{u} \at\cdot\f{O^3}{|u'|^2}du' \leq \f{\at}{|u|}\cdot O^3 \leq \f{\at}{|u|^{\f12}}.\\
\end{split}
\end{equation*}
Here we appeal to (\ref{4.4}) for the second inequality. Gather all the estimates and let $a$ to be sufficient large, we then derive
\[
 \sum_{i\leq 10}\|(\at\nab)^i\o\|_{L^2_{sc}(S_{u,\ub})} \ls \f{\at}{|u|^{\f12}}\big(\underline{\M R}[\rho]+1\big).
\]
\end{proof}

We then move to estimate $\chibh$. 
\begin{proposition}\label{chibh.bd}
Under the assumptions of Theorem \ref{main.thm1} and the bootstrap assumptions \eqref{BA.0}, we have
\[
 \sum_{i\leq 10}\f{\at}{|u|}\|(\at\nab)^i\chibh\|_{L^2_{sc}(S_{u,\ub})} \ls 1.
\]
\end{proposition}

\begin{proof}
We use the null structure equation
$$\nab_3\chibh+\tr\chib\,\chibh=\ab+\p\cdot\chibh.$$ 
Commuting this equation with $i$ angular derivatives, by Proposition \ref{commute} we have

\begin{equation*}
\begin{split}
&\nab_3 \nab^i\chibh +\frac{i+2}{2}\tr\chib\nab^i\chibh\\
= &\nab^{i}\ab+\sum_{i_1+i_2+i_3=i-1}\nab^{i_1}\p^{i_2+1}\nab^{i_3}\ab+
\sum_{i_1+i_2+i_3+i_4=i}\nabla^{i_1}\p^{i_2}\nabla^{i_3}(\p, \chibh, \tc)\nabla^{i_4} \chibh\\
&+\sum_{i_1+i_2+i_3+i_4=i-1}\nabla^{i_1}\p^{i_2+1}\nabla^{i_3}\tr\chib\nabla^{i_4} \chibh.\\
\end{split}
\end{equation*}
Rewrite the above equation as
$$\nab_3 \nab^i\chibh +\frac{i+2}{2}\tr\chib\nab^i\chibh=F.$$
Applying  Proposition \ref{el}, we have
\begin{equation}\label{chibh F}
\begin{split}
|u|^{i+1}\|\nab^i\chibh\|_{L^2(\S)}\leq&|u_{\infty}|^{i+1}\|\nab^i\chibh\|_{L^2(S_{u_{\infty},\ub})}+\int_{u_{\infty}}^u|u'|^{i+1}\|F\|_{L^{2}(S_{u',\underline{u}})}du'.
\end{split}
\end{equation}
By signature consideration, we have
$$s_2(\nab^i\chibh)=s_2(\chibh)+i\cdot\f{1}{2}=\f{i}{2}+1, \quad s_2(F)=s_2(\nab_3\nab^i\chib)=\f{i}{2}+2.$$ 
Using the definition of $L^2_{sc}(\S)$ norms
$$\|\phi\|_{L_{sc}^{2}(\S)}:=a^{-s_2(\phi)}|u|^{2s_2(\phi)}\|\phi\|_{L^{2}(\S)},$$
we have
$$\|\nab^i\chibh\|_{L^2_{sc}(\S)}=a^{-\f{i}{2}-1}|u|^{i+2}\|\nab^i\chih\|_{L^2(\S)}, \quad \|F\|_{L^2_{sc}(\S)}=a^{-\f{i}{2}-2}|u|^{i+4}\|F\|_{L^2(\S)},$$
which are equivalent to
$$|u|^{i+1}\|\nab^i\chih\|_{L^2(\S)}=\f{a}{|u|}\|(\at\nab)^i\chibh\|_{L^2_{sc}(\S)}, \quad |u|^{i+1}\|F\|_{L^2(\S)}=\f{a^2}{|u|^3}\|(\at)^iF\|_{L^2_{sc}(\S)}.$$
Rewrite (\ref{chibh F}) in $L^2_{sc}(\S)$ norms, it follows
$$\f{a}{|u|}\|(\at\nab)^i\chibh\|_{L^2_{sc}(\S)}\leq \f{a}{|\ui|}\|(\at\nab)^i\chibh\|_{L^2_{sc}(S_{u_{\infty},\ub})}+\int_{u_{\infty}}^u\f{a^2}{|u'|^3}\|(\at)^i F\|_{L^{2}_{sc}(S_{u',\underline{u}})}du'.$$
Multiplying $a^{-\f12}$ on both sides, with the expression of $F$ we have  
\begin{equation*}
\begin{split}
\f{\at}{|u|}\|(\at\nab)^i\chibh&\|_{L^2_{sc}(\S)}\leq\f{\at}{|\ui|}\|(\at\nab)^i\chibh\|_{L^2_{sc}(S_{u_{\infty},\ub})}+\int_{u_{\infty}}^{u}\f{a^{\f32}}{|u'|^3}\|(\at\nab)^{i}\ab\|_{L^2_{sc}(S_{u',\ub})}du'\\
&+\int_{u_{\infty}}^{u}\f{a^{\f32}}{|u'|^3}\|\sum_{i_1+i_2+i_3=i-1}(\at)^i\nabla^{i_1}\p^{i_2+1}\nabla^{i_3}\ab\|_{L^2_{sc}(S_{u',\ub})}du'\\
&+\int_{u_{\infty}}^{u}\f{a^{\f32}}{|u'|^3}\|\sum_{i_1+i_2+i_3+i_4=i}(\at)^i\nabla^{i_1}\p^{i_2}\nabla^{i_3}(\p, \chibh, \tc)\nabla^{i_4}  \chibh\|_{L^2_{sc}(S_{u',\ub})}du'\\
&+\int_{u_{\infty}}^{u}\f{a^{\f32}}{|u'|^3}\|\sum_{i_1+i_2+i_3+i_4=i-1}(\at)^i\nabla^{i_1}\p^{i_2+1}\nabla^{i_3}\tr\chib\nabla^{i_4} \chibh\|_{L^2_{sc}(S_{u',\ub})}du'.
\end{split}
\end{equation*}
For initial data, we have
$$\f{\at}{|u_{\infty}|}\|(\at\nab)^i\chibh\|_{L^2_{sc}(S_{u_{\infty},\ub})}\leq \M I^{(0)}(\ub) \lesssim 1. $$
For $\ab$ terms, we have
\begin{equation*}
\begin{split}
\int_{u_{\infty}}^{u}\f{a^{\f32}}{|u'|^3}\|(\at\nab)^{i}\ab\|_{L^2_{sc}(S_{u',\ub})}du'\leq& \bigg(\int_{u_{\infty}}^{u}\f{a}{|u'|^2}\|(\at\nab)^{i}\ab\|^2_{L^2_{sc}(S_{u',\ub})}du'  \bigg)^{\f12} \bigg(\int_{\ui}^{u}\f{a^2}{|u'|^4} du'\bigg)^{\f12}\\
\leq&\|(\at\nab)^i\ab\|_{\shb}\cdot \f{a}{|u|^{\f32}}\leq \f{\Rb}{\at}\leq 1.
\end{split}
\end{equation*}
And by (\ref{4.6})
\begin{equation*}
\begin{split}
&\int_{u_{\infty}}^{u}\f{a^{\f32}}{|u'|^3}\|\sum_{i_1+i_2+i_3=i-1}(\at)^i\nabla^{i_1}\p^{i_2+1}\nabla^{i_3}\ab\|_{L^2_{sc}(S_{u',\ub})}du'\\
\leq&\int_{u_{\infty}}^{u}\f{a^{\f32}}{|u'|^3}\cdot \f{\at}{|u'|}\cdot O^2 \,du'\\
\leq&\f{a^2}{|u|^3}\cdot O^2\leq \f{O^2}{a}\leq 1.
\end{split}
\end{equation*}
We then move to
\begin{equation*}
\begin{split}
&\int_{u_{\infty}}^{u}\f{a^{\f32}}{|u'|^3} \|\sum_{i_1+i_2+i_3+i_4=i} (\at)^i \nabla^{i_1}\p^{i_2}\nabla^{i_3}(\p,\chibh, \tc)\nab^{i_4}\chibh\|_{L^2_{sc}(S_{u',\ub})}du'\\
\leq&\int_{u_{\infty}}^{u}\f{a^{\f12}}{|u'|}\|\sum_{i_1+i_2+i_3+i_4=i} (\at)^i \nabla^{i_1}\p^{i_2}\nabla^{i_3}(\f{\at}{|u'|}\p, \f{\at}{|u'|}\chibh, \f{\at}{|u'|}\tc)\nab^{i_4}(\f{\at}{|u'|}\chibh)\|_{L^2_{sc}(S_{u',\ub})}du'\\
\leq&\int_{u_{\infty}}^{u}\f{\at}{|u'|}\cdot \f{O^2}{|u'|} du' \leq \f{\at}{|u|} O^{2} \leq \f{O^2}{\at}\leq 1,
\end{split}
\end{equation*}
where we use (\ref{4.3}) in the second inequality.  

\noindent We then deal with the last term 
\begin{equation*}
\begin{split}
&\int_{u_{\infty}}^{u}\f{a^{\f32}}{|u'|^3}\|\sum_{i_1+i_2+i_3+i_4=i-1}(\at)^i\nabla^{i_1}\p^{i_2+1}\nabla^{i_3}\tr\chib\nabla^{i_4} \chibh\|_{L^2_{sc}(S_{u',\ub})}du'\\
\leq&\int_{u_{\infty}}^{u} \at\|\sum_{i_1+i_2+i_3+i_4=i-1}(\at)^{i-1}\nabla^{i_1}\p^{i_2+1}\nabla^{i_3}(\f{a}{|u'|^2}\tr\chib)\nabla^{i_4} (\f{\at}{|u'|}\chibh)\|_{L^2_{sc}(S_{u',\ub})}du'\\
\leq& \int_{u_{\infty}}^{u} \at\cdot\f{O^3}{|u'|^2}du' \leq \f{\at}{|u|}\cdot O^3 \leq \f{O^3}{\at}\leq 1.\\
\end{split}
\end{equation*}
Here we appeal to (\ref{4.4}) for the second inequality. Gathering all the estimates, and letting $a$ to be sufficiently large we have obtained
\[
 \sum_{i\leq 10}\f{\at}{|u|}\|(\at\nab)^i\chibh\|_{L^2_{sc}(S_{u,\ub})} \ls 1.
\]
\end{proof}

Next, we deal with $\chih$. 
\begin{proposition}\label{chih.bd}
Under the assumptions of Theorem \ref{main.thm1} and the bootstrap assumptions \eqref{BA.0}, we have
$$\sum_{i\leq 10} \f{1}{\at}\|(\at\nab)^i \chih \|_{L^2_{sc}(\S)}\ls \M R[\a]+1.$$
\end{proposition}
\begin{proof}
We employ the null structure equation
$$\nab_4 \chih=\a+\p\cdot\chih.$$
Commuting this equation with $i$ angular derivatives, by Proposition \ref{commute} we have
\begin{equation*}
\begin{split}
\nab_4 \nab^i\chih=& \nab^{i}\a+\sum_{i_1+i_2+i_3=i-1}\nab^{i_1}\p^{i_2+1}\nab^{i_3}\a+\sum_{i_1+i_2+i_3+i_4=i}\nabla^{i_1}\p^{i_2}\nabla^{i_3}(\p, \chih)\nabla^{i_4} \chih.\\
\end{split}
\end{equation*}
Applying Proposition \ref{transport} and multiplying $(\at)^{i-1}$ on both sides of equation, we have
\begin{equation*}
\begin{split}
&\f{1}{\at}\|(\at\nab)^{i}\chih\|_{L^2_{sc}(\S)}\\
\leq &\f{1}{\at}\int_0^{\ub}\|(\at\nab)^i \a\|_{L^2_{sc}(S_{u,\ub'})}d\ub'+\sum_{i_1+i_2+i_3=i-1}\int_0^{\ub} \f{1}{\at}\|(\at)^{i}\nabla^{i_1}\p^{i_2+1}\nabla^{i_3} \a\|
_{L^{2}_{sc}(S_{u,\ub'})}d\ub'\\
&+\sum_{i_1+i_2+i_3+i_4=i}\int_0^{\ub} \f{1}{\at}\|(\at)^{i}\nab^{i_1}\p^{i_2}\nab^{i_3}(\p, \chih)\nab^{i_4}\chih\|
_{L^{2}_{sc}(S_{u,\ub'})}d\ub'\\
\leq& \f{1}{\at}\bigg(\int_0^{\ub}\|(\at\nab)^i \a\|^2_{L^2_{sc}(S_{u,\ub'})}d\ub'\bigg)^{\f12}\bigg(\int_0^{\ub}1 \, d\ub'\bigg)^{\f12}\\
&+\sum_{i_1+i_2+i_3=i-1}\int_0^{\ub} \|(\at)^{i}\nabla^{i_1}\p^{i_2+1}\nabla^{i_3} (\f{\a}{\at})\|
_{L^{2}_{sc}(S_{u,\ub'})}d\ub'\\
&+\sum_{i_1+i_2+i_3+i_4=i}\int_0^{\ub} \at\|(\at)^{i}\nab^{i_1}\p^{i_2}\nab^{i_3}(\f{\p}{\at}, \f{\chih}{\at})\nab^{i_4}(\f{\chih}{\at})\|
_{L^{2}_{sc}(S_{u,\ub'})}d\ub'\\
\leq& \f{1}{\at}\|(\at\nab)^i\a\|_{\sh}+\f{\at}{|u|}\cdot O^2+\at\cdot \f{O^2}{|u|}\\
\leq& \M R[\a]+\f{O^2}{\at}\leq \M R[\a]+1,
\end{split}
\end{equation*}
where we use (\ref{4.6}) and (\ref{4.3}) in the third inequality. \\
\end{proof}
In the same fashion, we derive estimate for $\omb$.
\begin{proposition}\label{omb.bd}
Under the assumptions of Theorem \ref{main.thm1} and the bootstrap assumptions \eqref{BA.0}, we have
$$\sum_{i\leq 10} \|(\at\nab)^i \omb\|_{L^2_{sc}(\S)}\ls \M R[\rho]+1.$$
\end{proposition}
\begin{proof}
We have the schematic null structure equation  
$$\nab_4\omb=\rho+\p\cdot\p.$$
Commuting this equation with $i$ angular derivatives, by Proposition \ref{commute} we have
\begin{equation*}
\begin{split}
\nab_4 \nab^i\omb=& \nab^{i}\rho+\sum_{i_1+i_2+i_3=i-1}\nab^{i_1}\p^{i_2+1}\nab^{i_3}\rho+\sum_{i_1+i_2+i_3+i_4=i}\nabla^{i_1}\p^{i_2}\nabla^{i_3}(\p, \chih)\nabla^{i_4} \p\\
\end{split}
\end{equation*}
Applying Proposition \ref{transport} and multiplying $(\at)^{i}$ on both sides of equation, we have
\begin{equation*}
\begin{split}
&\|(\at\nab)^{i}\omb\|_{L^2_{sc}(\S)}\\
\leq &\int_0^{\ub}\|(\at\nab)^i \rho\|_{L^2_{sc}(S_{u,\ub'})}d\ub'+\sum_{i_1+i_2+i_3=i-1}\int_0^{\ub}\|(\at)^{i}\nabla^{i_1}\p^{i_2+1}\nabla^{i_3} \rho\|
_{L^{2}_{sc}(S_{u,\ub'})}d\ub'\\
&+\sum_{i_1+i_2+i_3+i_4=i}\int_0^{\ub}\|(\at)^{i}\nab^{i_1}\p^{i_2}\nab^{i_3}(\p, \chih)\nab^{i_4}\p\|_{L^{2}_{sc}(S_{u,\ub'})}d\ub'\\
\leq&\bigg(\int_0^{\ub}\|(\at\nab)^i \rho\|^2_{L^2_{sc}(S_{u,\ub'})}d\ub'\bigg)^{\f12}\bigg(\int_0^{\ub}1 \, d\ub'\bigg)^{\f12}\\
&+\sum_{i_1+i_2+i_3=i-1}\int_0^{\ub} \|(\at)^{i}\nabla^{i_1}\p^{i_2+1}\nabla^{i_3} \rho\|
_{L^{2}_{sc}(S_{u,\ub'})}d\ub'\\
&+\sum_{i_1+i_2+i_3+i_4=i}\int_0^{\ub} \at\|(\at)^{i}\nab^{i_1}\p^{i_2}\nab^{i_3}(\f{\p}{\at}, \f{\chih}{\at})\nab^{i_4}\p\|_{L^{2}_{sc}(S_{u,\ub'})}d\ub'\\
\leq& \|(\at\nab)^i\rho\|_{\sh}+\f{\at}{|u|}\cdot O^2+\at\cdot\f{O^2}{|u|}\leq \M R[\rho]+1.\\
\end{split}
\end{equation*}
We use (\ref{4.6}) and (\ref{4.3}) in the third inequality. \\
\end{proof}

Similarly, for $\eta$ we have
\begin{proposition}\label{eta.bd}
Under the assumptions of Theorem \ref{main.thm1} and the bootstrap assumptions \eqref{BA.0}, we have
$$\sum_{i\leq 10} \|(\at\nab)^i \eta\|_{L^2_{sc}(\S)}\ls \M R[\b]+1.$$
\end{proposition}
\begin{proof}
We have the schematic null structure equation  
$$\nab_4\eta=\b+\p\cdot\chih.$$
Commuting this equation with $i$ angular derivatives, by Proposition \ref{commute} we have
\begin{equation*}
\begin{split}
\nab_4 \nab^i\eta=& \nab^{i}\b+\sum_{i_1+i_2+i_3=i-1}\nab^{i_1}\p^{i_2+1}\nab^{i_3}\b+\sum_{i_1+i_2+i_3+i_4=i}\nabla^{i_1}\p^{i_2}\nabla^{i_3}\p\nabla^{i_4} \chih\\
\end{split}
\end{equation*}
Applying Proposition \ref{transport} and multiplying $(\at)^{i}$ on both sides of equation, we have
\begin{equation*}
\begin{split}
&\|(\at\nab)^{i}\eta\|_{L^2_{sc}(\S)}\\
\leq &\int_0^{\ub}\|(\at\nab)^i \b\|_{L^2_{sc}(S_{u,\ub'})}d\ub'+\sum_{i_1+i_2+i_3=i-1}\int_0^{\ub}\|(\at)^{i}\nabla^{i_1}\p^{i_2+1}\nabla^{i_3} \b\|
_{L^{2}_{sc}(S_{u,\ub'})}d\ub'\\
&+\sum_{i_1+i_2+i_3+i_4=i}\int_0^{\ub}\at\|(\at)^{i}\nab^{i_1}\p^{i_2}\nab^{i_3}\p\nab^{i_4}(\f{\p}{\at}, \f{\chih}{\at})\|_{L^{2}_{sc}(S_{u,\ub'})}d\ub'\\
\leq& \|(\at\nab)^i\b\|_{\sh}+\f{\at}{|u|}\cdot O^2+\at\cdot\f{O^2}{|u|}\leq \M R[\b]+1.\\
\end{split}
\end{equation*}
In the third inequality, (\ref{4.6}) and (\ref{4.3}) are used. \\
\end{proof}

We move to estimate $\tr\chi$
\begin{proposition}\label{trchi.bd}
Under the assumptions of Theorem \ref{main.thm1} and the bootstrap assumptions \eqref{BA.0}, we have
$$\sum_{i\leq 10} \|(\at\nab)^i \tr\chi\|_{L^2_{sc}(\S)}\ls (\M R[\a]+1)^2.$$
\end{proposition}
\begin{proof}
From (\ref{eqn 2 trchib}), we have the schematic null structure equation:
$$\nab_4\tr\chi=\chih\cdot\chih+\p\cdot\p.$$
Commuting this equation with $i$ angular derivatives, by Proposition \ref{commute} we have
\begin{equation*}
\begin{split}
\nab_4 \nab^i \tr\chi=&\sum_{i_1+i_2=i}\nab^{i_1}\chih\nab^{i_2}\chih+\sum_{i_1+i_2+i_3+i_4+1=i}\nabla^{i_1}\p^{i_2+1}\nabla^{i_3}\chih\nabla^{i_4} (\p, \chih)\\
&+\sum_{i_1+i_2+i_3+i_4=i}\nab^{i_1}\p^{i_2}\nab^{i_3}\p\nab^{i_4}\p.
\end{split}
\end{equation*}
Applying Proposition \ref{transport} and multiplying $(\at)^{i}$ on both sides of equation, we calculate as above
\begin{equation*}
\begin{split}
&\|(\at\nab)^{i}\tr\chi\|_{L^2_{sc}(\S)}\\
\leq& \sum_{i_1+i_2=i}\int_0^{\ub}a\|(\at)^{i}\nab^{i_1}(\f{\chih}{\at})\nab^{i_2}(\f{\chih}{\at})\|_{L^{2}_{sc}(S_{u,\ub'})}d\ub'\\
&+\sum_{i_1+i_2+i_3+i_4+1=i}\int_0^{\ub}a\|(\at)^{i}\nab^{i_1}\p^{i_2+1}\nab^{i_3}(\f{\chih}{\at})\nab^{i_4}(\f{\p}{\at} ,\f{\chih}{\at})\|_{L^{2}_{sc}(S_{u,\ub'})}d\ub'\\
&+\sum_{i_1+i_2+i_3+i_4=i}\int_0^{\ub}\|(\at)^{i}\nab^{i_1}\p^{i_2}\nab^{i_3}\p\nab^{i_4}\p\|_{L^{2}_{sc}(S_{u,\ub'})}d\ub'.\\
\leq& \f{a}{|u|}O[\chih]\cdot O[\chih]+\f{a}{|u|^2}O^3+\f{1}{|u|}O^2\\
\leq&O[\chih]\cdot O[\chih]+1\leq (\M R[\a]+1)^2.
\end{split}
\end{equation*}
For the last inequality, we use Proposition \ref{chih.bd}. \\
\end{proof}

We derive estimates for $\tr\chib$.  
\begin{proposition}\label{trchib.bd}
Under the assumptions of Theorem \ref{main.thm1} and the bootstrap assumptions \eqref{BA.0}, we have 
$$\sum_{i\leq 10} \f{a}{|u|}\|(\at\nab)^i (\tr\chib+\f{2}{|u|})\|_{L^2_{sc}(\S)}\ls \M R[\rho]+\underline{\M R}[\rho]+1, \quad \sum_{i\leq 10} \f{a}{|u|^2}\|(\at\nab)^i \tr\chib\|_{L^2_{sc}(\S)}\ls 1.$$
\end{proposition}
\begin{proof}
From (\ref{eqn 2 trchib}), we have the schematic null structure equation:   
\begin{equation*}
\begin{split}
\nab_3\tc+\tr\chib\tc=\f{2}{|u|^2}(\O^{-1}-1)+\tc\tc+\p\tr\chib-|\chibh|^2.
\end{split}
\end{equation*}
Commuting this equation with $i$ angular derivatives, by Proposition \ref{commute} we have
\begin{equation*}
\begin{split}
&\nab_3\nab^i\tc+\f{i+2}{2}\tr\chib\nab^i\tc=\sum_{i_1+i_2+i_3=i}\nab^{i_1}\p^{i_2}\nab^{i_3}\bigg(\f{2}{|u|^2}(\O^{-1}-1)+\tc\tc+\p\tr\chib-|\chibh|^2 \bigg)\\
&\quad\quad+\sum_{i_1+i_2+i_3+i_4=i}\nab^{i_1}\p^{i_2}\nab^{i_3}(\p, \chibh, \tc)\nab^{i_4} \tc+\sum_{i_1+i_2+i_3+i_4=i-1}\nab^{i_1}\p^{i_2+1}\nab^{i_3}\tr\chib\nab^{i_4}\tc.
\end{split}
\end{equation*}
Denote the RHS of above equation to be $\tilde{F}$. As proceeded in Proposition \ref{chibh.bd}, applying Proposition \ref{el} and rewriting everything in scale invariant norms, we arrive at 
\begin{equation*}
\begin{split}
\f{a}{|u|}\|(\at\nab)^i\tc\|_{L^2_{sc}(\S)}\leq& \f{a}{|u_{\infty}|}\|(\at\nab)^i\tc\|_{L^2_{sc}(S_{u_{\infty},\ub})}+\int_{u_{\infty}}^u \f{a^2}{|u'|^3}\|(\at)^i\tilde{F}\|_{L^2_{sc}(S_{u',\ub})}du'\\
=& \f{a}{|u_{\infty}|}\|(\at\nab)^i\tc\|_{L^2_{sc}(S_{u_{\infty},\ub})}+I_1+I_2+I_3+I_4,
\end{split}
\end{equation*}
where 
$$\f{a}{|u_{\infty}|}\|(\at\nab)^i\tc\|_{L^2_{sc}(S_{u_{\infty},\ub})}\leq \mathcal{I}_0\lesssim 1,$$
\begin{equation*}
\begin{split}
I_1=&\int_{u_{\infty}}^u \f{a^2}{|u'|^3}\|(\at)^i \sum_{i_1+i_2+i_3+i_4=i}\nab^{i_1}\p^{i_2}\nab^{i_3}(\p, \tc, \chibh)\nab^{i_4}(\p,\tc, \chibh)\|_{L^2_{sc}(S_{u',\ub})}du'\\
=&\int_{u_{\infty}}^u \f{a}{|u'|}\|(\at)^{i} \sum_{i_1+i_2+i_3+i_4=i}\nab^{i_1}\p^{i_2}\nab^{i_3}(\f{\at}{|u'|}\p, \f{\at}{|u'|}\tc, \f{\at}{|u'|}\chibh)\nab^{i_4}(\f{\at}{|u'|}\p, \f{\at}{|u'|}\tc, \f{\at}{|u'|}\chibh)\|_{L^2_{sc}(S_{u',\ub})}du'\\
\leq& \int_{u_{\infty}}^u \f{a}{|u'|}\cdot\f{1}{|u'|}\cdot \big (O^2[\chibh]+1\big) \, du' \quad  (\mbox{by Proposition \ref{4.3} and letting}\,\, a \,\, \mbox{to be sufficiently large})\\
\ls& O^2[\chibh]+1\ls 1  \quad (\mbox{by Proposition \ref{chibh.bd}}),
\end{split}
\end{equation*}
\begin{equation*}
\begin{split}
I_2=&\int_{u_{\infty}}^u \f{a^2}{|u'|^3}\|(\at)^i \sum_{i_1+i_2+i_3+i_4=i}\nab^{i_1}\p^{i_2}\nab^{i_3}\omb\nab^{i_4}\tr\chib\|_{L^2_{sc}(S_{u',\ub})}du'\\
=&\int_{u_{\infty}}^u \f{a}{|u'|^2}\|(\at)^{i} \sum_{i_1+i_2+i_3+i_4+1=i}\nab^{i_1}\p^{i_2}\nab^{i_3}\omb\nab^{i_4+1}(\f{a}{|u'|}\tc)\|_{L^2_{sc}(S_{u',\ub})}du'\\
&+\int_{u_{\infty}}^u \f{a}{|u'|}\|(\at)^{i} \sum_{i_1+i_2+i_3=i}\nab^{i_1}\p^{i_2}\nab^{i_3}\omb\cdot(\f{a}{|u'|^2}\tr\chib)\|_{L^2_{sc}(S_{u',\ub})}du'\\
\leq& \int_{u_{\infty}}^u \f{a}{|u'|}\cdot\f{1}{|u'|}\cdot \big (O[\omb]+1\big) \, du' \quad (\mbox{by Proposition \ref{4.3} and letting}\,\, a \,\, \mbox{to be sufficiently large})\\
\ls& O[\omb]+1\ls \mathcal{R}[\rho]+1  \quad (\mbox{by Proposition \ref{omb.bd}}),
\end{split}
\end{equation*}
\begin{equation*}
\begin{split}
I_3=&\int_{u_{\infty}}^u \f{a^2}{|u'|^3}\|(\at)^i \sum_{i_1+i_2+i_3=i}\nab^{i_1}\p^{i_2}\nab^{i_3}(\f{\O^{-1}-1}{|u'|^2})\|_{L^2_{sc}{(S_{u',\ub})}}du'\\
=&\int_{u_{\infty}}^u |u'|^{i+1} \|\sum_{i_1+i_2+i_3=i}\nab^{i_1}\p^{i_2}\nab^{i_3}(\f{\O^{-1}-1}{|u'|^2})\|_{L^2{(S_{u',\ub})}}du' \quad (\mbox{in standard norms})\\
=&\int_{u_{\infty}}^u |u'|^{i+1} \|\sum_{i_1+i_2+i_3=i}\nab^{i_1}\p^{i_2}\nab^{i_3}(\f{\O^{-1}-1}{|u'|^2})\|_{L^2{(S_{u',\ub})}}du' \,\, (\mbox{Using} \,\, \f{\partial}{\partial \ub}\O^{-1}=2\o \Leftrightarrow \nab_4\O^{-1}=2\O^{-1}\o)\\
=&\int_{u_{\infty}}^u |u'|^{i+1} \|\sum_{i_1+i_2+i_3=i}\nab^{i_1}\p^{i_2}\nab^{i_3}[\f{1}{|u'|^2}\cdot \int_0^{\ub}2\o(u',\ub',\theta^1, \theta^2)d\ub']\|_{L^2{(S_{u',\ub})}}du' \\
=&\int_{u_{\infty}}^u |u'|^{i+1} \|\sum_{i_1+i_2+i_3=i}\nab^{i_1}\p^{i_2}[\f{1}{|u'|^2}\cdot \int_0^{\ub}2\nab^{i_3}\o(u',\ub',\theta^1, \theta^2)d\ub']\|_{L^2{(S_{u',\ub})}}du' \\
\leq&  |u'|^{i+1} \|\sum_{i_1+i_2+i_3=i}\f{1}{|u'|^{i_1+i_2}}\cdot\f{1}{|u'|^2}\cdot\f{1}{|u'|^{i_3}}\cdot \f{\at}{|u'|^{\f12}}\cdot \big(\underline{\mathcal{R}}[\rho]+1\big)\, du'    (\mbox{by Proposition \ref{o.bd}})\\
\leq& \int_{u_{\infty}}^u \f{\at}{|u'|^{\f32}} \big(\underline{\mathcal{R}}[\rho]+1\big)\,du' \ls \underline{\mathcal{R}}[\rho]+1,
\end{split}
\end{equation*}
\begin{equation*}
\begin{split}
I_4=&\int_{u_{\infty}}^u \f{a^2}{|u'|^3}\|(\at)^i \sum_{i_1+i_2+i_3=i-1}\nab^{i_1}\p^{i_2+1}\cdot\tr\chib\cdot\nab^{i_3}\tc\|_{L^2_{sc}(S_{u',\ub})}du'\\
=&\int_{u_{\infty}}^u \at \|(\at)^{i-1} \sum_{i_1+i_2+i_3=i-1}\nab^{i_1}\p^{i_2+1}\cdot\f{a}{|u'|^2}\tr\chib\cdot\nab^{i_3}(\f{a}{|u'|}\tc)\|_{L^2_{sc}(S_{u',\ub})}du'\\
\leq& \int_{u_{\infty}}^u \at\cdot \f{O^3}{|u'|^2} du'\leq 1 \quad  (\mbox{by Proposition \ref{4.4}}).
\end{split}
\end{equation*}
In summary, we have obtained
$$\sum_{i\leq 10} \f{a}{|u|}\|(\at\nab)^i (\tr\chib+\f{2}{|u|})\|_{L^2_{sc}(\S)}\ls \M R[\rho]+\underline{\M R}[\rho]+1.$$
This implies
$$\sum_{i\leq 10} \f{a}{|u|^2}\|(\at\nab)^i \tr\chib\|_{L^2_{sc}(\S)}\ls 1.$$
\end{proof}

We move to the last term $\etb$.
\begin{proposition}\label{etab.bd}
Under the assumptions of Theorem \ref{main.thm1} and the bootstrap assumptions \eqref{BA.0}, we have
\[
 \sum_{i\leq 10}\|(\at\nab)^i\etab\|_{L^2_{sc}(S_{u,\ub})} \ls \underline{\M R}[\beb]+\M R[\b]+1.
\]
\end{proposition}
\begin{proof}
We use the following schematic null structure equation for $\o$:
$$\nab_3\etab+\f12\tr\chib\,\etab=\beb+\tr\chib\eta+\chibh\cdot\p.$$ 
Commuting it with angular derivative for $i$ times, we have 
\begin{equation*}
\begin{split}
&\nab_3 \nab^i\etab +\frac {i+1}{2} \trchb\nab^i\etab\\
= &\nab^i\beb+\sum_{i_1+i_2+i_3+1=i}\nab^{i_1}\p^{i_2+1}\nab^{i_3}\beb+\tr\chib\nab^i\eta+\sum_{i_1+i_2+1=i}\nab^{i_1+1}\tr\chib\nab^{i_2}(\eta,\etab)\\
&+\sum_{i_1+i_2+i_3+i_4+1=i}\nabla^{i_1}\p^{i_2+1}\nabla^{i_3}\p\nabla^{i_4}\tr\chib+\sum_{i_1+i_2+i_3+i_4=i}\nabla^{i_1}\p^{i_2}\nabla^{i_3}\p\nabla^{i_4}(\chibh, \tc).\\
\end{split}
\end{equation*}
Denote the above equality as 
$$\nab_3 \nab^i\etab +\f{i+1}{2} \trchb\nab^i\etab=G.$$
Applying Proposition \ref{el}, it holds
\begin{equation*}
\begin{split}
|u|^{i}\|\nab^i\etab\|_{L^2(\S)}\leq&|u_{\infty}|^{i}\|\nab^i\etab\|_{L^2(S_{u_{\infty},\ub})}+\int_{u_{\infty}}^u|u'|^{i}\|G\|_{L^{2}(S_{u',\underline{u}})}du'.
\end{split}
\end{equation*}
Times $a^{-\f12}$ on both sides and using $|u|\leq |u'|, \, |u|\leq |u_{\infty}|$ we have
\begin{equation}\label{etab G}
\begin{split}
a^{-\f12}|u|^{i}\|\nab^i\etab\|_{L^2(\S)}\leq&a^{-\f12}|u_{\infty}|^{i}\|\nab^i\etab\|_{L^2(S_{u_{\infty},\ub})}+a^{-\f12}\int_{u_{\infty}}^u|u'|^{i}\|G\|_{L^{2}(S_{u',\underline{u}})}du'.
\end{split}
\end{equation}
From signature table and property (\ref{signature of derivative}), we have
$$s_2(\nab^i\etab)=s_2(\etab)+i\cdot\f{1}{2}=\f12+\f{i}{2}=\f{i+1}{2}.$$ 
By conversation of signatures in each equation and property (\ref{signature of derivative}), it holds
$$s_2(G)=s_2(\nab_3 \nab^i\etab)=s_2(\nab^i \etb)+1=\f{i+3}{2}.$$
Using the definition of $L^2_{sc}(\S)$ norms
$$\|\phi\|_{L_{sc}^{2}(\S)}:=a^{-s_2(\phi)}|u|^{2s_2(\phi)}\|\phi\|_{L^{2}(\S)},$$
we have
$$\|\nab^i\etab\|_{L^2_{sc}(\S)}=a^{-\f{i+1}{2}}|u|^{i+1}\|\nab^i\etab\|_{L^2(\S)}, \quad \|G\|_{L^2_{sc}(\S)}=a^{-\f{i+3}{2}}|u|^{i+3}\|G\|_{L^2(\S)}.$$
That is equivalent to
$$a^{-\f12}|u|^{i}\|\nab^i\etab\|_{L^2(\S)}=\f{1}{|u|}\|(\at\nab)^i\etab\|_{L^2_{sc}(\S)}, \quad a^{-\f12}|u|^{i}\|G\|_{L^2(\S)}= \f{a}{|u|^3}\|(\at)^iG\|_{L^2_{sc}(\S)}.$$
We then rewrite (\ref{etab G}) in $L^2_{sc}(\S)$ norms 
\begin{equation*}
\begin{split}
\f{1}{|u|}\|(\at\nab)^i\etab\|_{L^2_{sc}(\S)}\leq&\f{1}{|\ui|}\|(\at\nab)^i\etab\|_{L^2_{sc}(S_{u_{\infty},\ub})}+\int_{u_{\infty}}^{u}\f{a}{|u'|^3}\|(\at\nab)^{i}\beb\|_{L^2_{sc}(S_{u',\ub})}du' \\
&+\int_{u_{\infty}}^{u}\f{a}{|u'|^3}\|\sum_{i_1+i_2+i_3+1=i}(\at)^i\nabla^{i_1}\p^{i_2+1}\nab^{i_3}\beb\|_{L^2_{sc}(S_{u',\ub})}du'\\
&+\int_{u_{\infty}}^{u}\f{a}{|u'|^3}\|\tr\chib(\at\nab)^{i}\eta\|_{L^2_{sc}(S_{u',\ub})}du'\\
&+\int_{u_{\infty}}^{u}\f{a}{|u'|^3}\|\sum_{i_1+i_2+1=i}(\at\nab)^{i_1+1}\tr\chib(\at\nab)^{i_2}\eta\|_{L^2_{sc}(S_{u',\ub})}du'\\
&+\int_{u_{\infty}}^{u}\f{a}{|u'|^3}\|\sum_{i_1+i_2+i_3+i_4+1=i}(\at)^i\nabla^{i_1}\p^{i_2+1}\nabla^{i_3}\p\nabla^{i_4} \tr\chib\|_{L^2_{sc}(S_{u',\ub})}du'\\
&+\int_{u_{\infty}}^{u}\f{a}{|u'|^3}\|\sum_{i_1+i_2+i_3+i_4=i}(\at)^i\nabla^{i_1}\p^{i_2}\nabla^{i_3}\p\nabla^{i_4} (\chibh, \tc)\|_{L^2_{sc}(S_{u',\ub})}du'.
\end{split}
\end{equation*}
For the first term, we have
$$\f{1}{|\ui|}\|(\at\nab)^i\etab\|_{L^2_{sc}(S_{u_{\infty},\ub})}\leq \f{\M I^{(0)}(\ub)}{|\ui|}\ls \f{1}{|\ui|}.$$
For the terms involving $\beb$, we have
\begin{equation*}
\begin{split}
&\int_{u_{\infty}}^{u}\f{a}{|u'|^3}\|(\at\nab)^{i}\beb\|_{L^2_{sc}(S_{u',\ub})}du'+\int_{u_{\infty}}^{u}\f{a}{|u'|^3}\|\sum_{i_1+i_2+i_3+1=i}(\at)^i\nabla^{i_1}\p^{i_2+1}\nab^{i_3}\beb\|_{L^2_{sc}(S_{u',\ub})}du'\\
\leq&\bigg(\int_{u_{\infty}}^{u}\f{a}{|u'|^2}\|(\at\nab)^{i}\beb\|^2_{L^2_{sc}(S_{u',\ub})}du'\bigg)^{\f12} \bigg(\int_{\ui}^u\f{a}{|u'|^4} du'\bigg)^{\f12}+\int_{\ui}^u \f{a}{|u'|^3}\cdot \f{\at}{|u'|}\cdot O^2\, du'\\
=&\|(\at\nab)^{i}\beb\|_{L^2_{sc}(\Hb_{\ub}^{(\ui,u)})}  \cdot \f{\at}{|u|^{\f32}}+\f{a^{\f32}}{|u|^3}O^{2}\\\leq& \underline{\M R}[\beb]\cdot \f{\at}{|u|^{\f32}}+\f{a^{\f32}}{|u|^3}O^{2}\leq \f{\underline{\M R}[\beb]}{|u|}+\f{O^2}{\at|u|}\leq \f{\underline{\M R}[\beb]+1}{|u|}.
\end{split}
\end{equation*}
Here we employ (\ref{4.6}), (\ref{scale invariant norms 2}) and (\ref{R i Hb}).

\noindent For next two terms, we use (\ref{Holder's}), bootstrap assumption (\ref{BA.0}) and Proposition \ref{eta.bd} to obtain
\begin{equation*}
\begin{split}
&\int_{u_{\infty}}^{u}\f{a}{|u'|^3}\|\tr\chib(\at\nab)^{i}\eta\|_{L^2_{sc}(S_{u',\ub})}du'+\int_{u_{\infty}}^{u}\f{a}{|u'|^3}\|\sum_{i_1+i_2+1=i}(\at\nab)^{i_1+1}\tr\chib(\at\nab)^{i_2}\eta\|_{L^2_{sc}(S_{u',\ub})}du'\\
&\leq \int_{u_{\infty}}^{u}\f{1}{|u'|^2}\cdot\f{a}{|u'|^2}\|\tr\chib\|_{L^{\infty}_{sc}(S_{u',\ub})}\|(\at\nab)^{i}\eta\|_{L^2_{sc}(S_{u',\ub})}du'\\
&\quad+\int_{u_{\infty}}^{u}\f{1}{|u'|^2}\|\sum_{i_1+i_2+1=i}(\at\nab)^{i_1+1}\bigg(\f{a}{|u'|}(\tr\chib+\f{2}{|u'|})\bigg)(\at\nab)^{i_2}\eta\|_{L^2_{sc}(S_{u',\ub})}du'\\
&\leq\f{O[\eta]}{|u|}+\int_{\ui}^u \f{O^2}{|u'|^3}du'\leq \f{\M R[\b]+1}{|u|}+\f{O^2}{|u|^2}\leq \f{\underline{\M R}[\beb]+1}{|u|}.
\end{split}
\end{equation*}
As calculated above, for the sixth term we have
\begin{equation*}
\begin{split}
&\int_{u_{\infty}}^{u}\f{a}{|u'|^3}\|\sum_{i_1+i_2+i_3+i_4+1=i}(\at)^i\nabla^{i_1}\p^{i_2+1}\nabla^{i_3}\p\nabla^{i_4} \tr\chib\|_{L^2_{sc}(S_{u',\ub})}du'\\
=&\int_{u_{\infty}}^{u}\f{\at}{|u'|}\|\sum_{i_1+i_2+i_3+i_4+1=i}(\at)^{i-1}\nabla^{i_1}\p^{i_2+1}\nabla^{i_3}\p\nabla^{i_4} (\f{a}{|u'|^2}\tr\chib)\|_{L^2_{sc}(S_{u',\ub})}du'\\
\leq&\int_{u_{\infty}}^{u} \f{\at}{|u'|}\cdot\f{O^3}{|u'|^2} du' \leq \f{\at \cdot O^3}{|u|^2}\leq \f{1}{|u|}.
\end{split}
\end{equation*}
And for the last term, with (\ref{4.3}) we have
\begin{equation*}
\begin{split}
&\int_{u_{\infty}}^{u}\f{a}{|u'|^3}\|\sum_{i_1+i_2+i_3+i_4=i}(\at)^i\nabla^{i_1}\p^{i_2}\nabla^{i_3}\p\nabla^{i_4} (\chibh, \tc)\|_{L^2_{sc}(S_{u',\ub})}du'\\
=&\int_{u_{\infty}}^{u}\f{\at}{|u'|^2}\|\sum_{i_1+i_2+i_3+i_4=i}(\at)^i\nabla^{i_1}\p^{i_2}\nabla^{i_3}\p\nabla^{i_4} (\f{\at}{|u'|}\chibh, \f{\at}{|u'|}\tc)\|_{L^2_{sc}(S_{u',\ub})}du'\\
\leq&\int_{u_{\infty}}^{u} \f{\at}{|u'|^2}\cdot\f{O^2}{|u'|} du' \leq \f{\at \cdot O^2}{|u|^2}\leq \f{1}{\at}\cdot\f{O^2}{|u|}\leq \f{1}{|u|}.
\end{split}
\end{equation*}
Combining all the estimates derived, we have
$$\f{1}{|u|}\|(\at\nab)^i\etab\|_{L^2_{sc}(\S)}\leq \f{1}{|u_{\infty}|}+\f{\underline{\M R}[\beb]+\M R[\b]+1}{|u|}.$$
Multiplying $|u|$ on both sides, we obtain
$$\|(\at\nab)^i\etab\|_{L^2_{sc}(\S)}\ls 1+\underline{\M R}[\beb]+\M R[\b].$$
\end{proof}

\section{$L^2(\S)$ ESTIMATE FOR CURVATURE}\label{secCurvatureL2}
For $i\leq 9$, we have
\begin{proposition}\label{a.bd}
Under the assumptions of Theorem \ref{main.thm1} and the bootstrap assumptions \eqref{BA.0}, we have
\[
 \sum_{i\leq 9}\f{1}{\at}\|(\at\nab)^i\a\|_{L^2_{sc}(S_{u,\ub})} \ls \underline{\M R}[\b]+1.
\]
\end{proposition}
\begin{proof}
We have systematical null Bianchi equation:
$$\nab_3 \a+\f12\tr\chib\a=\nab\b+\p\a+(\p,\chih)\Psi.$$
Commuting it with angular derivative for $i$ times, we have
\begin{equation*}
\begin{split}
&\nab_3 \nab^i\a +\frac {i+1}{2} \trchb\nab^i\a\\
= &\nab^{i+1}\b+\sum_{i_1+i_2+i_3+1=i}\nab^{i_1}\p^{i_2+1}\nab^{i_3+1}\b+\sum_{i_1+i_2+1=i}\nab^{i_1+1}\tr\chib\nab^{i_2}\a+\sum_{i_1+i_2=i}\nab^{i_1}\chibh\nab^{i_2}\a\\
&+\sum_{i_1+i_2+i_3+i_4+1=i}\nabla^{i_1}\p^{i_2+1}\nabla^{i_3}(\p, \tr\chib, \chibh)\nabla^{i_4}\a+\sum_{i_1+i_2+i_3+i_4=i}\nabla^{i_1}\p^{i_2}\nabla^{i_3}(\p,\chih)\nabla^{i_4}\Psi.\\
\end{split}
\end{equation*}
Denote the above equality as 
$$\nab_3 \nab^i\a +\f{i+1}{2} \trchb\nab^i\a=G.$$
Applying Proposition \ref{el}, it holds
\begin{equation*}
\begin{split}
|u|^{i}\|\nab^i\a\|_{L^2(\S)}\leq&|u_{\infty}|^{i}\|\nab^i\a\|_{L^2(S_{u_{\infty},\ub})}+\int_{u_{\infty}}^u|u'|^{i}\|G\|_{L^{2}(S_{u',\underline{u}})}du'.
\end{split}
\end{equation*}
Multiplying $a^{-\f12}$ on both sides, we have
\begin{equation}\label{a G}
\begin{split}
a^{-\f12}|u|^{i}\|\nab^i\a\|_{L^2(\S)}\leq&a^{-\f12}|u_{\infty}|^{i}\|\nab^i\a\|_{L^2(S_{u_{\infty},\ub})}+\int_{u_{\infty}}^u a^{-\f12}|u'|^{i}\|G\|_{L^{2}(S_{u',\underline{u}})}du'.
\end{split}
\end{equation}
From signature table and property (\ref{signature of derivative}), we have
$$s_2(\nab^i\a)=s_2(\a)+i\cdot\f{1}{2}=0+\f{i}{2}=\f{i}{2}.$$ 
By conversation of signature in each equation and property (\ref{signature of derivative}), it holds
$$s_2(G)=s_2(\nab_3 \nab^i\a)=s_2(\nab^i \a)+1=\f{i+2}{2}.$$
Using the definition of $L^2_{sc}(\S)$ norms
$$\|\phi\|_{L_{sc}^{2}(\S)}:=a^{-s_2(\phi)}|u|^{2s_2(\phi)}\|\phi\|_{L^{2}(\S)},$$
we have
$$\|\nab^i\a\|_{L^2_{sc}(\S)}=a^{-\f{i}{2}}|u|^{i}\|\nab^i\a\|_{L^2(\S)}, \quad \|G\|_{L^2_{sc}(\S)}=a^{-\f{i+2}{2}}|u|^{i+2}\|G\|_{L^2(\S)}.$$
That is equivalent to
$$a^{-\f12}|u|^{i}\|\nab^i\a\|_{L^2(\S)}=a^{-\f12}\|(\at\nab)^i\a\|_{L^2_{sc}(\S)}, \quad  a^{-\f12}|u|^{i}\|G\|_{L^2(\S)}= \f{\at}{|u|^2}\|(\at)^iG\|_{L^2_{sc}(\S)}.$$
We then rewrite (\ref{a G}) in $L^2_{sc}(\S)$ norms 
\begin{equation*}
\begin{split}
&a^{-\f12}\|(\at\nab)^i\a\|_{L^2_{sc}(\S)}\\
\leq&a^{-\f12}\|(\at\nab)^i\a\|_{L^2_{sc}(S_{u_{\infty},\ub})}+\int_{u_{\infty}}^{u}\f{\at}{|u'|^2}\|(\at\nab)^{i+1}\b\|_{L^2_{sc}(S_{u',\ub})}du' \\
&+\int_{u_{\infty}}^{u}\f{\at}{|u'|^2}\|\sum_{i_1+i_2+i_3+1=i}(\at)^i\nab^{i_1}\p^{i_2+1}\nab^{i_3+1}\b\|_{L^2_{sc}(S_{u',\ub})}du' \\
&+\int_{u_{\infty}}^{u}\f{\at}{|u'|^2}\|\sum_{i_1+i_2+1=i}(\at)^i\nab^{i_1+1}\tr\chib\nab^{i_2}\a\|_{L^2_{sc}(S_{u',\ub})}du' \\
&+\int_{u_{\infty}}^{u}\f{\at}{|u'|^2}\|\sum_{i_1+i_2=i}(\at)^i\nab^{i_1}\chibh\nab^{i_2}\a\|_{L^2_{sc}(S_{u',\ub})}du' \\
&+\int_{u_{\infty}}^{u}\f{\at}{|u'|^2}\|\sum_{i_1+i_2+i_3+i_4+1=i}(\at)^i\nabla^{i_1}\p^{i_2+1}\nabla^{i_3}(\p, \tr\chib, \chibh)\nabla^{i_4}\a\|_{L^2_{sc}(S_{u',\ub})}du' \\
&+\int_{u_{\infty}}^{u}\f{\at}{|u'|^2}\|\sum_{i_1+i_2+i_3+i_4=i}(\at)^i\nabla^{i_1}\p^{i_2}\nabla^{i_3}(\p,\chih)\nabla^{i_4}\Psi\|_{L^2_{sc}(S_{u',\ub})}du'. \\
\end{split}
\end{equation*}
For the first term, we have
$$a^{-\f12}\|(\at\nab)^i\a\|_{L^2_{sc}(S_{u_{\infty},\ub})}\leq {\M I^{(0)}(\ub)}\ls 1.$$
For the terms involving $\b$, we have
\begin{equation*}
\begin{split}
&\int_{u_{\infty}}^{u}\f{\at}{|u'|^2}\|(\at\nab)^{i+1}\b\|_{L^2_{sc}(S_{u',\ub})}du'\\
&+\int_{u_{\infty}}^{u}\f{1}{|u'|^2}\|\sum_{i_1+i_2+i_3+1=i}(\at)^{i+1}\nabla^{i_1}\p^{i_2+1}\nab^{i_3+1}\b\|_{L^2_{sc}(S_{u',\ub})}du'\\
\leq&\bigg(\int_{u_{\infty}}^{u}\f{a}{|u'|^2}\|(\at\nab)^{i+1}\b\|^2_{L^2_{sc}(S_{u',\ub})}du'\bigg)^{\f12} \bigg(\int_{\ui}^u\f{1}{|u'|^2} du'\bigg)^{\f12}+\int_{\ui}^u \f{1}{|u'|^2}\cdot \f{\at}{|u'|}\cdot O^2\, du'\\
\leq&\|(\at\nab)^{i+1}\b\|_{L^2_{sc}(\Hb_{\ub}^{(\ui,u)})}  \cdot \f{1}{|u|^{\f12}}+\f{a^{\f12}}{|u|^2}O^{2}\\ \leq&a^{-\f12}\|(\at\nab)^{i+1}\b\|_{L^2_{sc}(\Hb_{\ub}^{(\ui,u)})}  \cdot \f{\at}{|u|^{\f12}}+\f{a^{\f12}}{|u|^2}O^{2}\\
\leq& a^{-\f12}\|(\at\nab)^{i+1}\b\|_{L^2_{sc}(\Hb_{\ub}^{(\ui,u)})}+1 \leq \underline{\M R}[\beb]+1,
\end{split}
\end{equation*}
where we employ (\ref{4.6}), (\ref{scale invariant norms 2}) and (\ref{R i Hb}).

\noindent For the next two terms, we use (\ref{4.6}) and obtain 
\begin{equation*}
\begin{split}
&\int_{u_{\infty}}^{u}\f{\at}{|u'|^2}\|\sum_{i_1+i_2+1=i}(\at)^i\nab^{i_1+1}\tr\chib\nab^{i_2}\a\|_{L^2_{sc}(S_{u',\ub})}du' \\
&+\int_{u_{\infty}}^{u}\f{\at}{|u'|^2}\|\sum_{i_1+i_2=i}(\at)^i\nab^{i_1}\chibh\nab^{i_2}\a\|_{L^2_{sc}(S_{u',\ub})}du' \\
\leq& \int_{u_{\infty}}^{u}\f{a^{-\f12}}{|u'|}\|\sum_{i_1+i_2+1=i}(\at)^{i+1}\nab^{i_1+1}\bigg(\f{a}{|u'|}(\tr\chib+\f{2}{|u'|})\bigg)\nab^{i_2}(\f{\a}{\at})\|_{L^2_{sc}(S_{u',\ub})}du' \\
&+\int_{u_{\infty}}^{u}\f{1}{|u'|}\|\sum_{i_1+i_2=i}(\at)^{i+1}\nab^{i_1}(\f{\at}{|u'|}\chibh)\nab^{i_2}(\f{\a}{\at})\|_{L^2_{sc}(S_{u',\ub})}du' \\
\leq& \int_{\ui}^{u}\f{a^{-\f12}}{|u'|}\cdot \f{\at}{|u'|}O^2 \,du'+ \int_{\ui}^{u}\f{1}{|u'|}\cdot \f{\at}{|u'|}O^2 \,du'\\
\leq& \f{O^2}{|u|}+\f{\at\cdot O^2}{|u|}\leq 1.
\end{split}
\end{equation*}
For the last two terms, we have
\begin{equation*}
\begin{split}
&\int_{u_{\infty}}^{u}\f{\at}{|u'|^2}\|\sum_{i_1+i_2+i_3+i_4+1=i}(\at)^i\nabla^{i_1}\p^{i_2+1}\nabla^{i_3}(\p, \tr\chib, \chibh)\nabla^{i_4}\a\|_{L^2_{sc}(S_{u',\ub})}du' \\
&+\int_{u_{\infty}}^{u}\f{\at}{|u'|^2}\|\sum_{i_1+i_2+i_3+i_4=i}(\at)^i\nabla^{i_1}\p^{i_2}\nabla^{i_3}(\p,\chih)\nabla^{i_4}\Psi\|_{L^2_{sc}(S_{u',\ub})}du' \\
\leq&\int_{u_{\infty}}^{u}a^{-\f12}\|\sum_{i_1+i_2+i_3+i_4+1=i}(\at)^{i+1}\nabla^{i_1}\p^{i_2+1}\nabla^{i_3}(\f{a}{|u'|^2}\p, \f{a}{|u'|^2}\tr\chib, \f{a}{|u'|^2}\chibh)\nabla^{i_4}(\f{\a}{\at})\|_{L^2_{sc}(S_{u',\ub})}du' \\
&+\int_{u_{\infty}}^{u}\f{\at}{|u'|^2}\|\sum_{i_1+i_2+i_3+i_4=i}(\at)^{i+1}\nabla^{i_1}\p^{i_2}\nabla^{i_3}(\f{\p}{\at},\f{\chih}{\at})\nabla^{i_4}\Psi\|_{L^2_{sc}(S_{u',\ub})}du' \\
\leq& \int_{\ui}^{u}a^{-\f12} \f{a}{|u|^2} O^3 \,du'+ \int_{\ui}^{u}\f{\at}{|u'|^2}\cdot \f{\at}{|u'|}O^2 \,du'\\
\leq& \f{\at\cdot O^2}{|u|}+\f{a\cdot O^2}{|u|^2}\leq 1.
\end{split}
\end{equation*}
Gathering all the estimates, for sufficiently large $a$ we have showed that
$$\sum_{i\leq 9}\f{1}{\at}\|(\at\nab)^i\a\|_{L^2_{sc}(S_{u,\ub})} \ls \underline{\M R}[\b]+1.$$
\end{proof}

Let $\Psi\in \{\b, \rho, \sigma, \beb, \ab\}$, we proceed to prove
\begin{proposition}\label{etab.bd}
Under the assumptions of Theorem \ref{main.thm1} and the bootstrap assumptions \eqref{BA.0}, we have
\[
 \sum_{i\leq 9}\|(\at\nab)^i\Psi\|_{L^2_{sc}(S_{u,\ub})} \ls \M R[\a]+\underline{\M R}[\b]+1.
\]
\end{proposition}
\begin{proof}
For $\Psi\in \{\b, \rho, \sigma, \beb, \ab\}$,  we have the systematical null Bianchi equations: 
$$\nab_4 \Psi=\nab\Psi+\nab\alpha+(\chibh,\p) (\Psi, \a).$$
By Proposition \ref{commute}, we commute this equations with $i$ angular derivatives
\begin{equation*}
\begin{split}
\nab_4\nab^i\Psi=&\nab^{i+1}\Psi+\nab^{i+1}\a+\sum_{i_1+i_2+i_3+1=i}\nabla^{i_1}\p^{i_2+1}\nabla^{i_3+1}(\Psi, \a)\\
&+\sum_{i_1+i_2+i_3+i_4=i}\nabla^{i_1}\p^{i_2}\nab^{i_3}(\p,\chih)\nabla^{i_4}\Psi+\sum_{i_1+i_2+i_3+i_4=i}\nabla^{i_1}\p^{i_2}\nabla^{i_3}(\p, \chibh)\nab^{i_4}(\Psi, \a).
\end{split}
\end{equation*}
Applying Proposition \ref{transport} and multiplying $(\at)^i$ on both sides, we have

\begin{equation*}
\begin{split}
&\|(\at\nab)^{i}\Psi\|_{L^2_{sc}(\S)}\\
\leq&\int_0^{\ub}\|(\at)^{i}\nab^{i+1}\Psi\|
_{L^{2}_{sc}(S_{u,\ub'})}d\ub'+\int_0^{\ub}\|(\at)^{i}\nab^{i+1}\a\|
_{L^{2}_{sc}(S_{u,\ub'})}d\ub'\\
&+\sum_{i_1+i_2+i_3+1=i}\int_0^{\ub}\|(\at)^{i}\nabla^{i_1}\p^{i_2+1}\nabla^{i_3+1}(\Psi, \a)\|
_{L^{2}_{sc}(S_{u,\ub'})}d\ub'\\
&+\sum_{i_1+i_2+i_3+i_4=i}\int_0^{\ub}\|(\at)^{i}\nabla^{i_1}\p^{i_2}\nabla^{i_3}(\p,\chih)\nab^{i_4}\Psi\|
_{L^{2}_{sc}(S_{u,\ub'})}d\ub'\\
&+\sum_{i_1+i_2+i_3+i_4=i}\int_0^{\ub}\|(\at)^{i}\nabla^{i_1}\p^{i_2}\nabla^{i_3}(\p,\chibh)\nab^{i_4}(\Psi, \a)\|
_{L^{2}_{sc}(S_{u,\ub'})}d\ub'\\
\leq&a^{-\f12}\int_0^{\ub}\|(\at)^{i+1}\nab^{i+1}\Psi\|
_{L^{2}_{sc}(S_{u,\ub'})}d\ub'+a^{-\f12}\int_0^{\ub}\|(\at)^{i+1}\nab^{i+1}\a\|
_{L^{2}_{sc}(S_{u,\ub'})}d\ub'\\
&+\sum_{i_1+i_2+i_3+1=i}\int_0^{\ub}\|(\at)^{i+1}\nabla^{i_1}\p^{i_2+1}\nabla^{i_3+1}(\f{\Psi}{\at}, \f{\a}{\at})\|
_{L^{2}_{sc}(S_{u,\ub'})}d\ub'\\
&+\sum_{i_1+i_2+i_3+i_4=i}\int_0^{\ub}\|(\at)^{i+1}\nabla^{i_1}\p^{i_2}\nabla^{i_3}(\f{\p}{\at},\f{\chih}{\at})\nab^{i_4}\Psi\|
_{L^{2}_{sc}(S_{u,\ub'})}d\ub'\\
&+\sum_{i_1+i_2+i_3+i_4+1=i}\int_0^{\ub}a^{-\f12}|u|\|(\at)^{i+1}\nabla^{i_1}\p^{i_2+1}\nabla^{i_3}(\f{\at}{|u|}\p,\f{\at}{|u|}\chibh)\nab^{i_4}(\f{\Psi}{\at}, \f{\a}{\at})\|
_{L^{2}_{sc}(S_{u,\ub'})}d\ub'\\
&+\sum_{i_1+i_2=i}\int_0^{\ub}a^{-\f12}|u|\|(\at)^{i+1}\nabla^{i_1}(\f{\at}{|u|}\p,\f{\at}{|u|}\chibh)\nab^{i_2}(\f{\Psi}{\at}, \f{\a}{\at})\|
_{L^{2}_{sc}(S_{u,\ub'})}d\ub'\\
\leq& a^{-\f12}\bigg(\int_0^{\ub}\|(\at)^{i+1}\nab^{i+1}\Psi\|^2
_{L^{2}_{sc}(S_{u,\ub'})}d\ub'\bigg)^{\f12}+a^{-\f12}\bigg(\int_0^{\ub}\|(\at)^{i+1}\nab^{i+1}\a\|^2
_{L^{2}_{sc}(S_{u,\ub'})}d\ub'\bigg)^{\f12}\\
&+\f{\at}{|u|}O^2+a^{-\f12}|u|\f{a}{|u|^2}O^3+a^{-\f12}|u|\f{\at}{|u|}(O[\chibh]\cdot O[\a]+1)\\
\leq& a^{-\f12}\|(\at\nab)^{i+1}\Psi\|_{\sh}+a^{-\f12}\|(\at\nab)^{i+1}\a\|_{\sh}+O[\chibh]\cdot O[\a]+1\\
\leq& R[\a]+O[\chibh]\cdot O[\a]+1\ls \M R[\a]+O[\a]+1\ls \M R[\a]+\underline{\M R}[\b]+1 ,
\end{split}
\end{equation*}
where we use (\ref{4.6}), (\ref{4.7}), Proposition \ref{chibh.bd} and Proposition \ref{a.bd}.
\end{proof}

\section{ENERGY ESTIMATE}\label{energy estimate}

In this section with scale invariant norms we will derive energy estimates for curvature components and their angular derivatives. Our goal is to show that 
\begin{equation}\label{R Rb}
\M R+\Rb \ls \M I^{(0)}+(\M I^{(0)})^2+1.
\end{equation}
Together with the estimates derived in previous sections, we hence improve all the bootstrap assumptions in \eqref{BA.0}. And Theorem \ref{main.thm1} is therefore established. 

\subsection{Integration by Parts Formula}
We first state and prove several useful formula.   For $D_{u,\ub}=(u_{\infty}, u)\times (0,\ub)$ by direct computations, we have
\begin{proposition}\label{ee 1}
Suppose $\phi_1$ and $\phi_2$ are $r$ tensorfields, then
\begin{equation*}
\begin{split}
&\int_{D_{u,\ub}} \phi_1 \nabla_4\phi_2+\int_{D_{u,\ub}}\phi_2\nabla_4\phi_1\\
=& \int_{\Hb_{\ub}^{(\ui,u)}} \phi_1\phi_2-\int_{\Hb_0^{(\ui,u)}} \phi_1\phi_2+\int_{D_{u,\ub}}(2\omega-\trch)\phi_1\phi_2.
\end{split}
\end{equation*}
\end{proposition}
\begin{proposition}\label{ee 2}
Suppose we have an $r$ tensorfield $^{(1)}\phi$ and an $r-1$ tensorfield $^{(2)}\phi$, then
\begin{equation*}
\begin{split}
&\int_{D_{u,\ub}}{ }^{(1)}\phi^{A_1A_2...A_r}\nabla_{A_r}{ }^{(2)}\phi_{A_1...A_{r-1}}+\int_{D_{u,\ub}}\nabla^{A_r}{ }^{(1)}\phi_{A_1A_2...A_r}{ }^{(2)}\phi^{A_1...A_{r-1}}\\
=& -\int_{D_{u,\ub}}(\eta+\etab){ }^{(1)}\phi{ }^{(2)}\phi.
\end{split}
\end{equation*}
\end{proposition}
\noindent We will also need the following analogue of Proposition \ref{ee 1} with $u$ weights incorporated. 
\begin{proposition}\label{ee 3}
Suppose $\phi$ is an $r$ tensorfield and let $\lambda_1=2(\lambda_0-\f12)$. Then
\begin{equation*}
\begin{split}
&2\int_{D_{u,\ub}} |u'|^{2\lambda_1}\phi (\nabla_3+\lambda_0\trchb)\phi\\
=& \int_{H_u^{(0,\ub)}} |u|^{2\lambda_1}|\phi|^2-\int_{H_{\ui}^{(0,\ub)}} |\ui|^{2\lambda_1}|\phi|^2+\int_{D_{u,\ub}}|u'|^{2\lambda_1}f|\phi|^2,
\end{split}
\end{equation*}
where $f$ obeys bound
$$|f|\ls \f{O}{|u|^2}.$$
\end{proposition}
\begin{proof}
Slightly modifying \eqref{evolution.id}, we have
\begin{equation*}
\begin{split}
\frac{d}{du}(\int_{\S}&|u|^{2\lambda_1}\Omega|\phi|^2)=\Lb\bigg(\int_{\S}|u|^{2\lambda_1}\O|\phi|^2 \bigg)\\
=&\int_{\S}\Omega^2 \l 2|u|^{2\lambda_1}<\phi, \nab_3\phi+\lambda_0\trchb\phi>\r\\
&+\int_{\S}\Omega^2\l  |u|^{2\lambda_1} (-\f{2\lambda_1 (e_3u)}{|u|}+(1-2\lambda_0)\trchb-2\omb)|\phi|^2\r.
\end{split}
\end{equation*}
Here we use $\Lb=\O e_3=\f{\partial}{\partial u}+b^A\f{\partial}{\partial \theta^A}$. By (\ref{trchib additional}) and bootstrap assumption \ref{BA.0}, we have
$$|-\f{2\lambda_1 (e_3u)}{|u|}+(1-2\lambda_0)\trchb-2\omb|\ls \f{O}{|u|^2}.$$
The proposition then follows via integrating with respect to $du\,d\ub$ and applying the fundamental theorem of calculus in $u$.\\
\end{proof}

\begin{remark}
Observe that for $(\Psi_1, \Psi_2)\in \{ (\a,\b), \big(\b, (\rho, \sigma)\big), \big( (\rho, \sigma), \beb\big), (\beb, \ab)\}$, we can rewrite null Bianchi equations into the systematic forms:
\begin{equation}\label{snb 1}
\nab_3 \Psi_1+\bigg(\f{1}{2}+s_2(\Psi_1)\bigg)\tr\chib\Psi_1-\M D\Psi_2=(\p, \chih)\Psi,
\end{equation}
\begin{equation}\label{snb 2}
\nab_4 \Psi_2-\M D^*\Psi_1=\p\Psi+\chih(\Psi, \a),
\end{equation}
where $\Psi\in \{\b, \rho, \sigma, \beb, \ab\}$. Here we denote $\M D$ to be a differential operator on $\S$, and $\M D^*$ is the $L^2(\S)$ adjoint operator of $\M D$. We further commute (\ref{snb 1}) and (\ref{snb 2}) with angular derivative $\nab$ for $i$ times and get
\begin{equation}\label{snb 3}
\nab_3\nab^i\Psi_1+\big(\f{1+i}{2}+s_2(\Psi_1)\big)\tr\chib\nab^i\Psi_1-\M D \nab^i\Psi_2=P,
\end{equation}
\begin{equation}\label{snb 4}
\nab_4\nab^i\Psi_2-\M D^*\nab^{i}\Psi=1+Q.
\end{equation}
The forms of $P$ and $Q$ will be specified later. Check signature $s_2$, we have  
\begin{equation*}
\begin{split}
s_2(\nab^i \Psi_1)=\f{i}{2}+s_2(\Psi_1), \quad &s_2(\nab^i \Psi_2)=\f{i+1}{2}+s_2(\Psi_1)\\
s_2(P)=s_2(\nab_3\nab^i\Psi_1)=\f{i+2}{2}+s_2(\Psi_1), \quad &s_2(Q)=s_2(\M D^*\nab^i\Psi_1)=\f{i+1}{2}+s_2(\Psi_1).
\end{split}
\end{equation*}

\end{remark}
\begin{remark}
The Hodge structure will play a crucial role: for pair $(\Psi_1, \Psi_2)$ or pair $(\nab^i\Psi_1, \nab^i\Psi_2)$, the angular derivative operator $\M D$ and its $L^2$ adjoint operator $\M D^*$ form a Hodge system. Through Proposition \ref{ee 2}, we have
\begin{equation}\label{snb 5}
\begin{split}
&\int_{\S}\Psi_1\,\M D \Psi_2+\Psi_2\,\M D^*\Psi_1=-\int_{\S}(\eta+\etb)\Psi_1 \Psi_2,\\
&\int_{\S}\nab^i\Psi_1\,\M D \nab^i\Psi_2+\nab^i\Psi_2\,\M D^*\nab^i\Psi_1=-\int_{\S}(\eta+\etb)\nab^i\Psi_1\nab^i\Psi_2.
\end{split}
\end{equation}
\end{remark}

We now move forward and apply Proposition \ref{ee 3} for $\nab^i\Psi_1$.  With 
$$\lambda_0=\f{1+i}{2}+s_2(\Psi_1), \quad \lambda_1:=2\lambda_0-1=i+2s_2(\Psi_1), \mbox{ we get}$$
\begin{equation}\label{6.6}
\begin{split}
&2\int_{D_{u,\ub}} |u'|^{2i+4s_2(\Psi_1)}\nab^i\Psi_1 \bigg(\nabla_3+\big(\f{1+i}{2}+s_2(\Psi_1)\big)\trchb\bigg)\nab^i\Psi_1\\\
=& \int_{H_u^{(0,\ub)}} |u|^{2i+4s_2(\Psi_1)}|\nab^i\Psi_1|^2-\int_{H_{\ui}^{(0,\ub)}} |\ui|^{2i+4s_2(\Psi_1)}|\nab^i\Psi_1|^2+\int_{D_{u,\ub}}|u'|^{2i+4s_2(\Psi_1)}f|\nab^i\Psi_1|^2,
\end{split}
\end{equation}
where $|f|\leq O/|u'|^2$. 

\noindent We also use Proposition \ref{ee 1} with substitution $\phi_1=\phi_2=|u|^{i+2s_2(\Psi_1)}\nab^i\Psi_2$
\begin{equation}\label{6.7}
\begin{split}
&2\int_{D_{u,\ub}} |u'|^{2i+4s_2(\Psi_1)} \nab^i  \Psi_2 \nabla_4\nab^i\Psi_2\\
=& \int_{\Hb_{\ub}^{(\ui,u)}} |u'|^{2i+4s_2(\Psi_1)}|\nab^i\Psi_2|^2-\int_{\Hb_0^{(\ui,u)}} |u'|^{2i+4s_2(\Psi_1)}|\nab^i\Psi_2|^2\\
&+\int_{D_{u,\ub}}|u'|^{2i+4s_2(\Psi_1)}(2\omega-\trch)|\nab^i\Psi_2|^2.
\end{split}
\end{equation}
Add (\ref{6.6}) and (\ref{6.7}), we hence obtain
\begin{equation*}
\begin{split}
&2\int_{D_{u,\ub}} |u'|^{2i+4s_2(\Psi_1)}\nab^i\Psi_1 \bigg(\nabla_3+\big(\f{1+i}{2}+s_2(\Psi_1)\big)\trchb\bigg)\nab^i\Psi_1\\
&+2\int_{D_{u,\ub}} |u'|^{2i+4s_2(\Psi_1)} \nab^i\Psi_2 \nabla_4\nab^i\Psi_2\\
=& \int_{H_u^{(0,\ub)}} |u|^{2i+4s_2(\Psi_1)}|\nab^i\Psi_1|^2-\int_{H_{\ui}^{(0,\ub)}} |\ui|^{2i+4s_2(\Psi_1)}|\nab^i\Psi_1|^2+\int_{D_{u,\ub}}|u'|^{2i+4s_2(\Psi_1)}f|\nab^i\Psi_1|^2\\
 &+\int_{\Hb_{\ub}^{(\ui,u)}}|u'|^{2i+4s_2(\Psi_1)}|\nab^i\Psi_2|^2-\int_{\Hb_0^{(\ui,u)}} |u'|^{2i+4s_2(\Psi_1)}|\nab^i\Psi_2|^2\\
 &+\int_{D_{u,\ub}}|u'|^{2i+4s_2(\Psi_1)}(2\omega-\trch)|\nab^i\Psi_2|^2.
\end{split}
\end{equation*}
Apply (\ref{snb 3}) and (\ref{snb 4}). Wtih the help of (\ref{snb 5}), we then arrive at 
\begin{equation*}
\begin{split}
&\int_{H_u^{(0,\ub)}} |u|^{2i+4s_2(\Psi_1)}|\nab^i\Psi_1|^2+\int_{\Hb_{\ub}^{(\ui,u)}}|u'|^{2i+4s_2(\Psi_1)}|\nab^i\Psi_2|^2\\
=&\int_{H_{\ui}^{(0,\ub)}} |\ui|^{2i+4s_2(\Psi_1)}|\nab^i\Psi_1|^2+\int_{\Hb_{0}^{(\ui,u)}}|u'|^{2i+4s_2(\Psi_1)}|\nab^i\Psi_2|^2\\
&+2\int_{D_{u,\ub}} |u'|^{2i+4s_2(\Psi_1)}\nab^i\Psi_1 \cdot P+2\int_{D_{u,\ub}} |u'|^{2i+4s_2(\Psi_1)}\nab^i\Psi_2 \cdot Q\\
&-2\int_{D_{u,\ub}}|u'|^{2i+4s_2(\Psi_1)}(\eta+\etb)\nab^i\Psi_1\nab^i\Psi_2\\
&+\int_{D_{u,\ub}}|u'|^{2i+4s_2(\Psi_1)}f|\nab^i\Psi_1|^2+\int_{D_{u,\ub}}|u'|^{2i+4s_2(\Psi_1)}(2\omega-\trch)|\nab^i\Psi_2|^2.
\end{split}
\end{equation*}
Using $|(\eta+\etb)\nab^i\Psi_1\nab^i\Psi_2|\leq |\eta+\etb|(\nab^i\Psi_1)^2+|\eta+\etb|(\nab^i\Psi_2)^2,$ and the fact
$$|\eta+\etb|\leq \at O/{|u'|^2}, \quad |f|\leq O/|u'|^2, \quad |2\o-\tr\chi|\leq {O}/{|u'|},$$
by applying Gr\"onwall's inequality twice (one for $du$, one for $d\ub$), we obtain
\begin{equation*}
\begin{split}
&\int_{H_u^{(0,\ub)}} |u|^{2i+4s_2(\Psi_1)}|\nab^i\Psi_1|^2+\int_{\Hb_{\ub}^{(\ui,u)}}|u'|^{2i+4s_2(\Psi_1)}|\nab^i\Psi_2|^2\\
\ls&\int_{H_{\ui}^{(0,\ub)}} |\ui|^{2i+4s_2(\Psi_1)}|\nab^i\Psi_1|^2+\int_{\Hb_{0}^{(\ui,u)}}|u'|^{2i+4s_2(\Psi_1)}|\nab^i\Psi_2|^2\\
&+2\int_{D_{u,\ub}} |u'|^{2i+4s_2(\Psi_1)}\nab^i\Psi_1 \cdot P+2\int_{D_{u,\ub}}|u'|^{2i+4s_2(\Psi_1)}\nab^i\Psi_2 \cdot Q.
\end{split}
\end{equation*}
Multiply $a^{-i-2s_2(\Psi_1)}$ on both sides
\begin{equation}\label{6.8}
\begin{split}
&\int_{H_u^{(0,\ub)}} a^{-i-2s_2(\Psi_1)}|u|^{2i+4s_2(\Psi_1)}|\nab^i\Psi_1|^2+\int_{\Hb_{\ub}^{(\ui,u)}}a^{-i-2s_2(\Psi_1)}|u'|^{2i+4s_2(\Psi_1)}|\nab^i\Psi_2|^2\\
\ls&\int_{H_{\ui}^{(0,\ub)}} a^{-i-2s_2(\Psi_1)}|\ui|^{2i+4s_2(\Psi_1)}|\nab^i\Psi_1|^2+\int_{\Hb_{0}^{(\ui,u)}}a^{-i-2s_2(\Psi_1)}|u'|^{2i+4s_2(\Psi_1)}|\nab^i\Psi_2|^2\\
&+2\int_{D_{u,\ub}} a^{-i-2s_2(\Psi_1)}|u'|^{2i+4s_2(\Psi_1)}\nab^i\Psi_1 \cdot P+2\int_{D_{u,\ub}}a^{-i-2s_2(\Psi_1)}|u'|^{2i+4s_2(\Psi_1)}\nab^i\Psi_2 \cdot Q.
\end{split}
\end{equation}
With signature identities
\begin{equation*}
\begin{split}
s_2(\nab^i \Psi_1)=\f{i}{2}+s_2(\Psi_1), \quad &s_2(\nab^i \Psi_2)=\f{i+1}{2}+s_2(\Psi_1)\\
s_2(P)=s_2(\nab_3\nab^i\Psi_1)=\f{i+2}{2}+s_2(\Psi_1), \quad &s_2(Q)=s_2(\M D^*\nab^i\Psi_1)=\f{i+1}{2}+s_2(\Psi_1),
\end{split}
\end{equation*}
and definitions
$$\|\phi\|_{L^2_{sc}(\S)}=a^{-s_2(\phi)}|u|^{2s_2(\phi)}\|\phi\|_{L^2(\S)},$$
$$\|\phi\|_{L^1_{sc}(\S)}=a^{-s_2(\phi)}|u|^{2s_2(\phi)-1}\|\phi\|_{L^1(\S)},$$
we rewrite (\ref{6.8}) as   
\begin{equation*}
\begin{split}
&\int_{H_u^{(0,\ub)}} \|\nab^i\Psi_1\|^2_{L^2_{sc}(\S)}+\int_{\Hb_{\ub}^{(\ui,u)}}\f{a}{|u'|^2}\|\nab^i\Psi_2\|^2_{L^2_{sc}(S_{u',\ub})}\\
\ls&\int_{H_{\ui}^{(0,\ub)}} \|\nab^i\Psi_1\|^2_{L^2_{sc}(S_{\ui,\ub})}+\int_{\Hb_{0}^{(\ui,u)}}\f{a}{|\ui|^2}\|\nab^i\Psi_2\|^2_{L^2_{sc}(S_{\ui,\ub})}\\
&+2\int_{D_{u,\ub}} \f{a}{|u'|}\|\nab^i\Psi_1 \cdot P\|_{L^1_{sc}(S_{u',\ub'})}+2\int_{D_{u,\ub}}\f{a}{|u'|}\|\nab^i\Psi_2 \cdot Q\|_{L^1_{sc}(S_{u',\ub'})}.
\end{split}
\end{equation*}
Recall the definition in (\ref{scale invariant norms 2}) 
\begin{equation*}
\begin{split}
\|\phi\|^2_{L^2_{sc}(H_u^{(0,\underline{u})})}:=&
\int_0^{\underline{u}}\|\phi\|^2_{L^2_{sc}(S_{u,\underline{u}'})}d\underline{u}',\\
\|\phi\|^2_{L^2_{sc}(\underline{H}_{\underline{u}}^{(u_{\infty},u)})}:=&
\int_{u_{\infty}}^{u}{\frac{a}{|u'|^2}}\|\phi\|^2_{L^2_{sc}(S_{u',\underline{u}})}du',
\end{split}
\end{equation*}
we therefore arrive at 
\begin{proposition} \label{EE}
Under the assumptions of Theorem \ref{main.thm1} and the bootstrap assumptions \eqref{BA.0}, assume pair $(\nab^i\Psi_1, \nab^i\Psi_2)$ satisfying 
\begin{equation*}
\nab_3\nab^i\Psi_1+\big(\f{1+i}{2}+s_2(\Psi_1)\big)\tr\chib\nab^i\Psi_1-\M D \nab^i\Psi_2=P,
\end{equation*}
\begin{equation*}
\nab_4\nab^i\Psi_2-\M D^*\nab^{i}\Psi_1=Q,
\end{equation*}
where $\M D^*$ is the $L^2$ adjoint operator of $\M D$, then it follows
\begin{equation}\label{EE1}
\begin{split}
&\int_{H_u^{(0,\ub)}} \|\nab^i\Psi_1\|^2_{L^2_{sc}(\S)}+\int_{\Hb_{\ub}^{(\ui,u)}}\f{a}{|u'|^2}\|\nab^i\Psi_2\|^2_{L^2_{sc}(S_{u',\ub})}\\
\ls&\int_{H_{\ui}^{(0,\ub)}} \|\nab^i\Psi_1\|^2_{L^2_{sc}(S_{\ui,\ub})}+\int_{\Hb_{0}^{(\ui,u)}}\f{a}{|\ui|^2}\|\nab^i\Psi_2\|^2_{L^2_{sc}(S_{\ui,\ub})}\\
&+\int_{D_{u,\ub}} \f{a}{|u'|}\|\nab^i\Psi_1 \cdot P\|_{L^1_{sc}(S_{u',\ub'})}+\int_{D_{u,\ub}}\f{a}{|u'|}\|\nab^i\Psi_2 \cdot Q\|_{L^1_{sc}(S_{u',\ub'})}.
\end{split}
\end{equation}
\end{proposition}

\subsection{Energy Estimate in Scale Invariant Norms}
We are now ready to prove 
$$\M R+\Rb\ls \M I^{(0)}+(\M I^{(0)})^2+1.$$
Let's start with the pair $(\a, \b)$.
\begin{proposition} \label{ee 4}
Under the assumptions of Theorem \ref{main.thm1} and the bootstrap assumptions \eqref{BA.0}, we have 
\begin{equation*}
\begin{split}
&\f{1}{\at}\|(\at\nab)^i\a\|^2_{L^2_{sc}(H_u^{(0,\ub)})}+\f{1}{\at}\|(\at\nab)^i\b\|^2_{L^2_{sc}(\Hb_{\ub}^{(u_{\infty},u)})}\\
\leq&\f{1}{\at}\|(\at\nab)^i\a\|^2_{L^2_{sc}(H_{u_{\infty}}^{(0,\ub)})}+\f{1}{\at}\|(\at\nab)^i\b\|^2_{L^2_{sc}(\Hb_{0}^{(u_{\infty},u)})}+\f{1}{a^{\f13}} .
\end{split}
\end{equation*}
\end{proposition}

\begin{proof}
For pair $(\Psi_1, \Psi_2)=(\a, \b)$ we have 
$$\nab_4 \b-{\M D}^* \a=\p(\b, \a),$$
$$\nab_3 \a+\f12\tr\chib\a-\M D\b=(\p, \chih)(\Psi, \b, \a).$$
\noindent Commuting the above equations with $\nab$ for i times, we get 
\begin{equation*}
\begin{split}
&\nab_3\nab^i\a+\f{1+i}{2}\tr\chib\nab^i\a-\M D\nab^i\b\\
=&\sum_{i_1+i_2+i_3=i}\nabla^{i_1}\p^{i_2+1}\nabla^{i_3}\b+\sum_{i_1+i_2+i_3=i}\nabla^{i_1}\p^{i_2}\nabla^{i_3}\bigg((\p,\chih)(\Psi, \b, \a)\bigg)\\
&+\sum_{i_1+i_2+i_3+i_4+1=i} \nabla^{i_1}\p^{i_2+1}\nabla^{i_3}(\chibh,\tr\chib)\nab^{i_4}\a+\sum_{i_1+i_2+i_3+i_4=i}\nabla^{i_1}\p^{i_2}\nabla^{i_3}(\p, \chibh, \tc)\nabla^{i_4} \a\\
=&\sum_{i_1+i_2+i_3+i_4+1=i} \nabla^{i_1}\p^{i_2+1}\nabla^{i_3}(\chibh,\tr\chib)\nab^{i_4}\a+\sum_{i_1+i_2+i_3+i_4=i}\nabla^{i_1}\p^{i_2}\nabla^{i_3}(\p, \chih, \chibh, \tc)\nabla^{i_4} (\Psi, \b, \a)\\
=&F_1.
\end{split}
\end{equation*}
\noindent And
\begin{equation*}
\begin{split}
&\nab_4\nab^i\b-{\M D}^*\nab^i\a\\
=&\sum_{i_1+i_2+i_3=i}\nabla^{i_1}\p^{i_2+1}\nabla^{i_3}\b +\sum_{i_1+i_2+i_3=i}\nabla^{i_1}\p^{i_2}\nabla^{i_3} \bigg(\p(\b, \a)\bigg)\\
&+\sum_{i_1+i_2+i_3+i_4=i}\nabla^{i_1}\p^{i_2}\nabla^{i_3}(\p, \chih)\nabla^{i_4} \b\\
=&\sum_{i_1+i_2+i_3+i_4=i}\nabla^{i_1}\p^{i_2}\nabla^{i_3}(\p, \chih)\nabla^{i_4} (\b, \a)\\
=&G_1.
\end{split}
\end{equation*}
\noindent Applying Proposition \ref{ee 1} to Proposition \ref{ee 3}, it follows
\begin{equation*}
\begin{split}
&\|(\at\nab)^i\a\|^2_{L^2_{sc}(H_u^{(0,\ub)})}+\|(\at\nab)^i\b\|^2_{L^2_{sc}(\Hb_{\ub}^{(u_{\infty},u)})}\\
\leq & \|(\at\nab)^i\a\|^2_{L^2_{sc}(H_{u_{\infty}}^{(0,\ub)})}+\|(\at\nab)^i\b\|^2_{L^2_{sc}(\Hb_{0}^{(u_{\infty},u)})}+N_1+M_1.
\end{split}
\end{equation*}
\noindent where
\begin{equation*}
\begin{split}
N_1=&\int_0^{\ub}\int_{u_{\infty}}^u \frac{a}{|u'|}\|(\at)^iF_1\cdot(\at\nab)^i\a\|_{L^1_{sc}(S_{u',\ub'})} du'd{\ub}',\\
M_1=&\int_0^{\ub}\int_{u_{\infty}}^u \frac{a}{|u'|}\|(\at)^iG_1\cdot(\at\nab)^i\b\|_{L^1_{sc}(S_{u',\ub'})} du'd{\ub}',\\
\end{split}
\end{equation*}
By H\"older's inequalities in scale invariant norms, we have
\begin{equation*}
\begin{split}
N_1=&\int_0^{\ub}\int_{u_{\infty}}^u \frac{a}{|u'|}\|(\at)^iF_1\cdot(\at\nab)^i\a\|_{L^1_{sc}(S_{u',\ub'})} du'd{\ub}',\\
\leq&\int_0^{\ub}\int_{u_{\infty}}^u \frac{a}{|u'|^2}\|(\at)^iF_1\|_{L^2_{sc}(S_{u',\ub'})} \|(\at\nab)^i\a\|_{L^2_{sc}(S_{u',\ub'})} du'd{\ub}'\\
\leq&\int_{u_{\infty}}^u \frac{a}{|u'|^2}\bigg(\int_0^{\ub}\|(\at)^iF_1\|^2_{L^2_{sc}(S_{u',\ub'})}d\ub'\bigg)^{\f12} \|(\at\nab)^i\a\|_{L^2_{sc}(H_{u'}^{(0,\ub)})} du'.\\
\leq&\int_{u_{\infty}}^u \frac{a}{|u'|^2}\bigg(\int_0^{\ub}\|(\at)^iF_1\|^2_{L^2_{sc}(S_{u',\ub'})}d\ub'\bigg)^{\f12}du'\cdot \sup_{u'}\|(\at\nab)^i\a\|_{L^2_{sc}(H_{u'}^{(0,\ub)})},
\end{split}
\end{equation*}
where  
\begin{equation*}
\begin{split}
F_1=&\sum_{i_1+i_2+i_3+i_4+1=i} \nabla^{i_1}\p^{i_2+1}\nabla^{i_3}(\chibh,\tr\chib)\nab^{i_4}\a\\
&+\sum_{i_1+i_2+i_3+i_4=i}\nabla^{i_1}\p^{i_2}\nabla^{i_3}(\p, \chih, \chibh, \tc)\nabla^{i_4} (\Psi, \b, \a).
\end{split}
\end{equation*}
Denote
$$H_1:=\int_0^{\ub}\|(\at)^i F_1\|^2_{L^2_{sc}(S_{u',\ub'})}d\ub'.$$
We further have 
\begin{equation*}
\begin{split}
H_1=&  \int_{0}^{\ub} \|(\at)^iF_1\|^2_{L^2_{sc}(S_{u',\ub'})} d\ub'\\
\leq&  \int_{0}^{\ub} a^{-1}\|\f{\at}{|u'|}\p, \f{\at}{|u'|}\chih, \f{\at}{|u'|}\chibh, \f{\at}{|u'|}\tc\|^2_{L^{\infty}_{sc}(S_{u',\ub'})}  \|(\at\nab)^i(\Psi, \b, \a)\|^2_{L^2_{sc}(S_{u',\ub'})} d\ub'\\
&+ \sum_{i_1+i_2+i_3+i_4+1=i} \int_{0}^{\ub} \f{|u'|^4}{a^2}\|(\at)^{i+1}  \nabla^{i_1}\p^{i_2+1}\nabla^{i_3}(\f{a}{|u'|^2}\chibh,\f{a}{|u'|^2}\tr\chib)\nab^{i_4}(a^{-\f12}\a)\|^2_{L^2_{sc}(S_{u',\ub'})} d\ub'\\
&+\sum_{\substack{i_1+i_2+i_3+i_4=i\\ i_4\leq i-1}} \int_{0}^{\ub} \f{|u'|^2}{a}\|(\at)^{i+1}  \nabla^{i_1}\p^{i_2}\nabla^{i_3}\big(\f{\at}{|u'|}\p, \f{\at}{|u'|}\chih, \f{\at}{|u'|}\chibh, \f{\at}{|u'|}\tc\big)\\
&\quad\quad\quad \quad \quad\quad\quad\quad \quad \quad \quad \quad\quad\quad\quad\quad \times\nab^{i_4}(a^{-\f12}\Psi, a^{-\f12}\b, a^{-\f12}\a)\|^2_{L^2_{sc}(S_{u',\ub'})} d\ub'\\
\leq& \int_{0}^{\ub} \|(\at\nab)^i (a^{-\f12}\Psi, a^{-\f12}\b, a^{-\f12}\a)\|^2_{L^2_{sc}(S_{u',\ub'})} d\ub'\cdot \big(O^2[\chih]+O^2[\chibh]\big)\\
&+\int_0^{\ub} \f{|u'|^4}{a^2}\f{a^2}{|u'|^4} d\ub'\cdot O^{6}+\int_0^{\ub}\f{|u'|^2}{a}\f{a}{|u'|^2}\cdot O^{4}\\
\leq&R^2[\a]\cdot \big(O^2[\chih]+O^2[\chibh]\big)+O^6+O^4.
\end{split}
\end{equation*}
Therefore, for $N_1$ we have
$$N_1\ls (R[\a]\cdot O[\chih]+R[\a]\cdot O[\chibh]+O^3+O^2)\cdot \sup_{u'}\|(\at\nab)^i\a\|_{L^2_{sc}(H_{u'}^{(0,\ub)})}.$$
\noindent We treate $M_1$ in a similar way.
\begin{equation*}
\begin{split}
M_1=&\int_0^{\ub}\int_{u_{\infty}}^u \frac{a}{|u'|}\|(\at)^iG_1\cdot(\at\nab)^i\b\|_{L^1_{sc}(S_{u',\ub'})} du'd{\ub}'\\
\leq&\int_0^{\ub}\int_{u_{\infty}}^u \frac{a}{|u'|^2}\|(\at)^i G_1\|_{L^2_{sc}(S_{u',\ub'})} \|(\at\nab)^i\b\|_{L^2_{sc}(S_{u',\ub'})} du'd{\ub}'\\
\leq&\int_0^{\ub}\bigg(\int_{u_{\infty}}^u \frac{a}{|u'|^2}\|(\at)^i G_1\|^2_{L^2_{sc}(S_{u',\ub'})}du'\bigg)^{\f12} \bigg(\int_{u_{\infty}}^u \frac{a}{|u'|^2}\|(\at\nab)^i \b\|^2_{L^2_{sc}(S_{u',\ub'})}du'\bigg)^{\f12} d{\ub}'\\
\leq&\int_0^{\ub}\bigg(\int_{u_{\infty}}^u \frac{a}{|u'|^2}\|(\at)^i G_1\|^2_{L^2_{sc}(S_{u',\ub'})}du'\bigg)^{\f12}\|(\at\nab)^i\b\|_{L^2_{sc}(\Hb_{\ub'}^{(u_{\infty},u)})}d{\ub}'\\
\leq&\bigg(\int_0^{\ub}\int_{u_{\infty}}^u \frac{a}{|u'|^2}\|(\at)^i G_1\|^2_{L^2_{sc}(S_{u',\ub'})}du' d\ub'\bigg)^{\f12}\cdot \|(\at\nab)^i\b\|_{L^2_{sc}(\Hb_{\ub'}^{(u_{\infty},u)})},
\end{split}
\end{equation*}
where
\begin{equation*}
\begin{split}
G_1=&\sum_{i_1+i_2+i_3+i_4=i}\nabla^{i_1}\p^{i_2}\nabla^{i_3}(\p, \chih)\nabla^{i_4} (\b, \a).
\end{split}
\end{equation*}
Let 
$$J_1:=\int_0^{\ub}\int_{u_{\infty}}^u \frac{a}{|u'|^2}\|(\at)^i G_1\|^2_{L^2_{sc}(S_{u',\ub})}du'd\ub'.$$
Then, by (\ref{4.6})
\begin{equation*}
\begin{split}
J_1\leq& \int_0^{\ub}\int_{u_{\infty}}^u \frac{a}{|u'|^2}\|(\at)^i (\p,\chih)\nab^i(\b, \a)\|^2_{L^2_{sc}(S_{u',\ub})}du'd\ub'\\
&+\int_0^{\ub}\int_{u_{\infty}}^u \f{a}{|u'|^2} \|(\at)^i\sum_{\substack{i_1+i_2+i_3+i_4=i\\i_4\leq i-1}}\nabla^{i_1}\p^{i_2}\nabla^{i_3}(\p, \chih)\nabla^{i_4} (\b, \a)\|^2_{L^2_{sc}(S_{u',\ub})}  du'd\ub'\\
\leq& \int_0^{\ub}\int_{u_{\infty}}^u \frac{a}{|u'|^4}\|\p, \chih\|^2_{L^{\infty}_{sc}(S_{u',\ub})} \|(\at)^i\nab^i (\b, \a)\|^2_{L^2_{sc}(S_{u',\ub})}du'd\ub'\\
&+\int_0^{\ub}\int_{u_{\infty}}^u \f{a^2}{|u'|^2}\|(\at)^{i+1}\sum_{\substack{i_1+i_2+i_3+i_4=i\\i_4\leq i-1}}\nabla^{i_1}\p^{i_2}\nabla^{i_3}(a^{-\f12}\p, a^{-\f12}\chih)\nabla^{i_4} (a^{-\f12}\b, a^{-\f12}\a)\|^2_{L^2_{sc}(S_{u',\ub})}  du'd\ub'\\
\leq& \int_0^{\ub}\int_{u_{\infty}}^u \frac{a^2}{|u'|^4}\|a^{-\f12}\p, a^{-\f12}\chih\|^2_{L^{\infty}_{sc}(S_{u',\ub})} \|(\at)^i\nab^i (\b, \a)\|^2_{L^2_{sc}(S_{u',\ub})}du'd\ub'\\
&+\int_0^{\ub}\int_{u_{\infty}}^u \f{a^2}{|u'|^2}\|(\at)^{i+1}\sum_{\substack{i_1+i_2+i_3+i_4=i\\i_4\leq i-1}}\nabla^{i_1}\p^{i_2}\nabla^{i_3}(a^{-\f12}\p, a^{-\f12}\chih)\nabla^{i_4} (a^{-\f12}\b, a^{-\f12}\a)\|^2_{L^2_{sc}(S_{u',\ub})}  du'd\ub'\\
\leq&O^2\cdot \sup_{u'}\int_0^{\ub}\|(\at)^i\nab^i (\b, \a)\|^2_{L^2_{sc}(S_{u',\ub})}d\ub'\cdot \int_{u_{\infty}}^{u}\f{a^2}{|u'|^4}du'+ \int_{u_{\infty}}^{u} \f{a^3}{|u'|^4}du'\cdot O^{4}\\
\leq&a^{-1}\sup_{u'}\|(\at\nab)^i(\b,\a)\|^2_{L^2_{sc}(H_{u'}^{(0,\ub)})}\cdot O^2+O^{4}\leq R^2\cdot O^2+O^4.
\end{split}
\end{equation*}
Hence, 
\begin{equation*}
\begin{split}
M_1\leq&\bigg(\int_0^{\ub}\int_{u_{\infty}}^u \frac{a}{|u'|^2}\|(\at)^i G\|^2_{L^2_{sc}(S_{u',\ub'})}du' d\ub'\bigg)^{\f12}\cdot \|(\at\nab)^i\b\|_{L^2_{sc}(\Hb_{\ub'}^{(u_{\infty},u)})}\\
\leq& (R\cdot O+O^2)\cdot \|(\at\nab)^i\b\|_{L^2_{sc}(\Hb_{\ub'}^{(u_{\infty},u)})}\\
\end{split}
\end{equation*}
Combining all the above estimates together, we obtain
\begin{equation*}
\begin{split}
&\|(\at\nab)^i\a\|^2_{L^2_{sc}(H_u^{(0,\ub)})}+\|(\at\nab)^i\b\|^2_{L^2_{sc}(\Hb_{\ub}^{(u_{\infty},u)})}\\
\leq & \|(\at\nab)^i\a\|^2_{L^2_{sc}(H_{u_{\infty}}^{(0,\ub)})}+\|(\at\nab)^i\b\|^2_{L^2_{sc}(\Hb_{0}^{(u_{\infty},u)})}+N_1+M_1\\
\leq& \|(\at\nab)^i\a\|^2_{L^2_{sc}(H_{u_{\infty}}^{(0,\ub)})}+\|(\at\nab)^i\b\|^2_{L^2_{sc}(\Hb_{0}^{(u_{\infty},u)})}\\
&+(R\cdot O+O^3+O^2)\cdot \sup_{u'}\|(\at\nab)^i\a\|_{L^2_{sc}(H_{u'}^{(0,\ub)})}\\
&+\at(R\cdot O+O^2)\cdot a^{-\f12}\|(\at\nab)^i\b\|_{L^2_{sc}(\Hb_{\ub'}^{(u_{\infty},u)})}.
\end{split}
\end{equation*}
Hence, for sufficiently large $a$, we have
\begin{equation*}
\begin{split}
&a^{-1}\|(\at\nab)^i\a\|^2_{L^2_{sc}(H_u^{(0,\ub)})}+a^{-1}\|(\at\nab)^i\b\|^2_{L^2_{sc}(\Hb_{\ub}^{(u_{\infty},u)})}\\
\leq& a^{-1}\|(\at\nab)^i\a\|^2_{L^2_{sc}(H_{u_{\infty}}^{(0,\ub)})}+a^{-1}\|(\at\nab)^i\b\|^2_{L^2_{sc}(\Hb_{0}^{(u_{\infty},u)})}\\
&+a^{-\f12}(R\cdot O+O^{3}+O^2)\cdot \sup_{u'}a^{-\f12}\|(\at\nab)^i\a\|_{L^2_{sc}(H_{u'}^{(0,\ub)})}\\
&+a^{-\f12}(R\cdot O+O^2)\cdot a^{-\f12}\|(\at\nab)^i\b\|_{L^2_{sc}(\Hb_{\ub'}^{(u_{\infty},u)})}\\
\leq& a^{-1}\|(\at\nab)^i\a\|^2_{L^2_{sc}(H_{u_{\infty}}^{(0,\ub)})}+a^{-1}\|(\at\nab)^i\b\|^2_{L^2_{sc}(\Hb_{0}^{(u_{\infty},u)})}\\
&+a^{-\f12}(R\cdot O+O^3+O^2)\cdot R+a^{-\f12}(R\cdot O+O^2)\cdot R\\
\leq& a^{-1}\|(\at\nab)^i\a\|^2_{L^2_{sc}(H_{u_{\infty}}^{(0,\ub)})}+a^{-1}\|(\at\nab)^i\b\|^2_{L^2_{sc}(\Hb_{0}^{(u_{\infty},u)})}+\f{1}{a^{\f13}}.
\end{split}
\end{equation*}
This further implies
\begin{equation}\label{eee 4}
\M R^2[\a]+\Rb^2[\b]\leq \M R^2_0[\a]+\Rb^2_0[\b]+\f{1}{a^{\f14}}  \quad \mbox{ and }
\end{equation}
\begin{equation}\label{eee 5}
\M R[\a]+\Rb[\b]\leq 2\,\M R_0[\a]+2\,\Rb_0[\b]+\f{1}{a^{\f18}}. 
\end{equation}
\end{proof}

We next derive estimates for other pairs. 
\begin{proposition} \label{ee 5}
Under the assumptions of Theorem \ref{main.thm1} and the bootstrap assumptions \eqref{BA.0}, we have
\begin{equation*}
\begin{split}
&\|(\at\nab)^i\b\|^2_{L^2_{sc}(H_u^{(0,\ub)})}+\|(\at\nab)^i\rho\|^2_{L^2_{sc}(H_u^{(0,\ub)})}\\
&\quad \quad \quad \quad \quad+\|(\at\nab)^i\sigma\|^2_{L^2_{sc}(H_u^{(0,\ub)})}+\|(\at\nab)^i\beb\|^2_{L^2_{sc}(H_u^{(0,\ub)})}\ls \M I^{(0)}+(\M I^{(0)})^2+1,
\end{split}
\end{equation*}

\begin{equation*}
\begin{split}
&\|(\at\nab)^i\rho\|^2_{L^2_{sc}(\Hb_{\ub}^{(u_{\infty},u)})}+\|(\at\nab)^i\sigma\|^2_{L^2_{sc}(\Hb_{\ub}^{(u_{\infty},u)})}\\
&\quad \quad \quad \quad \quad+\|(\at\nab)^i\beb\|^2_{L^2_{sc}(\Hb_{\ub}^{(u_{\infty},u)})}+\|(\at\nab)^i\ab\|^2_{L^2_{sc}(\Hb_{\ub}^{(u_{\infty},u)})}
\ls \M I^{(0)}+(\M I^{(0)})^2+1.
\end{split}
\end{equation*}
\end{proposition}
\begin{proof} For $(\Psi_1, \Psi_2)\in \{\big(\b, (\rho, \sigma)\big), \big( (\rho, \sigma), \beb\big), (\beb, \ab)\}$ and
$\Psi\in \{\b, \rho, \sigma, \beb, \ab\}$, we have the systematic null Bianchi equations:
$$\nab_3 \Psi_1+\bigg(\f{1}{2}+s_2(\Psi_1)\bigg)\tr\chib\Psi_1-\M D\Psi_2=(\p, \chih)\Psi,$$
$$\nab_4 \Psi_2-{\M D}^* \Psi_1=\p\Psi+\chih(\Psi, \a).$$
Commuting these equations with $\nab$ for i times, we have 
\begin{equation*}
\begin{split}
&\nab_3\nab^i\Psi_1+\bigg(\f{1+i}{2}+s_2(\Psi_1)\bigg)\tr\chib\nab^i\Psi_1-\M D\nab^{i}\Psi_2\\
=&\sum_{i_1+i_2+i_3+1=i}\nabla^{i_1}\p^{i_2+1}\nabla^{i_3}\nab\Psi+\sum_{i_1+i_2+i_3=i}\nabla^{i_1}\p^{i_2}\nabla^{i_3} \bigg((\p,\chih)\Psi\bigg)\\
&+\sum_{i_1+i_2+i_3+i_4+1=i} \nabla^{i_1}\p^{i_2+1}\nabla^{i_3}(\chibh,\tr\chib)\nab^{i_4}\Psi+\sum_{i_1+i_2+i_3+i_4=i}\nabla^{i_1}\p^{i_2}\nabla^{i_3}(\p, \chibh, \tc)\nabla^{i_4} \Psi\\
=&\sum_{i_1+i_2+i_3+i_4+1=i} \nabla^{i_1}\p^{i_2+1}\nabla^{i_3}(\chibh,\tr\chib)\nab^{i_4}\Psi+\sum_{i_1+i_2+i_3+i_4=i}\nabla^{i_1}\p^{i_2}\nabla^{i_3}(\p, \chih, \chibh, \tc)\nabla^{i_4} \Psi\\
=&F_2.\\
\end{split}
\end{equation*}
\noindent And
\begin{equation*}
\begin{split}
&\nab_4\nab^i\Psi_2-{\M D}^*\nab^{i}\Psi_1\\
=&\sum_{i_1+i_2+i_3+1=i}\nabla^{i_1}\p^{i_2+1}\nabla^{i_3}\nab\Psi+\sum_{i_1+i_2+i_3=i}\nabla^{i_1}\p^{i_2}\nabla^{i_3} \bigg(\p\Psi+\chih(\Psi, \a)\bigg)+\sum_{i_1+i_2+i_3+i_4=i}\nabla^{i_1}\p^{i_2}\nabla^{i_3}(\p, \chih)\nabla^{i_4} \Psi\\
=&\sum_{i_1+i_2+i_3+i_4=i}\nabla^{i_1}\p^{i_2}\nabla^{i_3}(\p, \chih, \chibh)\nabla^{i_4} (\Psi, \a)\\
=&G_2.
\end{split}
\end{equation*}
Applying Proposition \ref{ee 1} and Proposition \ref{ee 2}, we have
\begin{equation}\label{N2+M2}
\begin{split}
&\|(\at\nab)^i\Psi_1\|^2_{L^2_{sc}(H_u^{(0,\ub)})}+\|(\at\nab)^i\Psi_2\|^2_{L^2_{sc}(\Hb_{\ub}^{(u_{\infty},u)})}\\
\leq & \|(\at\nab)^i\Psi_1\|^2_{L^2_{sc}(H_{u_{\infty}}^{(0,\ub)})}+\|(\at\nab)^i\Psi_2\|^2_{L^2_{sc}(\Hb_{0}^{(u_{\infty},u)})}+N_2+M_2.
\end{split}
\end{equation}
\noindent Here
\begin{equation*}
\begin{split}
N_2=&\int_0^{\ub}\int_{u_{\infty}}^u \frac{a}{|u'|}\|(\at)^iF_2\cdot(\at\nab)^i\Psi_1\|_{L^1_{sc}(S_{u',\ub'})} du'd{\ub}',\\
M_2=&\int_0^{\ub}\int_{u_{\infty}}^u \frac{a}{|u'|}\|(\at)^iG_2\cdot(\at\nab)^i\Psi_2\|_{L^1_{sc}(S_{u',\ub'})} du'd{\ub}'.\\
\end{split}
\end{equation*}
Employing (\ref{Holder's}) and bootstrap assumption (\ref{BA.0}), we hence have 
\begin{equation*}
\begin{split}
N_2=&\int_0^{\ub}\int_{u_{\infty}}^u \frac{a}{|u'|}\|(\at)^iF_2\cdot(\at\nab)^i\Psi\|_{L^1_{sc}(S_{u',\ub'})} du'd{\ub}',\\
\leq&\int_0^{\ub}\int_{u_{\infty}}^u \frac{a}{|u'|^2}\|(\at)^iF_2\|_{L^2_{sc}(S_{u',\ub'})} \|(\at\nab)^i\Psi\|_{L^2_{sc}(S_{u',\ub'})} du'd{\ub}'\\
\leq&\int_{u_{\infty}}^u \frac{a}{|u'|^2}\bigg(\int_0^{\ub}\|(\at)^iF_2\|^2_{L^2_{sc}(S_{u',\ub'})}d\ub'\bigg)^{\f12} \|(\at\nab)^i\Psi\|_{L^2_{sc}(H_{u'}^{(0,\ub)})} du'.\\
\leq& \int_{u_{\infty}}^u \frac{a}{|u'|^2}\sup_{\ub'}\|(\at)^iF_2\|_{L^2_{sc}(S_{u',\ub'})}\cdot\sup_{\ub}\|(\at\nab)^i\Psi\|_{L^2_{sc}(H_{u'}^{(0,\ub)})} du'\\
\leq&  \int_{u_{\infty}}^u \frac{a}{|u'|^2}\sup_{\ub'}\|(\at)^iF_2\|_{L^2_{sc}(S_{u',\ub'})} du'\cdot R,
\end{split}
\end{equation*}
where $R$ is bootstrap constant for curvature estimates. And recall $F_2$ is of the form:
\begin{equation*}
\begin{split}
F_2=&\sum_{i_1+i_2+i_3+i_4+1=i} \nabla^{i_1}\p^{i_2+1}\nabla^{i_3}(\chibh,\tr\chib)\nab^{i_4}\Psi+\sum_{i_1+i_2+i_3+i_4=i}\nabla^{i_1}\p^{i_2}\nabla^{i_3}(\p, \chih, \chibh, \tc)\nabla^{i_4} \Psi\\
\end{split}
\end{equation*}
Thus, by using (\ref{Holder's}), (\ref{4.6}), (\ref{4.7}) and letting $a$ to be sufficiently large, we have
\begin{equation}\label{N2}
\begin{split}
&N_2\leq  \int_{u_{\infty}}^u \frac{a}{|u'|^2}\sup_{\ub'}\|(\at)^iF\|_{L^2_{sc}(S_{u',\ub'})} du'\cdot R\\
\leq&  \int_{u_{\infty}}^u \frac{a^{\f12}}{|u'|^2}\sup_{\ub'}\|\f{\at}{|u'|}\p, \f{\at}{|u'|}\chih, \f{\at}{|u'|}\chibh, \f{\at}{|u'|}\tc\|_{L^{\infty}_{sc}(S_{u',\ub'})}  \|(\at\nab)^i\Psi\|_{L^2_{sc}(S_{u',\ub'})} du'\cdot R\\
&+ \sum_{i_1+i_2+i_3+i_4+1=i} \int_{u_{\infty}}^u a^{-\f12}\sup_{\ub'}\|(\at)^{i+1}  \nabla^{i_1}\p^{i_2+1}\nabla^{i_3}(\f{a}{|u'|^2}\chibh,\f{a}{|u'|^2}\tr\chib)\nab^{i_4}\Psi \|_{L^2_{sc}(S_{u',\ub'})} du'\cdot R\\
+&\sum_{\substack{i_1+i_2+i_3+i_4=i\\i_4\leq i-1}} \int_{u_{\infty}}^u \f{1}{|u'|}\sup_{\ub'}\|(\at)^{i+1}  \nabla^{i_1}\p^{i_2}\nabla^{i_3}\big(\f{\at}{|u'|}\p, \f{\at}{|u'|}\chih, \f{\at}{|u'|}\chibh, \f{\at}{|u'|}\tc\big)\nab^{i_4}\Psi \|_{L^2_{sc}(S_{u',\ub'})} du'\cdot R\\
\leq& \sup_{\ub'}\bigg(\int^u_{u_{\infty}}\f{1}{|u'|^2}du' \bigg)^{\f12}\cdot O\cdot   \sup_{\ub'}\bigg(\int_{u_{\infty}}^u \frac{a}{|u'|^2}\|(\at\nab)^i\Psi\|^2_{L^2_{sc}(S_{u',\ub'})} du'\bigg)^{\f12}\cdot R\\
&+a^{-\f12}\int^u_{u_{\infty}}\f{a}{|u'|^2} du'\cdot O^{3}\cdot R+\int_{u_{\infty}}^u \f{a^{\f12}}{|u'|^2}\cdot O^{2}\cdot du'\cdot R.\\
\leq&\f{1}{|u|^{\f12}}\cdot O \cdot R \cdot R+a^{-\f12}\cdot \f{a}{|u|} \cdot O^{3}\cdot R+\f{\at}{|u|}\cdot O^{2} \cdot R\leq \f{1}{a^{\f13}}.
\end{split}
\end{equation}
We then treat $M_2$ in the same fashion. By (\ref{Holder's}) and H\"older's inequality, we obtain
\begin{equation*}
\begin{split}
M_2=&\int_0^{\ub}\int_{u_{\infty}}^u \frac{a}{|u'|}\|(\at)^iG_2\cdot(\at\nab)^i\Psi\|_{L^1_{sc}(S_{u',\ub'})} du'd{\ub}'\\
\leq&\int_0^{\ub}\int_{u_{\infty}}^u \frac{a}{|u'|^2}\|(\at)^i G_2\|_{L^2_{sc}(S_{u',\ub'})} \|(\at\nab)^i\Psi\|_{L^2_{sc}(S_{u',\ub'})} du'd{\ub}'\\
\leq&\int_0^{\ub}\bigg(\int_{u_{\infty}}^u \frac{a}{|u'|^2}\|(\at)^i G_2\|^2_{L^2_{sc}(S_{u',\ub'})}du'\bigg)^{\f12} \bigg(\int_{u_{\infty}}^u \frac{a}{|u'|^2}\|(\at\nab)^i \Psi\|^2_{L^2_{sc}(S_{u',\ub'})}du'\bigg)^{\f12} d{\ub}'\\
\leq&\bigg(\int_0^{\ub}\int_{u_{\infty}}^u \frac{a}{|u'|^2}\|(\at)^i G_2\|^2_{L^2_{sc}(S_{u',\ub'})}du' d\ub'\bigg)^{\f12}\cdot \|(\at\nab)^i\Psi\|_{L^2_{sc}(\Hb_{\ub'}^{(u_{\infty},u)})},
\end{split}
\end{equation*}
where
\begin{equation*}
\begin{split}
G_2=&\sum_{i_1+i_2+i_3+i_4=i}\nabla^{i_1}\p^{i_2}\nabla^{i_3}(\p, \chibh, \chih)\nabla^{i_4} (\Psi, \a).
\end{split}
\end{equation*}
Denote
$$J_2:=\int_0^{\ub}\int_{u_{\infty}}^u \frac{a}{|u'|^2}\|(\at)^i G_2\|^2_{L^2_{sc}(S_{u',\ub})}du'd\ub'.$$
Then by (\ref{Holder's}), (\ref{4.3}), (\ref{4.4}), Proposition \ref{a.bd}, Proposition \ref{ee 4} and letting $a$ to be sufficiently large, we have
\begin{equation*}
\begin{split}
J_2\leq& \int_0^{\ub}\int_{u_{\infty}}^u \frac{a}{|u'|^2}\|(\at)^i (\p,\chih, \chibh)\nab^i(\Psi, \a)\|^2_{L^2_{sc}(S_{u',\ub})}du'd\ub'\\
&+\int_0^{\ub}\int_{u_{\infty}}^u \f{a}{|u'|^2} \|(\at)^i\sum_{\substack{i_1+i_2+i_3+i_4=i\\i_4\leq i-1}}\nabla^{i_1}\p^{i_2}\nabla^{i_3}(\p, \chih, \chibh)\nabla^{i_4} (\Psi, \a)\|^2_{L^2_{sc}(S_{u',\ub})}  du'd\ub'\\
\leq& \int_0^{\ub}\int_{u_{\infty}}^u \frac{a}{|u'|^4}\|\p, \chih, \chibh\|^2_{L^{\infty}_{sc}(S_{u',\ub})} \|(\at)^i\nab^i (\Psi, \a)\|^2_{L^2_{sc}(S_{u',\ub})}du'd\ub'\\
&+\int_0^{\ub}\int_{u_{\infty}}^u \|(\at)^{i+1}\sum_{\substack{i_1+i_2+i_3+i_4=i\\i_4\leq i-1}}\nabla^{i_1}\p^{i_2}\nabla^{i_3}(\f{\at}{|u'|}\p,  \f{\at}{|u'|}\chih,\f{\at}{|u'|}\chibh)\nabla^{i_4} (a^{-\f12}\Psi, a^{-\f12}\a)\|^2_{L^2_{sc}(S_{u',\ub})}  du'd\ub'\\
\leq& \int_0^{\ub}\int_{u_{\infty}}^u a^{-1}\|\f{\at}{|u'|}\p,  \f{\at}{|u'|}\chih, \f{\at}{|u'|}\chibh\|^2_{L^{\infty}_{sc}(S_{u',\ub})} \f{a}{|u'|^2} \|(\at)^i\nab^i\Psi\|^2_{L^2_{sc}(S_{u',\ub})}du'd\ub'\\
&+ \int_0^{\ub}\int_{u_{\infty}}^u \|\f{\at}{|u'|} \p,  \f{\at}{|u'|}\chih, \f{\at}{|u'|} \chibh\|^2_{L^{\infty}_{sc}(S_{u',\ub})} \f{a}{|u'|^2} \|(\at)^i\nab^i (a^{-\f12}\a)\|^2_{L^2_{sc}(S_{u',\ub})}du'd\ub'\\
&+ \int_0^{\ub}\int_{u_{\infty}}^u \bigg(\f{a}{|u'|^2}O^2[\chibh]O^2[\a]+\f{a^{-\f12}\cdot a\cdot O^4}{|u'|^2}\bigg) du' d\ub' \\
\leq& \int_0^{\ub} a^{-1}\sup_{u'}\bigg(\|\f{\at}{|u'|}\p,  \f{\at}{|u'|}\chih, \f{\at}{|u'|}\chibh\|^2_{L^{\infty}_{sc}(S_{u',\ub})}\bigg) \bigg(\int_{u_{\infty}}^u  \f{a}{|u'|^2} \|(\at)^i\nab^i\Psi\|^2_{L^2_{sc}(S_{u',\ub})}du'\bigg)d\ub'\\
&+\int_{u_{\infty}}^u \sup_{u'}\bigg(\|\f{\at}{|u'|} \p,  \f{\at}{|u'|}\chih, \f{\at}{|u'|} \chibh\|^2_{L^{\infty}_{sc}(S_{u',\ub})}\bigg) \f{a}{|u'|^2} \bigg(\int_0^{\ub}\|(\at)^i\nab^i (a^{-\f12}\a)\|^2_{L^2_{sc}(S_{u',\ub})}d\ub'\bigg) du'\\
&+\f{a}{|u|}\bigg(O^2[\chibh]+O^2[\chih]+1\bigg)O^2[\a]+\f{\at}{|u|}\cdot O^4\\
\leq&a^{-1}\cdot O^2 \cdot R^2+\bigg(O^2[\chibh]+O^2[\chih]+1\bigg) \cdot \big(\Rb^2[\b]+\f{1}{a}\big)\\
&+\f{a}{|u|}\bigg(O^2[\chibh]+O^2[\chih]+1\bigg) \cdot \bigg(\M R^2[\a]+O^2[\a]\bigg)+\f{\at}{|u|}\cdot O^4\\
\ls& \f{1}{a^{\f13}}+(1+\M R^2[\a])\cdot(\Rb^2[\b]+\M R^2[\a]+1),
\end{split}
\end{equation*}
where we use $\M O[\chibh]\ls 1$, $\M O[\chih]\ls \M R[\a]+1$ and $O[\a]\ls \Rb[\b]+1$. 
This implies
$$M_2\leq J_2^{\f12}\|(\at\nab)^i\Psi\|_{L^2_{sc}(\Hb_{\ub}^{(u_{\infty},u)})}\leq J_2+\f{1}{4}\|(\at\nab)^i\Psi\|^2_{L^2_{sc}(\Hb_{\ub}^{(u_{\infty},u)})}.$$
From (\ref{N2}), we have $N_2\leq 1/a^{\f13}$. Together with (\ref{N2+M2}), we have
\begin{equation*}
\begin{split}
&\|(\at\nab)^i\Psi_1\|^2_{L^2_{sc}(H_u^{(0,\ub)})}+\|(\at\nab)^i\Psi_2\|^2_{L^2_{sc}(\Hb_{\ub}^{(u_{\infty},u)})}\\
\leq & \|(\at\nab)^i\Psi_1\|^2_{L^2_{sc}(H_{u_{\infty}}^{(0,\ub)})}+\|(\at\nab)^i\Psi_2\|^2_{L^2_{sc}(\Hb_{0}^{(u_{\infty},u)})}+N_2+M_2\\
\leq & \|(\at\nab)^i\Psi_1\|^2_{L^2_{sc}(H_{u_{\infty}}^{(0,\ub)})}+\|(\at\nab)^i\Psi_2\|^2_{L^2_{sc}(\Hb_{0}^{(u_{\infty},u)})}+\f{1}{a^{\f13}}+J_2+\f{1}{4}\cdot\|(\at\nab)^i\Psi\|^2_{L^2_{sc}(\Hb_{\ub}^{(u_{\infty},u)})}.
\end{split}
\end{equation*}
The last term $1/4 \|(\at\nab)^i\Psi\|^2_{L^2_{sc}(\Hb_{\ub}^{(u_{\infty},u)})}$ could be absorbed by the left. Recall 
$$J_2\ls  \f{1}{a^{\f13}}+(1+\M R^2[\a])\cdot(\Rb^2[\b]+\M R^2[\a]+1).$$
We hence derive
\begin{equation*}
\begin{split}
&\M R^2[\b]+\M R^2[\rho]+\M R^2[\sigma]+\M R^2[\beb]+\Rb^2[\rho]+\Rb^2[\sigma]+\Rb^2[\beb]+\Rb^2[\ab]\\
\ls& \M R^2_0[\b]+\M R^2_0[\rho]+\M R^2_0[\sigma]+\M R^2_0[\beb]+\Rb^2_0[\rho]+\Rb^2_0[\sigma]+\Rb^2_0[\beb]+\Rb^2_0[\ab]\\
&+\big(1+\M R^2_0[\a]+\Rb^2_0[\b]\big)^2+\f{1}{a^{\f14}}\ls \M (I^{(0)})^2+(\M I^{(0)})^4+1+\f{1}{a^{\f14}},
\end{split}
\end{equation*}
where we use (\ref{eee 4}) in the last step. This implies
\begin{equation}\label{eee 6}
\begin{split}
&\M R[\b]+\M R[\rho]+\M R[\sigma]+\M R[\beb]+\Rb[\rho]+\Rb[\sigma]+\Rb[\beb]+\Rb[\ab]\ls \M I^{(0)}+(\M I^{(0)})^2+1. 
\end{split}
\end{equation}
Recall (\ref{eee 5})
$$\M R[\a]+\Rb[\b]\leq 2\,\M R_0[\a]+2\,\Rb_0[\b]+\f{1}{a^{\f18}}.$$
These together conclude
\begin{equation}\label{bootstrap curvature conclusion}
\M R(u,\ub)+\Rb(u,\ub)\ls \M I^{(0)}+(\M I^{(0)})^2+1.
\end{equation}
By estimates in Section \ref{secRicci} and Section \ref{secCurvatureL2}, we have
$$\M O(u,\ub)\ls \M R(u,\ub)+\Rb(u,\ub)+1.$$
Hence, with (\ref{bootstrap curvature conclusion}) we also deduce
\begin{equation}\label{bootstrap Ricci conclusion}
\M O(u,\ub)\ls \M I^{(0)}+(\M I^{(0)})^2+1.
\end{equation}

\begin{remark}\label{bootstrap openness}
Conclusions in (\ref{bootstrap curvature conclusion}) and (\ref{bootstrap Ricci conclusion}) are improvements of bootstrap assumption (\ref{BA.0}): 
$$\M O(u, \ub)\leq O, \quad  \M R(u,\ub) +\underline{\M R}(u,\ub) \leq R, \mbox{ where } O \mbox{ and } R \mbox{ are large numbers satisfying}$$
$$\M I^{(0)}+(\M I^{(0)})^2+1\ll O, \quad \M I^{(0)}+(\M I^{(0)})^2+1\ll R, \quad  (O+R)^{20}\leq {a^{\f{1}{16}}}.$$
And these improvements prove openness in the bootstrap argument. 
\end{remark}
\end{proof}

\section{FORMATION OF TRAPPED SURFACES}\label{TSF}

In this section, we will prove\\

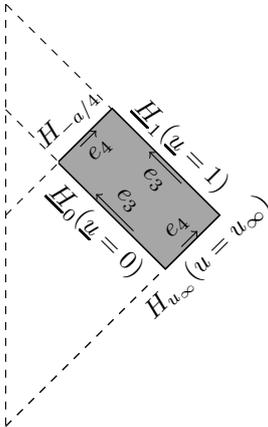
\begin{minipage}[!t]{0.27\textwidth}
\begin{tikzpicture}[scale=0.70]
\draw [white](3,-1)-- node[midway, sloped, below,black]{$H_{u_{\infty}}(u=u_{\infty})$}(4,0);
\draw [white](1,1)-- node[midway,sloped,above,black]{$H_{-a/4}$}(2,2);
\draw [white](2,2)--node [midway,sloped,above,black] {$\Hb_{1}(\ub=1)$}(4,0);
\draw [white](1,1)--node [midway,sloped, below,black] {$\Hb_{0}(\ub=0)$}(3,-1);
\draw [dashed] (0, 4)--(0, -4);
\draw [dashed] (0, -4)--(4,0)--(0,4);
\draw [dashed] (0,0)--(2,2);
\draw [dashed] (0,-4)--(4,0);
\draw [dashed] (0,2)--(3,-1);
\draw [very thick] (1,1)--(3,-1)--(4,0)--(2,2)--(1,1);
\fill[black!35!white] (1,1)--(3,-1)--(4,0)--(2,2)--(1,1);
\draw [->] (3.3,-0.6)-- node[midway, sloped, above,black]{$e_4$}(3.6,-0.3);
\draw [->] (1.4,1.3)-- node[midway, sloped, below,black]{$e_4$}(1.7,1.6);
\draw [->] (3.3,0.6)-- node[midway, sloped, below,black]{$e_3$}(2.7,1.2);
\draw [->] (2.4,-0.3)-- node[midway, sloped, above,black]{$e_3$}(1.7,0.4);
\end{tikzpicture}
\end{minipage}
\hspace{0.01\textwidth} 
\begin{minipage}[!t]{0.63\textwidth}

{\bf Theorem \ref{main.thm2} }

Given $\M I^{(0)}$, there exists a sufficiently large $a_0=a_0(\M I^{(0)})$.  For $0<a_0<a$, for Einstein vacuum equations with initial data:
\begin{itemize}
\item $\sum_{i\leq 10, k\leq 3}a^{-\frac12}\|\nab^{k}_4(|u_{\infty}|\nab)^{i}\chih_0\|_{L^{\infty}(S_{u_{\infty},\ub})}\leq \M I^{(0)}$ 

\noindent along $u=u_{\infty}$,

\item Minkowskian initial data  along $\ub=0$,

\item $\int_0^1|u_{\infty}|^2|\chih_0|^2(u_{\infty}, \ub')d\ub'\geq a$ for every direction   

\noindent along $u=u_{\infty}$,
\end{itemize}
we have that $S_{-a/4,1}$ is a trapped surface.
\end{minipage}

Proof. We first derive pointwise estimates for $|\chih|^2_{\gamma}$. 
Fix $(\theta^1, \theta^2)\in S^2$. We consider the following null structure equation 
\begin{equation*}
 \nab_3\chih+\frac 12 \tr\chib \chih-2\omegab \chih=\nab\widehat{\otimes} \eta-\frac 12 \tr\chi \chibh +\eta\widehat{\otimes} \eta.
\end{equation*}
We contract this $2$-tensor with another $2$-tensor $\chih$ and get
\begin{equation}\label{chih square}
\f12 \nab_3|\chih|^2_{\gamma}+\f12\tr\chib |\chih|^2_{\gamma}-2\omb|\chih|^2_{\gamma}=\chih(\nab\widehat{\otimes}\eta-\f12\tr\chi\chibh+\eta\widehat{\otimes}\eta).
\end{equation}
Employing the fact $\omb=-\f12\nab_3(\log \Omega)=-\f12\O^{-1}\nab_3\O$, we rewrite (\ref{chih square}) as 
\begin{equation*}
\begin{split} 
\nab_3(\O^2|\chih|^2_{\gamma})+\O^2\tr\chib|\chih|^2_{\gamma}=2\O^2\chih(\nab\widehat{\otimes}\eta-\f12\tr\chi\chibh+\eta\widehat{\otimes}\eta).
\end{split}
\end{equation*}
Using $\nab_3=\f{1}{\O}(\f{\partial}{\partial u}+b^A\f{\partial}{\partial \theta^A})$, we rewrite the above equation as
\begin{equation*}
\begin{split} 
\f{\partial}{\partial u}(\O^2|\chih|^2_{\gamma})+\O\tr\chib\cdot \O^2|\chih|^2_{\gamma}=&2\O^3\chih(\nab\widehat{\otimes}\eta-\f12\tr\chi\chibh+\eta\widehat{\otimes}\eta)-b^A\f{\partial}{\partial \theta^A}(\O^2|\chih|^2_{\gamma}).
\end{split}
\end{equation*}
Substitute $\O\tr\chib$ with
$$\Omega \tr\chib=\Omega(\tr\chib+\f{2}{|u|})-\Omega\f{2}{|u|}=\Omega(\tr\chib+\f{2}{|u|})-(\Omega-1)\f{2}{|u|}-\f{2}{|u|}$$
we have
\begin{equation*}
\begin{split} 
\f{\partial}{\partial u}(\O^2|\chih|^2_{\gamma})-\f{2}{|u|}\O^2|\chih|^2_{\gamma}=&2\O^3\chih(\nab\widehat{\otimes}\eta-\f12\tr\chi\chibh+\eta\widehat{\otimes}\eta)-b^A\f{\partial}{\partial \theta^A}(\O^2|\chih|^2_{\gamma})\\
&-\O(\tr\chib+\f{2}{|u|})(\O^2|\chih|^2_{\gamma})+(\O-1)\cdot\f{2}{|u|}\cdot(\O^2|\chih|^2_{\gamma}).
\end{split}
\end{equation*}
This gives
\begin{equation}\label{chih square 2}
\begin{split} 
\f{\partial}{\partial u}\bigg(u^2\O^2|\chih|^2_{\gamma}\bigg)=&2\cdot|u|^2\cdot\O^3\chih(\nab\widehat{\otimes}\eta-\f12\tr\chi\chibh+\eta\widehat{\otimes}\eta)-|u|^2\cdot b^A\f{\partial}{\partial \theta^A}(\O^2|\chih|^2_{\gamma})\\
&-|u|^2\cdot\O(\tr\chib+\f{2}{|u|})(\O^2|\chih|^2_{\gamma})+|u|^2\cdot(\O-1)\cdot\f{2}{|u|}\cdot(\O^2|\chih|^2_{\gamma}).
\end{split}
\end{equation}
For $b^{A}$, we have equation
\begin{equation*}
\f{\partial b^{A}}{\partial \ub}=-4\Omega^2\zeta^{A},
\end{equation*}
which is from 
\begin{equation*}
[L,\Lb]=\f{\partial b^{A}}{\partial \ub}\f{\partial}{\partial \theta^{A}}.
\end{equation*}
Applying the identity $\zeta_A=\f12\eta_A-\f12\etb_A$, Propositions \ref{Omega}, derived estimates of $\eta, \etb$, it holds in $D_{u,\ub}$
\begin{equation*}
\|b^{A}\|_{L^{\infty}(\S)}\leq \f{\at}{|u|^2}.
\end{equation*}
For the right hand side of (\ref{chih square 2}), we have
$$\|2\cdot|u|^2\cdot\O^3\chih(\nab\widehat{\otimes}\eta-\f12\tr\chi\chibh+\eta\widehat{\otimes}\eta)\|_{L^{\infty}(\S)}\leq |u|^2\cdot\f{\at}{|u|}\cdot(\f{\at}{|u|^3}+\f{a}{|u|^4})\leq\f{a}{|u|^2},$$
$$\||u|^2\cdot b^A\f{\partial}{\partial \theta^A}(\O^2|\chih|^2_{\gamma})\|_{L^{\infty}(\S)}\leq
|u|^2\cdot\f{\at}{|u|^2}\cdot\f{a}{|u|^2}\leq\f{a^{\f32}}{|u|^2},$$
$$\|-|u|^2\cdot\O(\tr\chib+\f{2}{|u|})(\O^2|\chih|^2_{\gamma})\|_{L^{\infty}(\S)}\leq
|u|^2\cdot\f{1}{|u|^2}\cdot\f{a}{|u|^2}\leq\f{a}{|u|^2},$$
$$\||u|^2\cdot(\O-1)\cdot\f{2}{|u|}\cdot(\O^2|\chih|^2_{\gamma}) \|_{L^{\infty}(\S)}\leq
|u|^2\cdot\f{1}{|u|}\cdot\f{2}{|u|}\cdot\f{a}{|u|^2}\leq\f{a}{|u|^2}.$$
In summary, for (\ref{chih square 2}) we have
$$\f{\partial}{\partial u}\bigg(u^2\O^2|\chih|^2_{\gamma}\bigg)=M, \mbox{ and } |M|\ls \f{a^{\f32}}{|u|^2}\ll \f{a^{\f74}}{|u|^2},$$
which {\color{black}implies} 
\begin{equation*}
-\f{a^{\f74}}{|u|}+\f{a^{\f74}}{|\ui|}  \leq |u|^2\O^2|\chih|^2_{\gamma}(u,\ub,\theta^1,\theta^2)-|u_{\infty}|^2\O^2|\chih|^2_{\gamma}(u_{\infty},\ub,\theta^1,\theta^2). 
\end{equation*}
Recall $\O(u_{\infty},\ub,\theta^1,\theta^2)=1$. We hence have
\begin{equation*}
|u|^2\O^2|\chih|^2_{\gamma}(u,\ub,\theta^1,\theta^2)\geq|u_{\infty}|^2|\chih|^2_{\gamma}(u_{\infty},\ub,\theta^1,\theta^2)-\f{a^{\f74}}{|u|}.
\end{equation*}
Together with the assumption in Theorem \ref{main.thm2}, we further have for $\ui\leq u\leq -a/4$
\begin{equation*}
\begin{split}
\int_0^1 |u|^2\O^2|\chih|^2_{\gamma}(u,\ub',\theta^1,\theta^2)d\ub'\geq \int_0^1 |u_{\infty}|^2|\chih|^2_{\gamma}(u_{\infty},\ub',\theta^1,\theta^2)d\ub'-\f{a^{\f74}}{|u|}\geq a-\f{a^{\f74}}{|u|}\geq a-\f{4a^{\f74}}{a} \geq \f{7a}{8}.  
\end{split}
\end{equation*}
Pick $u=-a/4$. With the fact $\|\O-1\|_{L^{\infty}(\S)}\ls {1}/{a}$, for sufficiently large $a$, we hence have
\begin{equation*}
\begin{split}
(-\f{a}{4})^2\int_0^1 |\chih|^2_{\gamma}(-\f{a}{4},\ub',\theta^1,\theta^2)d\ub'\geq& \f67\cdot
\int_0^1 (-\f{a}{4})^2\O^2|\chih|^2_{\gamma}(-\f{a}{4},\ub',\theta^1,\theta^2)d\ub'\\
\geq&\f67\cdot \f{7a}{8}=\f{3a}{4}.   
\end{split}
\end{equation*}
This implies
\begin{equation}\label{int chih square}
\begin{split}
\int_0^1 |\chih|^2_{\gamma}(-\f{a}{4},\ub',\theta^1,\theta^2)d\ub'\geq&\f{3a}{4}\cdot \f{16}{a^2}= \f{12}{a}
\end{split}
\end{equation}

Now we consider the other null structure equation
$$\nab_4 \tr\chi+\f12(\tr\chi)^2=-|\chih|^2_{\gamma}-2\o\tr\chi.$$
Using $\o=-\f12\nab_4(\log \Omega)$, we have
\begin{equation*}
\begin{split}
\nab_4 \tr\chi+\f12(\tr\chi)^2=&-|\chih|^2-2\o\tr\chi\\
=&-|\chih|^2_{\gamma}+\nab_4(\log\Omega)\tr\chi=-|\chih|^2_{\gamma}+\f{1}{\Omega}\nab_4\Omega\cdot \tr\chi.
\end{split}
\end{equation*}
Hence,
\begin{equation*}
\begin{split}
\nab_4(\Omega^{-1} \tr\chi)=&-\O^{-2}\nab_4\O\cdot\tr\chi+\O^{-1}\nab_4\tr\chi\\
=&\O^{-1}(\nab_4\tr\chi-\O^{-1}\cdot\nab_4\O\cdot\tr\chi)=\O^{-1}\bigg(-\f12(\tr\chi)^2-|\chih|^2_{\gamma}\bigg).
\end{split}
\end{equation*}
With the fact $e_4=\Omega^{-1}\f{\partial}{\partial \ub}$, we have
\begin{equation}\label{O trchi}
\begin{split}
\f{\partial}{\partial \ub}(\Omega^{-1} \tr\chi)=&-\f12(\tr\chi)^2-|\chih|^2_{\gamma}.
\end{split}
\end{equation}
For every $(\theta^1,\theta^2)\in \mathbb{S}^2$, along $\Hb_0$ we have 
$$(\Omega^{-1}\tr\chi)(-\f{a}{4}, 0, \theta^1, \theta^2)=1^{-1}\cdot \f{2}{a/4}=\f{8}{a}.$$
We then integrate (\ref{O trchi}). Using (\ref{int chih square}) we obtain
\begin{equation*}
\begin{split}
&(\Omega^{-1}\tr\chi)(-\f{a}{4},1, \theta^1, \theta^2)\\
\leq & (\Omega^{-1}\tr\chi)(-\f{a}{4}, 0, \theta^1, \theta^2)-\int_0^{1}|\chih|^2_{\gamma}(-\f{a}{4},\ub',\theta^1,\theta^2)d\ub'\\
\leq & \f{8}{a}-\f{12}{a}<0.
\end{split}
\end{equation*}
Recall in $D_{u,\ub}$ the following estimate holds
\begin{equation*}
\|\tr\chib+\f{2}{|u|}\|_{L^{\infty}(S_{u,\ub})}\leq \f{1}{|u|^2}.
\end{equation*}
In particular, this implies
$$\tr\chib(-\f{a}{4},1, \theta^1, \theta^2)<0 \mbox{ for every} (\theta^1, \theta^2)\in \mathbb{S}^2.$$
Therefore, we conclude that $S_{-\f{a}{4},1}$ is a trapped surface. 

\section{A Scaling Argument and a connection to [An-Luk]}\label{sec rescale} 
In this article, we use coordinate system $(u, \ub, \theta^1, \theta^2)$ based on double null foliations, where $(\theta^1, \theta^2)$ are stereographic coordinates on $\mathbb{S}^2$. In these coordinates, we study spacetime region
$$u_{\infty} \leq u \leq -\f{a}{4}, \, \, \quad 0\leq \ub \leq 1. $$
The Lorentzian metric $g$ satisfies ansatz
$$g=-2\O^2(du\otimes d\ub+d\ub\otimes du)+\gamma_{AB}(d\theta^A-d^A du)\otimes(d\theta^B-d^B du).$$

In the below, by exploring a rescaling, we will find an interesting connection between the results above and the results in \cite{A-L} proved by An-Luk.

\subsection{A Spacetime Rescaling}  
We introduce a new coordinate system $(u', \ub', \theta^{1'}, \theta^{2'})$, where
\begin{equation}\label{rescale}
u'=\d u, \, \, \, \ub'=\d\ub, \, \, \, \theta^{1'}=\d \theta^1, \, \, \,\theta^{2'}=\d \theta^2. 
\end{equation}
Note that coordinates $(\theta^1, \theta^2)$ on $\S$ are set up through stereographic projection. Assume $(x_1, x_2, x_3)$  satisfying $x_1^2+x_2^2+x_3^2=a^2$ and lying on the upper hemisphere of $S_{-a,0}$ (with radius $a$). It then has stereographic projection $(\zeta_1, \zeta_2)=(\f{ax_1}{a+x_3}, \f{ax_2}{a+x_3})$. Scale down the length by a factor $\d$, we then have $x_1'=\d x_1, \, x_2'=\d x_2, \, x_3'=\d x_3, \,\, (x_1')^2+(x_2')^2+(x_2')^2=\d^2 a^2$ and $(x_1', x_2', x_3')$ has stereographic projection 
$$(\zeta_1', \zeta_2')=(\f{\d a x_1'}{\d a+x_3'}, \f{\d ax_2'}{\d a+x_3'})=(\f{\d a\cdot \d x_1}{\d a+\d x_3}, \f{\d a \cdot \d x_2}{\d a+\d x_3})=(\f{\d ax_1}{a+x_3}, \f{\d ax_2}{a+x_3})=(\d\zeta_1, \d\zeta_2).$$ Therefore, the rescaled coordinates $(\theta^{1'}, \theta^{2'})=(\d \theta^1, \d \theta^2)$ on $S_{u', \ub'}$ make perfect sense  since $2$-sphere $S_{u',\ub'}=S_{\d u, \d \ub}$ is scaled down from $\S$ by a factor $\delta$.

\noindent We then rewrite Theorem \ref{main.thm1} and Theorem \ref{main.thm2} in coordinate system $(u', \ub', \theta^{1'}, \theta^{2'})$:

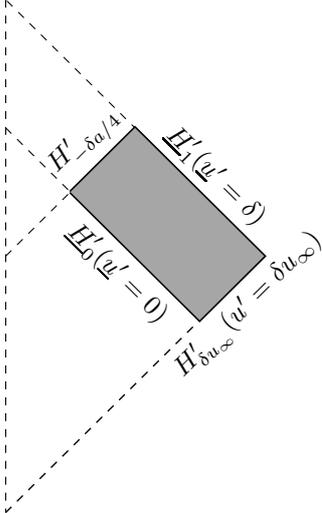
\begin{minipage}[!t]{0.29\textwidth}
\begin{tikzpicture}[scale=0.85]
\draw [white](3,-1)-- node[midway, sloped, below,black]{$H'_{\d u_{\infty}}(u'=\d u_{\infty})$}(4,0);
\draw [white](1,1)-- node[midway,sloped,above,black]{$H'_{-\d a/4}$}(2,2);
\draw [white](2,2)--node [midway,sloped,above,black] {$\Hb'_{1}(\ub'=\d)$}(4,0);
\draw [white](1,1)--node [midway,sloped, below,black] {$\Hb'_{0}(\ub'=0)$}(3,-1);
\draw [dashed] (0, 4)--(0, -4);
\draw [dashed] (0, -4)--(4,0)--(0,4);
\draw [dashed] (0,0)--(2,2);
\draw [dashed] (0,-4)--(4,0);
\draw [dashed] (0,2)--(3,-1);
\draw [very thick] (1,1)--(3,-1)--(4,0)--(2,2)--(1,1);
\fill[black!35!white] (1,1)--(3,-1)--(4,0)--(2,2)--(1,1);
\end{tikzpicture}
\end{minipage}
\hspace{0.01\textwidth} 
\begin{minipage}[!t]{0.58\textwidth}

With an open set of characteristic initial data (corresponding to the initial data in Theorem \ref{main.thm1} and Theorem \ref{main.thm2}), 

\begin{itemize}

\item Einstein vacuum equations (\ref{1.1}) admit a unique smooth solution in the colored region:
$$\d\cdot u_{\infty}\leq u' \leq -{\d\cdot a}/{4}, \quad 0\leq \ub' \leq \d.$$
\item $S'_{-{\d a}/{4}, \,\d}:=\{u'=-\d a/4\}\cap\{\ub'=\d\}$ is a trapped surface. \\

\end{itemize}

The above conclusion is very similar to the main theorem in \cite{A-L}. In the following, we will verify that this conclusion is indeed an extension of \cite{A-L}. In particular, we will show that all the rescaled Ricci coefficients $\Gamma'$ and rescaled curvature components $R'$ would obey the same apriori estimates as in \cite{A-L}.
\end{minipage}   

\noindent Under the rescaling (\ref{rescale}), it follows
$$g'(u',\ub', \theta^{1'}, \theta^{2'})=\d^2\cdot g(u,\ub, \theta^1, \theta^2).$$
In $(u',\ub',\theta^{1'}, \theta^{2'})$ coordinates, we let
$$g'(u', \ub', \theta^1, \theta^2)=-2{\O'}^2(du'\otimes d\ub'+d\ub'\otimes du')+\gamma'_{A'B'}(d\theta^{A'}-d^{A'} du')\otimes(d\theta^{B'}-d^{B' }du').$$
Compare with
$$g(u, \ub, \theta^1, \theta^2)=-2\O^2(du\otimes d\ub+d\ub\otimes du)+\gamma_{AB}(d\theta^A-d^A du)\otimes(d\theta^B-d^B du).$$
Here we have
$$du'=\d\cdot du, \quad d\ub'=\d\cdot d\ub, \quad d\theta^{A'}=\d \cdot d \theta^A \,\mbox{  for } A=1, 2 , $$
$${\O'}^2(u',\ub', \theta^{1'}, \theta^{2'})=\O^2(u,\ub, \theta^1, \theta^2), \quad \gamma'_{A'B'}(u',\ub', \theta^{1'}, \theta^{2'})=\gamma_{AB}(u,\ub, \theta^1, \theta^2),$$ 
$$d^{A'}(u',\ub', \theta^{1'}, \theta^{2'})=d^A(u,\ub, \theta^1, \theta^2),$$
$$e'_3(u',\ub', \theta^{1'}, \theta^{2'})={\O'}^{-1}(\f{\partial}{\partial u'}+d^{A'}\f{\partial}{\partial \ub'})=\d^{-1}{\O}^{-1}(\f{\partial}{\partial u}+d^{A}\f{\partial}{\partial \ub})=\d^{-1}\cdot e_3(u, \ub, \theta^1, \theta^2),$$
\begin{equation}\label{e4 e4'}
e'_4(u',\ub', \theta^{1'}, \theta^{2'})={\O'}^{-1}\f{\partial}{\partial\ub'}=\d^{-1}{\O}^{-1}\f{\partial}{\partial\ub}=\d^{-1}\cdot e_4(u,\ub, \theta^1, \theta^2),
\end{equation}
\begin{equation}\label{eA eA'}
e'_A(u',\ub', \theta^{1'}, \theta^{2'})=\d^{-1}\cdot e_A(u,\ub, \theta^1, \theta^2), \mbox{ for } A=1,2. 
\end{equation}

\subsection{Rescaled Geometric Quantities}
As usual, with frame $\{e'_3, e'_4, e'_A,  e'_B\}$, we define 
\begin{equation*}
\begin{split}
&\chi'_{A'B'}=g'(D'_{A'} e'_4,e'_B),\, \,\, \quad \chib'_{A'B'}=g'(D'_{A'} e'_3,e'_B),\\
&\eta'_{A'}=-\frac 12 g'(D'_{3'} e'_A,e'_4),\quad \etab'_{A'}=-\frac 12 g'(D'_{4'} e'_A,e'_3),\\
&\omega'=-\frac 14 g'(D'_{4'} e'_3,e'_4),\quad\,\,\, \omegab'=-\frac 14 g'(D'_{3'} e'_4,e'_3),\\
&\zeta'_{A'}=\frac 1 2 g'(D'_{A'} e'_4,e'_3).
\end{split}
\end{equation*}
With $\gamma'_{A'B'}$ being the induced metric on $\S'$, we further decompose $\chi', \chib'$ into
$$\chi'_{A'B'}=\f12\tr\chi'\cdot \gamma'_{A'B'}+\chih'_{A'B'}, \quad \chib'_{A'B'}=\f12\tr\chib'\cdot \gamma'_{A'B'}+\chibh'_{A'B'}.$$
$$\mbox{Here} \quad D'_{e'_{\mu}}e'_{\nu}:=\Gamma'^{\lambda}_{\mu'\nu'}e'_{\lambda} \quad \mbox{ and } \quad \Gamma'^{\lambda}_{\mu'\nu'}:=\f12g'^{\lambda'\kappa'}(\f{\partial g'_{\kappa'\mu'}}{\partial x'_{\nu}}+\f{\partial g'_{\kappa' \nu'}}{\partial x'_{\mu}}-\f{\partial g'_{\mu'\nu'}}{\partial x'_{\kappa}}).$$
\begin{remark}
Note by rescaling 
\begin{equation}\label{rescaling 2}
g'=\d^2\cdot g, \mbox{ and } g'^{-1}=\d^{-2}\cdot g^{-1},
\end{equation}
we then have
\begin{equation}\label{Gamma Gamma'}
\Gamma'^{\lambda'}_{\mu'\nu'}:=\f12g'^{\lambda'\kappa'}(\f{\partial g'_{\kappa'\mu'}}{\partial x'_{\nu}}+\f{\partial g'_{\kappa' \nu'}}{\partial x'_{\mu}}-\f{\partial g'_{\mu'\nu'}}{\partial x'_{\kappa}})=\f12g^{\lambda'\kappa'}(\f{\partial g_{\kappa'\mu'}}{\partial x'_{\nu}}+\f{\partial g_{\kappa' \nu'}}{\partial x'_{\mu}}-\f{\partial g_{\mu'\nu'}}{\partial x'_{\kappa}})=\Gamma^{\lambda'}_{\mu'\nu'},
\end{equation}
which implies
\begin{equation}\label{D D'}
D'_{e'_{\mu}}e'_{\nu}=\Gamma'^{\lambda}_{\mu'\nu'}e'_{\lambda}=\Gamma^{\lambda}_{\mu'\nu'}e'_{\lambda}=D_{e'_{\mu}}e'_{\nu}.
\end{equation}
\end{remark}

We are ready to prove 
\begin{proposition}\label{Prop rescale 1}
For $\Gamma\in \{\chih, \tr\chi, \chibh, \tr\chib, \eta, \etb, \zeta, \o,\omb\}$ written in two different coordinates $(u',\ub',\theta^{1'}, \theta^{2'})$ and $(u, \ub, \theta^1, \theta^2)$, it holds that
$$\Gamma'(u',\ub',\theta^{1'}, \theta^{2'})=\d^{-1}\cdot \Gamma (u,\ub,\theta^{1}, \theta^{2}).$$
\end{proposition}
\begin{proof}
We first calculate $\chi'_{A'B'}(u',\ub', \theta^{1'}, \theta^{2'})$. With definition of $\chi'$ in the first step, (\ref{D D'}) in the second step,  (\ref{rescaling 2}) in the third step, (\ref{e4 e4'}) and (\ref{eA eA'}) in the fourth step,  we have
\begin{equation*}
\begin{split}
&\chi'_{A'B'}(u',\ub', \theta^{1'}, \theta^{2'})=g'(D'_{e'_A} e'_4,e'_B)\\
=&g'(D_{e'_A}e'_4, e'_B)=\d^2\cdot g(D_{e'_A}e'_4, e'_B)\\
=&\d^2\cdot \d^{-1}\cdot \d^{-1}\cdot \d^{-1} g(D_{e_A} e_4, e_B)=\d^{-1}\cdot\chi_{AB}(u, \ub, \theta^1, \theta^2).
\end{split}
\end{equation*}
In the same fashion, we conclude 
$$\chib'_{A'B'}(u',\ub', \theta^{1'}, \theta^{2'})=\d^{-1}\cdot\chib_{AB}(u, \ub, \theta^1, \theta^2), \,\, \zeta'_{A'B'}(u',\ub', \theta^{1'}, \theta^{2'})=\d^{-1}\cdot\zeta_{AB}(u, \ub, \theta^1, \theta^2),$$
$$\eta'_{A'}(u',\ub', \theta^{1'}, \theta^{2'})=\d^{-1}\cdot\eta_{A}(u, \ub, \theta^1, \theta^2), \quad \etb'_{A'}(u',\ub', \theta^{1'}, \theta^{2'})=\d^{-1}\cdot\etb_{A}(u, \ub, \theta^1, \theta^2),$$
$$\o'(u',\ub', \theta^{1'}, \theta^{2'})=\d^{-1}\cdot\o(u, \ub, \theta^1, \theta^2), \quad \omb'(u',\ub', \theta^{1'}, \theta^{2'})=\d^{-1}\cdot\omb(u, \ub, \theta^1, \theta^2).$$
We then focus on $\tr\chi'$ and $\chih'_{A'B'}$. As calculated above, we have
\begin{equation*}
\begin{split}
&\tr\chi'(u',\ub', \theta^{1'}, \theta^{2'})\\
=&g'^{A'B'}g'(D'_{A'} e'_4,e'_B)=g'^{A'B'}g'(D_{A'} e'_4,e'_B)=g^{A'B'}g(D_{A'} e'_4,e'_B)\\
=&\d^{-1}g^{AB}g(D_{A} e_4,e_B)=\d^{-1}\cdot\tr\chi(u,\ub,\theta^1, \theta^2), \quad \quad \mbox{and}
\end{split}
\end{equation*}
\begin{equation*}
\begin{split}
&\chih'_{A'B'}(u',\ub', \theta^{1'}, \theta^{2'})\\
=&\chi'_{A'B'}(u',\ub', \theta^{1'}, \theta^{2'})-\f12\tr\chi'(u',\ub', \theta^{1'}, \theta^{2'})\cdot\gamma'_{A'B'}(u',\ub', \theta^{1'}, \theta^{2'})\\
=&\chi'_{A'B'}(u',\ub', \theta^{1'}, \theta^{2'})-\f12\tr\chi'(u',\ub', \theta^{1'}, \theta^{2'})\cdot\d^2\cdot \d^{-1} \cdot \d^{-1}\cdot \gamma_{AB}(u,\ub, \theta^{1}, \theta^{2})\\
=&\d^{-1}\cdot\chi_{AB}(u,\ub, \theta^1, \theta^2)-\f12\cdot\d^{-1}\cdot\tr\chi(u,\ub, \theta^1, \theta^2)\cdot \gamma_{AB}(u,\ub, \theta^1, \theta^2)\\
=&\d^{-1}\cdot \chih_{AB}(u, \ub, \theta^1, \theta^2).
\end{split}
\end{equation*}
Similarly, it also holds
$$\tr\chib'(u',\ub', \theta^{1'}, \theta^{2'})=\d^{-1}\cdot\tr\chib(u,\ub,\theta^1, \theta^2), \quad \chibh'_{A'B'}(u',\ub', \theta^{1'}, \theta^{2'})=\d^{-1}\cdot \chibh(u, \ub, \theta^1, \theta^2).$$
We then conclude this proposition.\\
\end{proof} 

For curvature components, we further have
\begin{proposition}\label{Prop rescale 2}
For $\Psi\in \{\a, \b, \rho, \sigma, \beb, \ab\}$ written in coordinates $(u',\ub',\theta^{1'}, \theta^{2'})$ and $(u, \ub, \theta^1, \theta^2)$, the following identity is true
$$\Psi'(u',\ub',\theta^{1'}, \theta^{2'})=\d^{-2}\cdot\Psi(u,\ub,\theta^{1}, \theta^{2}).$$
\end{proposition}
\begin{proof}
We first write Riemann curvature in $(u',\ub', \theta^{1'}, \theta^{2'})$ coordinate.
$$R'^{l}_{ijk}=\partial_i\Gamma'^{l}_{jk}-\partial_j \Gamma'^{l}_{ik}+\Gamma'^{p}_{jk}\Gamma'^l_{ip}-\Gamma'^{p}_{ik}\Gamma'^{l}_{ip}$$
With the help of (\ref{Gamma Gamma'}), we obtain
\begin{equation*}
\begin{split}
R'^{l}_{ijk}=&\partial_i\Gamma'^{l}_{jk}-\partial_j \Gamma'^{l}_{ik}+\Gamma'^{p}_{jk}\Gamma'^l_{ip}-\Gamma'^{p}_{ik}\Gamma'^{l}_{ip}\\
=&\partial_i\Gamma^{l}_{jk}-\partial_j \Gamma^{l}_{ik}+\Gamma^{p}_{jk}\Gamma^l_{ip}-\Gamma^{p}_{ik}\Gamma'^{l}_{ip}=R^{l}_{ijk}.
\end{split}
\end{equation*}
This {\color{black}implies} 
$${R'}_{ijkl}={R'}^m_{ijk} g'_{ml}=R^m_{ijk}\cdot\d^2\cdot g_{ml}=\d^2\cdot R_{ijkl}.$$
Therefore, we obtain
\begin{equation*}
\begin{split}
a'_{A'B'}(u', \ub', \theta^{1'}, \theta^{2'}):=&R'(e'_A, e'_4, e'_B, e'_4)=\d^2\cdot R(e'_A, e'_4, e'_B, e'_4)\\
=&\d^{2}\cdot\d^{-4}\cdot R(e_A, e_4, e_B, e_4)=\d^{-2}\cdot\a_{AB}(u,\ub,\theta^1, \theta^2).
\end{split}
\end{equation*}
In the same manner, we have
$$\b'_{A'}(u', \ub', \theta^{1'}, \theta^{2'})=\d^{-2}\cdot\b_{A}(u,\ub,\theta^1, \theta^2), \quad \rho'(u', \ub', \theta^{1'}, \theta^{2'})=\d^{-2}\cdot\rho(u,\ub,\theta^1, \theta^2), $$
$$\sigma'_{A'}(u', \ub', \theta^{1'}, \theta^{2'})=\d^{-2}\cdot\sigma(u,\ub,\theta^1, \theta^2), \quad \beb'_{A'}(u', \ub', \theta^{1'}, \theta^{2'})=\d^{-2}\cdot\beb_{A}(u,\ub,\theta^1, \theta^2), $$
$$\ab'_{A'B'}(u', \ub', \theta^{1'}, \theta^{2'})=\d^{-2}\cdot\ab_{AB}(u,\ub,\theta^1, \theta^2).$$
These conclude the proposition. 
\end{proof}

\subsection{Rescaled Uniform Bounds}\label{rescale bounds}

Applying Proposition \ref{Prop rescale 1} and Proposition \ref{Prop rescale 2}, next we establish the connection to \cite{A-L}.  Take $\chih$ as an example. Applying Proposition \ref{Prop rescale 1}, estimates derived for $\mathcal{O}_{i,\infty}[\chih]$ and $u'=\d u$, we have 
\begin{equation*}
\begin{split}
&|\chih'_{A'B'}(u',\ub', \theta^{1'}, \theta^{2'})|=\d^{-1}\cdot |\chih_{AB}(u, \ub, \theta^1, \theta^2)|\leq \d^{-1}\cdot \f{\at}{|u|}=\f{\at}{\d|u|}=\f{\at}{|u'|}. 
\end{split}
\end{equation*}
In the same fashion, we have
\begin{equation*}
|\chibh'_{A'B'}(u',\ub', \theta^{1'}, \theta^{2'})|=\d^{-1}\cdot |\chibh_{AB}(u, \ub, \theta^1, \theta^2)|\leq \d^{-1}\cdot \f{\at}{|u|^2}=\f{\d\at}{\d^2|u|^2}=\f{\d\at}{|u'|^2}, 
\end{equation*}
\begin{equation*}
|\tr\chi'(u',\ub', \theta^{1'}, \theta^{2'})|=\d^{-1}\cdot |\tr\chi(u, \ub, \theta^1, \theta^2)|\leq \d^{-1}\cdot \f{1}{|u|}=\f{1}{\d|u|}=\f{1}{|u'|}, 
\end{equation*}
\begin{equation*}
|\eta'_{A'}(u',\ub', \theta^{1'}, \theta^{2'})|=\d^{-1}\cdot |\eta_{A}(u, \ub, \theta^1, \theta^2)|\leq \d^{-1}\cdot \f{\at}{|u|^2}=\f{\d\at}{\d^2|u|^2}=\f{\d\at}{|u'|^2}, 
\end{equation*}
\begin{equation*}
|\etb'_{A'}(u',\ub', \theta^{1'}, \theta^{2'})|=\d^{-1}\cdot |\etb_{A}(u, \ub, \theta^1, \theta^2)|\leq \d^{-1}\cdot \f{\at}{|u|^2}=\f{\d\at}{\d^2|u|^2}=\f{\d\at}{|u'|^2}, 
\end{equation*}
\begin{equation*}
|\o'(u',\ub', \theta^{1'}, \theta^{2'})|=\d^{-1}\cdot |\o(u, \ub, \theta^1, \theta^2)|\leq \d^{-1}\cdot \f{1}{|u|}=\f{1}{\d|u|}=\f{1}{|u'|}, 
\end{equation*}
\begin{equation*}
|\omb'(u',\ub', \theta^{1'}, \theta^{2'})|=\d^{-1}\cdot |\omb(u, \ub, \theta^1, \theta^2)|\leq \d^{-1}\cdot \f{a}{|u|^3}=\f{\d^2 a}{\d^3|u|^3}=\f{\d^2 a}{|u'|^3}\boxed{\leq \f{\d\at}{|u'|^2}} \, ,
\end{equation*}
\begin{equation*}
|\tr\chib'(u',\ub', \theta^{1'}, \theta^{2'})+\f{2}{|u'|}|=\d^{-1}\cdot |\tr\chib(u, \ub, \theta^1, \theta^2)+\f{2}{|u|}|\leq \d^{-1}\cdot \f{a}{|u|^3}=\f{\d^2 a}{\d^3|u|^3}=\f{\d^2 a}{|u'|^3}\boxed{\leq \f{\d\at}{|u'|^2}} \, .
\end{equation*}
For the estimates of $\omb'$ and $\tr\chib'$, we use $|u'|\geq \d\at$. In the same manner, by Proposition \ref{Prop rescale 2} and with the help that $|u'|\geq  \d a/4$ we have
\begin{equation*}
\begin{split}
&|\b'_{A'}(u',\ub', \theta^{1'}, \theta^{2'})|=\d^{-2}\cdot |\b_{A}(u, \ub, \theta^1, \theta^2)|\leq \d^{-2}\cdot \f{\at}{|u|^2}=\f{\at}{\d^2|u|^2}=\f{\at}{|u'|^2}, 
\end{split}
\end{equation*}
\begin{equation*}
\begin{split}
&|\rho'(u',\ub', \theta^{1'}, \theta^{2'})|=\d^{-2}\cdot |\rho(u, \ub, \theta^1, \theta^2)|\leq \d^{-2}\cdot \f{a}{|u|^3}=\f{\d a}{\d^3|u|^3}=\f{\d a}{|u'|^3}, 
\end{split}
\end{equation*}
\begin{equation*}
\begin{split}
&|\sigma'(u',\ub', \theta^{1'}, \theta^{2'})|=\d^{-2}\cdot |\sigma(u, \ub, \theta^1, \theta^2)|\leq \d^{-2}\cdot \f{a}{|u|^3}=\f{\d a}{\d^3|u|^3}=\f{\d a}{|u'|^3}, 
\end{split}
\end{equation*}
\begin{equation*}
\begin{split}
&|\beb'_{A'}(u',\ub', \theta^{1'}, \theta^{2'})|=\d^{-2}\cdot |\beb_{A}(u, \ub, \theta^1, \theta^2)|\leq \d^{-2}\cdot \f{a^{\f32}}{|u|^4}=\f{\d^2 a^{\f32}}{\d^4|u|^4}=\f{\d^2 a^{\f32}}{|u'|^4}\boxed{\leq \f{\d\at}{|u'|^3}} \, , 
\end{split}
\end{equation*}
\begin{equation}\label{rescale ab}
\begin{split}
&|\ab'_{A'B'}(u',\ub', \theta^{1'}, \theta^{2'})|=\d^{-2}\cdot |\ab_{AB}(u, \ub, \theta^1, \theta^2)|\leq \d^{-2}\cdot \f{a^2}{|u|^5}=\f{\d^3 a^2}{\d^5|u|^5}=\f{\d^3 a^2}{|u'|^5}, 
\end{split}
\end{equation}
\begin{equation}\label{rescale alpha}
\begin{split}
&|\a'_{A'B'}(u',\ub', \theta^{1'}, \theta^{2'})|=\d^{-2}\cdot |\a_{AB}(u, \ub, \theta^1, \theta^2)|\leq \d^{-2}\cdot \f{\at}{|u|}=\f{\d^{-1}\at}{\d |u|}=\f{\d^{-1}\at}{|u'|}. 
\end{split}
\end{equation}

\begin{remark}\label{systematical improvement}
By above rescaling argument, we hence transfer the bounds derived in preceding sections into new bounds, holding in the spacetime region
$$\d a\leq |u'| \leq \d|u_{\infty}|, \quad 0\leq \ub' \leq \d.$$
If we focus on the region
$$\d a\leq |u'| \leq 1, \quad 0\leq \ub' \leq \d,$$
\textit{these bounds encode peeling properties, and they systematically sharpen the a priori bounds in \cite{A-L}}:
\begin{itemize}

\item For $\{\omb',\, \tr\chib'+\f{2}{|u'|},\, \beb'\}$, we improve their estimates in \cite{A-L}. For comparison, the terms boxed are the estimates obtained in \cite{A-L}.  

\item For $\{\ab', \a'\}$, their estimates are avoid in \cite{A-L} by several geometric renormalizations. \textit{But for here we have estimates for $\ab'$ and $\a$, and they respect peeling properties}.  
\end{itemize}

\noindent In summary, in \cite{A-L} we have $1\ll b\leq \at \leq \d^{-\f12}$. If we focus on the scenario $1\ll b=\at \leq \d^{-\f12}$, compared with \cite{A-L} the new approach in this paper completely avoids geometric renormalizations and still gives a systematical improvement encoding intrinsic\footnote{From conformal compactification.} peeling properties. 

\end{remark}

\begin{remark}
Since the estimates obtained above are uniform for $|u_{\infty}|$, we could keep $\d$ and let $|u_{\infty}|\rightarrow +\infty$. This extends \cite{A-L} and establish an existence result from pass null infinity.
\end{remark}

\noindent By repeating the arguments as in Section \ref{TSF}, we hence obtain one of the main conclusions in \cite{A-L} by An-Luk {\color{black}on formation of a small trapped surface}:

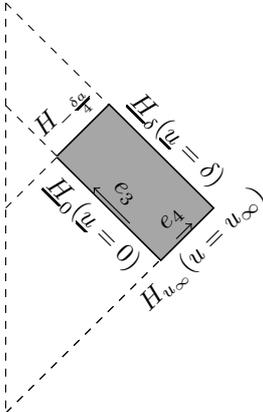
\begin{minipage}[!t]{0.27\textwidth}
\begin{tikzpicture}[scale=0.68]
\draw [white](3,-1)-- node[midway, sloped, below,black]{$H_{u_{\infty}}(u=u_{\infty})$}(4,0);
\draw [white](1,1)-- node[midway,sloped,above,black]{$H_{-\f{\d a}{4}}$}(2,2);
\draw [white](2,2)--node [midway,sloped,above,black] {$\Hb_{\d}(\ub=\d)$}(4,0);
\draw [white](1,1)--node [midway,sloped, below,black] {$\Hb_{0}(\ub=0)$}(3,-1);
\draw [dashed] (0, 4)--(0, -4);
\draw [dashed] (0, -4)--(4,0)--(0,4);
\draw [dashed] (0,0)--(2,2);
\draw [dashed] (0,-4)--(4,0);
\draw [dashed] (0,2)--(3,-1);
\draw [very thick] (1,1)--(3,-1)--(4,0)--(2,2)--(1,1);
\fill[black!35!white] (1,1)--(3,-1)--(4,0)--(2,2)--(1,1);
\draw [->] (3.3,-0.6)-- node[midway, sloped, above,black]{$e_4$}(3.6,-0.3);
\draw [->] (2.4,-0.3)-- node[midway, sloped, above,black]{$e_3$}(1.7,0.4);
\end{tikzpicture}
\end{minipage}
\hspace{0.01\textwidth} 
\begin{minipage}[!t]{0.65\textwidth}
{\bf Theorem \ref{thm3}}
We solve Einstein vacuum equations. Given $\M I^{(0)}$, for a fixed $\d$ there exists a sufficiently large $a_0=a_0(\M I^{(0)},\d)$.  For $0<a_0<a$, \, with initial data:
\begin{itemize}
\item $\sum_{i\leq 10, k\leq 3}a^{-\frac12}\|(\d\nab_4)^k(|u_{\infty}|\nab)^{i}\chih_0\|_{L^{\infty}(S_{u_{\infty},\ub})}\leq \M I^{(0)}$ 

\noindent along $u=u_{\infty}$,
  
\item Minkowskian initial data  along $\ub=0$,

\item $\int_0^{\delta}|u_{\infty}|^2|\chih_0|^2(u_{\infty}, \ub')d\ub'\geq \d a$ for every direction

\noindent  along $u=u_{\infty}$,
\end{itemize}
we have that $S_{-\d a/4,\d}$ is a trapped surface.
\end{minipage}   

If we further choose $a=4c\cdot\d^{-1}$, where $0< c\leq 1$ being of size $1$, we then obtain Corollary \ref{Corollary1.5}.

\section{Rescale from \cite{A-L} and Comparison}\label{rescale and comparison}
Since the results in \cite{A-L} by An and Luk are scaling-critical, first choosing $\d=1/|u_{\infty}|$ and then scaling up the length by a factor $|u_{\infty}|$, we could change the main conclusion in \cite{A-L} into a conclusion similar to Theorem \ref{main.thm2}. The full strength of the proof in \cite{A-L} ensures all the constants in the inequalities of \cite{A-L} are independent of $\d$; hence the constants in the new result after scaling up would be independent of $|u_{\infty}|$. This is similar to the proof of Theorem  \ref{main.thm2}. In the below, we will demonstrate this approach and make the comparsion. 

Fix $|u_{\infty|}$ to be a large positive constant. And set $\d=1/|u_{\infty}|$. By applying the main conclusion in \cite{A-L}, we have

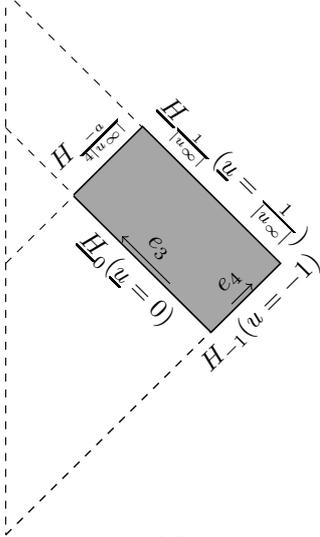
\begin{minipage}[!t]{0.27\textwidth}
\begin{tikzpicture}[scale=0.90]
\draw [white](3,-1)-- node[midway, sloped, below,black]{$H_{-1}(u=-1)$}(4,0);
\draw [white](1,1)-- node[midway,sloped,above,black]{$H_{\f{-a}{4|u_{\infty}|}}$}(2,2);
\draw [white](2,2)--node [midway,sloped,above,black] {$\Hb_{\f{1}{|u_{\infty}|}}(\ub=\f{1}{|u_{\infty}|})$}(4,0);
\draw [white](1,1)--node [midway,sloped, below,black] {$\Hb_{0}(\ub=0)$}(3,-1);
\draw [dashed] (0, 4)--(0, -4);
\draw [dashed] (0, -4)--(4,0)--(0,4);
\draw [dashed] (0,0)--(2,2);
\draw [dashed] (0,-4)--(4,0);
\draw [dashed] (0,2)--(3,-1);
\draw [very thick] (1,1)--(3,-1)--(4,0)--(2,2)--(1,1);
\fill[black!35!white] (1,1)--(3,-1)--(4,0)--(2,2)--(1,1);
\draw [->] (3.3,-0.6)-- node[midway, sloped, above,black]{$e_4$}(3.6,-0.3);
\draw [->] (2.4,-0.3)-- node[midway, sloped, above,black]{$e_3$}(1.7,0.4);
\end{tikzpicture}
\end{minipage}
\hspace{0.01\textwidth} 
\begin{minipage}[!t]{0.65\textwidth}
\begin{proposition}\label{Prop 9.1}
We solve Einstein vacuum equations. Given $\M I^{(0)}>0$, for a fixed $1/|u_{\infty}|$ there exists a sufficiently large $a_0=a_0(\M I^{(0)}, 1/|u_{\infty}|)$.  For $0<a_0<a$, \, with initial data: 
\begin{itemize}
\item $\sum_{i\leq 10, k\leq 3}a^{-\frac12}\|(\f{1}{|u_{\infty}|}\nab_4)^k\nab^{i}\chih_0\|_{L^{\infty}(S_{-1,\ub})}\leq \M I^{(0)}$ 

\noindent along $u=-1$,
  
\item Minkowskian initial data  along $\ub=0$,

\item $\int_0^{\f{1}{|u_{\infty}|}}|\chih_0|^2(-1, \ub')d\ub'\geq \f{a}{|u_{\infty}|}$ for every direction

\noindent  along $u=-1$,
\end{itemize}
we have that $S_{\f{-a}{4|u_{\infty}|},\f{1}{|u_{\infty}|}}$ is a trapped surface.\\
\end{proposition}

Note: In \cite{A-L}, in the proof of Proposition \ref{Prop 9.1} we only use the largeness of $a$ and hence all the constants in the inequalities are independent of $|u_{\infty}|$.

\end{minipage}   

By applying \cite{A-L} we also derive the following bounds
\begin{equation}\label{9.1}
\begin{split}
&|\chih_{AB}(u,\ub,\theta^1, \theta^2)|\leq \at/|u|,\quad |\o(u,\ub,\theta^1, \theta^2)|\leq \at/|u|, \quad |\tr\chi(u,\ub,\theta^1, \theta^2)|\leq 1/|u|,\\
&|\chibh_{AB}(u,\ub,\theta^1, \theta^2)|\leq \at/|u_{\infty}u^2|, \quad |\eta_A(u,\ub,\theta^1, \theta^2)|\leq \at/|u_{\infty}u^2|, \quad |\etb_A(u,\ub,\theta^1, \theta^2)|\leq \at/|u_{\infty}u^2|,\\
&|\omb(u,\ub,\theta^1, \theta^2)|\leq \at/|u_{\infty}u^2|, \quad |\tr\chib(u,\ub,\theta^1, \theta^2)+\f{2}{|u|}|\leq \at/|u_{\infty}u^2|,\\
&|\b_A(u,\ub,\theta^1, \theta^2)|\leq \at/|u^2|, \quad |\rho(u,\ub,\theta^1, \theta^2)|\leq a/|u_{\infty}u^3|,\\
&|\sigma(u,\ub,\theta^1, \theta^2)|\leq a/|u_{\infty}u^3|, \quad |\beb_A(u,\ub,\theta^1, \theta^2)|\leq \at/|u_{\infty}u^3|.
\end{split}
\end{equation}

We then rescale the length from small scale to large scale: 
$$\mbox{Set} \, \,\, u'=|u_{\infty}|u, \,\, \ub'=|u_{\infty}|\ub, \,\, {\theta^1}'=|u_{\infty}|\theta^1, \,\, {\theta^2}'=|u_{\infty}|\theta^2.$$

We have the rescaled result:

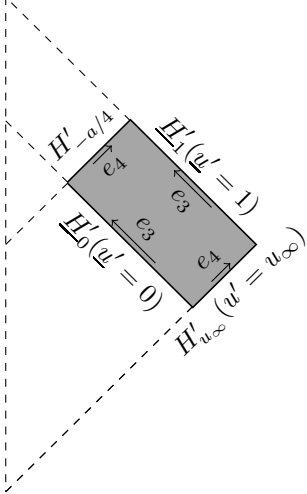
\begin{minipage}[!t]{0.27\textwidth}
\begin{tikzpicture}[scale=0.82]
\draw [white](3,-1)-- node[midway, sloped, below,black]{$H'_{u_{\infty}}(u'=u_{\infty})$}(4,0);
\draw [white](1,1)-- node[midway,sloped,above,black]{$H'_{-a/4}$}(2,2);
\draw [white](2,2)--node [midway,sloped,above,black] {$\Hb'_{1}(\ub'=1)$}(4,0);
\draw [white](1,1)--node [midway,sloped, below,black] {$\Hb'_{0}(\ub'=0)$}(3,-1);
\draw [dashed] (0, 4)--(0, -4);
\draw [dashed] (0, -4)--(4,0)--(0,4);
\draw [dashed] (0,0)--(2,2);
\draw [dashed] (0,-4)--(4,0);
\draw [dashed] (0,2)--(3,-1);
\draw [very thick] (1,1)--(3,-1)--(4,0)--(2,2)--(1,1);
\fill[black!35!white] (1,1)--(3,-1)--(4,0)--(2,2)--(1,1);
\draw [->] (3.3,-0.6)-- node[midway, sloped, above,black]{$e_4$}(3.6,-0.3);
\draw [->] (1.4,1.3)-- node[midway, sloped, below,black]{$e_4$}(1.7,1.6);
\draw [->] (3.3,0.6)-- node[midway, sloped, below,black]{$e_3$}(2.7,1.2);
\draw [->] (2.4,-0.3)-- node[midway, sloped, above,black]{$e_3$}(1.7,0.4);
\end{tikzpicture}
\end{minipage}
\hspace{0.01\textwidth} 
\begin{minipage}[!t]{0.63\textwidth}

\begin{proposition}\label{Prop 9.2}

Given $\M I^{(0)}$, there exists a sufficiently large $a_0=a_0(\M I^{(0)})$.  For $0<a_0<a$, for Einstein vacuum equations with initial data:
\begin{itemize}
\item $\sum_{i\leq 10, k\leq 3}a^{-\frac12}\|\nab^{k}_4(|u_{\infty}|\nab')^{i}\chih_0\|_{L^{\infty}(S_{u_{\infty},\ub'})}\leq \M I^{(0)}$ 

\noindent along $u'=u_{\infty}$,

\item Minkowskian initial data  along $\ub'=0$,

\item $\int_0^1|u_{\infty}|^2|\chih_0|^2(u_{\infty}, \ub'')d\ub''\geq a$ for every direction   

\noindent along $u'=u_{\infty}$,
\end{itemize}
we have that $S'_{-a/4,1}$ is a trapped surface.
\end{proposition}

Note: The statements in Proposition \ref{Prop 9.2} are the same as in Theorem  \ref{main.thm2}. In the below, we will explain that the derived bounds by these two approaches would be slightly different.  

\end{minipage}   

Proceed the same as in Section \ref{sec rescale}, via \ref{9.1}, for $\psi'\in\{\chih'_{A'B'}, \o'\}$ we have
\begin{equation*}
\begin{split}
&|\p'(u',\ub', \theta^{1'}, \theta^{2'})|=|u_{\infty}|^{-1}\cdot |\p(u, \ub, \theta^1, \theta^2)|\leq |u_{\infty}|^{-1}\cdot \f{\at}{|u|}=\f{\at}{|u_{\infty}u|}=\f{\at}{|u'|}.
\end{split}
\end{equation*}
Similarly, for $\q\in\{\chibh'_{A'B'}, \eta'_{A'}, \etb'_{A'}, \omb', \tr\chib'+\f{2}{|u'|}\}$, we have
\begin{equation*}
|\q'(u',\ub', \theta^{1'}, \theta^{2'})|=|u_{\infty}|^{-1}\cdot |\q(u, \ub, \theta^1, \theta^2)|\leq |u_{\infty}|^{-1}\cdot \f{\at}{|u_{\infty}u^2|}=\f{\at}{|u_{\infty}|^2|u|^2}=\f{\at}{|u'|^2}.
\end{equation*}
And
\begin{equation*}
|\tr\chi'(u',\ub', \theta^{1'}, \theta^{2'})|=|u_{\infty}|^{-1}\cdot |\tr\chi(u, \ub, \theta^1, \theta^2)|\leq |u_{\infty}|^{-1}\cdot \f{1}{|u|}=\f{1}{|u_{\infty} u|}=\f{1}{|u'|}.
\end{equation*}
For curvature components, with the same method as in Section \ref{sec rescale}, via \ref{9.1} we obtain
\begin{equation*}
\begin{split}
&|\b'_{A'}(u',\ub', \theta^{1'}, \theta^{2'})|=|u_{\infty}|^{-2}\cdot |\b_{A}(u, \ub, \theta^1, \theta^2)|\leq |u_{\infty}|^{-2}\cdot \f{\at}{|u|^2}=\f{\at}{|u_{\infty}u|^2}=\f{\at}{|u'|^2}, 
\end{split}
\end{equation*}
\begin{equation*}
\begin{split}
&|\rho'(u',\ub', \theta^{1'}, \theta^{2'})|=|u_{\infty}|^{-2}\cdot |\rho(u, \ub, \theta^1, \theta^2)|\leq |u_{\infty}|^{-2}\cdot \f{a}{|u_{\infty}u^3|}=\f{a}{|u_{\infty}u|^3}=\f{a}{|u'|^3}, 
\end{split}
\end{equation*}
\begin{equation*}
\begin{split}
&|\sigma'(u',\ub', \theta^{1'}, \theta^{2'})|=|u_{\infty}|^{-2}\cdot |\sigma(u, \ub, \theta^1, \theta^2)|\leq |u_{\infty}|^{-2}\cdot \f{a}{|u_{\infty}u^3|}=\f{a}{|u_{\infty}u|^3}=\f{a}{|u'|^3}, 
\end{split}
\end{equation*}
\begin{equation*}
\begin{split}
&|\beb'_{A'}(u',\ub', \theta^{1'}, \theta^{2'})|=|u_{\infty}|^{-2}\cdot |\beb_{A}(u, \ub, \theta^1, \theta^2)|\leq |u_{\infty}|^{-2}\cdot \f{\at}{|u_{\infty}u^3|}=\f{\at}{|u_{\infty}u|^3}=\f{\at}{|u'|^3}.
\end{split}
\end{equation*}
Note that these estimates are slight different from the estimates we obtained in previous sections. We don't have the proved peeling property and the $a$-weights are quite different.



\begin{thebibliography}{99} 
\bibitem[1]{An:Trapped} X. An, \textit{Formation of Trapped Surfaces from Past Null Infinity}, preprint (2012), arXiv:1207.5271. 

\bibitem[2]{An: AH} X. An, \textit{Emergence of Apparent Horizon in Gravitational Collapse}, preprint (2017), arXiv:1703.00118. 

\bibitem[3]{An:TSA} X. An, \textit{A scale-critical trapped surface formation criterion II: Applications}, in preparation. 

\bibitem[4]{A-A} X. An, N. Athanasiou, \textit{A scale-critical trapped surface formation criterion for the Einstein-Maxwell equations}, in preparation. 

\bibitem[5]{A-L} X. An, J. Luk, \textit{Trapped surfaces in vacuum arising dynamically from mild incoming radiation}, Advances in Theoretical and Mathematical Physics 21 (2017), 1-120.



\bibitem[6]{Chr.1} D. Christodoulou, \textit{The formation of black holes and singularities in spherically symmetric gravitational collapse}, Comm. Pure Appl. Math. 44 (1991), no. 3, 339--373.

\bibitem[7] {Chr.BV} D. Christodoulou, \textit{Bounded variation solutions of the spherically symmetric Einstein-scalar field equations}, Comm. Pure Appl. Math. 46 (1993), no. 8, 1131-1220.

\bibitem[8]{Chr.2} D. Christodoulou, \textit{Examples of naked singularity formation in the gravitational collapse of a scalar field}, Ann. of Math. (2) 140 (1994), no. 3, 607--653.

\bibitem[9]{Chr.3} D. Christodoulou, \textit{The instability of naked singularities in the gravitational collapse of a scalar field}, Ann. of Math. (2) 149 (1999), no. 1, 183--217.

\bibitem[10]{Chr:book} D. Christodoulou, \textit{The Formation of Black Holes in General Relativity}, Monographs in Mathematics, European Mathematical Soc. (2009). 

\bibitem[11]{Chr-Kl} D. Christodoulou, S. Klainerman,\textit{The global nonlinear stability of the Minkowski space}, Princeton mathematical series 41, (1993).

\bibitem[12]{Dafermos} M. Dafermos, \textit{The formation of black holes in general relativity}, Astrisque 352 (2013).

\bibitem[13]{D-H-R} M. Dafermos, G. Holzegel, I. Rodnianski, \textit{A scattering theory construction of dynamical vacuum black holes}, preprint (2013), arXiv: 1306.5364.

\bibitem[14]{Hol} G. Holzegel, 
\textit{Ultimately Schwarzschildean spacetimes and the black hole stability problem}, preprint (2010), arXiv:1010.3216.

\bibitem[15]{KNI:book} S. Klainerman,   F. Nicolo, \textit{The evolution problem in General Relativity}, Progress in Mathematical Physics, Birkha\"user (2003).

\bibitem[16]{KN:peeling} S. Klainerman,   F. Nicolo, \textit{Peeling properties of asymptotically flat solutions to the Einstein vacuum equations}, Class. Quantum Grav. 20 (2003), 3215-3257.


\bibitem[17]{K-L-R} S. Klainerman, J. Luk, I. Rodnianski, \textit{A fully anisotropic mechanism for formation of trapped surfaces in vacuum}, Invent. Math. 198 (2014), no.1, 1-26.

\bibitem[18]{KR:causal} S. Klainerman, I. Rodnianski, 
\textit{Causal geometry of Einstein-Vacuum 
spacetimes with finite curvature flux}, Invent. Math. 159 (2005), 437-529.  


\bibitem[19]{KR:LP} S. Klainerman,  I. Rodnianski,
\textit{A geometric approach to the Littlewood-Paley theory},
Geom. Funct. Anal. 16, (2006) no. 1, 126-163. 

\bibitem[20]{KR:Scarred} S. Klainerman, I. Rodnianski, 
\textit{On emerging scarred surfaces for the Einstein vacuum equations}, Discrete Contin. Dyn. Syst. , 28 (2010), no. 3, 1007-1031.

\bibitem[21]{KR:Trapped} S. Klainerman, I. Rodnianski, 
\textit{On the formation of trapped surfaces}, Acta Math. 208 (2012), no.2, 211-333.

\bibitem[22]{Le} P. Le, \textit{The intersection of a hyperplane with a lightcone in the Minkowski spacetime}, J. Differential Geom., 109, (2018), no.3, 497-507. 


\bibitem[23]{L-L} J. Li, J. Liu, \textit{Instability of spherical naked singularities of a scalar field under gravitational perturbations}, preprint, (2017), arXiv: 1710.02422.

\bibitem[24]{L-Y} J. Li, P. Yu, \textit{Construction of Cauchy data of vacuum Einstein field equations evolving to black holes}, Ann. of Math. 181 (2015), no.2, 699-768.


\bibitem[25]{Luk} J. Luk,
\textit{On the local existence for the characteristic initial value problem in general relativity}, Int. Mat. Res. Notices 20, (2012), 4625-4678.

\bibitem[26]{Luk2} J. Luk, \textit{Weak null singularity in general relativity}, Journal of AMS. 31 (2018) 1-63.

\bibitem[27]{L-R:Propagation} J. Luk, I. Rodnianski, 
\textit{Local propagation of impulsive gravitational waves}, Comm. Pure Appl. Math., 68 (2015), no.4, 511-624.

\bibitem[28]{L-R} J. Luk, I. Rodnianski, 
\textit{Nonlinear interactions of impulsive gravitational waves for the vacuum Einstein equations}, Cambridge Journal of Math. 5(4): 435-570, 2017.

\bibitem[29]{Penrose} R. Penrose, 
\textit{Gravitational collapse and space-time singularities}, Phys. Rev. Lett. 14 (1965), 57-59.


\bibitem[30]{R-T} M. Reiterer, E. Trubowitz, \textit{Strongly focused gravitational waves}, Comm. Math. Phys. 307 (2011), no. 2, 275-313.

\bibitem[31]{R-S} I. Rodnianski, Y. Shlapentokh-Rothman, \textit{The asymptotically self-similar regime for the Einstein vacuum equations}, Geom. Funct. Anal. Vol. 28 (2018) 755–878.

\bibitem[32]{Yu1} P. Yu, \textit{Energy estimates and gravitational collapse}, Comm. Math. Phys. 317 (2013), no. 2, 275-316.

\bibitem[33]{Yu2} P. Yu, \textit{Dynamical formation of black holes due to the condensation of matter field}, preprint (2011), arXiv: 1105.5898.  

\end{thebibliography}
\end{document}